\documentclass[12pt,letter]{amsart}

\usepackage{amssymb}
\usepackage{graphicx}
\usepackage{array}
\usepackage{mathtools}
\usepackage{thmtools}

\DeclareMathAlphabet{\mathpzc}{OT1}{pzc}{m}{it}

\usepackage{graphicx}
\usepackage[small]{caption}
\usepackage{datetime,verbatim,array}
\usepackage{amsthm, amsmath, amsfonts, amssymb}
\usepackage{mathrsfs, rotating}
\usepackage[usenames]{color}
\usepackage[margin=1in]{geometry}
\usepackage[latin2]{inputenc}
\usepackage{t1enc}
\usepackage{epstopdf} 
\usepackage[colorlinks, citecolor=blue, filecolor=black, linkcolor=red, urlcolor=blue]{hyperref}

\usepackage{comment}
\usepackage[all,cmtip]{xy}

\usepackage{pinlabel}
\usepackage[dvipsnames]{xcolor}

\usepackage{booktabs}

\usepackage{subcaption}

\newcommand{\bZ}{\mathbb{Z}}

\newcommand{\bN}{\mathbb{N}}

\newcommand{\bR}{\mathbb{R}}

\newcommand{\cC}{\mathscr{C}}

\newcommand{\cG}{\mathscr{G}}

\newcommand\into{\hookrightarrow}

\newcommand{\red}{\color{red}}
\definecolor{softgreen}{HTML}{00D400}
\newcommand{\green}{\color{softgreen}}
\definecolor{maroon}{HTML}{800000}

\definecolor{salmon}{HTML}{FF8080}

\definecolor{softblue}{HTML}{0066FF}
\newcommand{\blue}{\color{softblue}}
\definecolor{darkblue}{HTML}{0044AA}

\definecolor{darkgreen}{HTML}{008000}
\newcommand{\dgreen}{\color{darkgreen}}

\newcommand{\p}{\boldsymbol{\pi}}

\newcommand{\e}{\boldsymbol{\varepsilon}}

\newcommand{\K}{K}

\newcommand{\bK}{\mathbf{K}}
\newcommand{\bX}{\mathbf{X}}

\newcommand{\D}{\Delta}

\newcommand{\B}{\mathcal{B}}
\newcommand{\LL}{\mathscr{L}}

\newcommand{\abs}[1]{\left|#1\right|}

\newcommand{\td}{%
\mathpzc{D}}

\numberwithin{equation}{section}

\usepackage{environ}
\usepackage{etoolbox}
\usepackage{xparse}
\NewEnviron{requation}[1]{%
  \csxdef{reqno@#1}{\theequation}
  \protected@csxdef{req@#1}{\BODY}
  \equation\BODY\label{#1}\endequation
}
\newcounter{rememberedequation}
\NewDocumentCommand\Eqref{ s m }{%
  \relax\ifmmode\csuse{req@#2}%
  \else%
    \IfBooleanTF{#1}{\begin{equation*}\csuse{req@#2}\end{equation*}}
    {
       \setcounter{rememberedequation}{\value{equation}}
       \setcounter{equation}{\csuse{reqno@#2}}
       \begin{equation}\csuse{req@#2}\end{equation}
       \setcounter{equation}{\value{rememberedequation}}
    }%
  \fi%
}

\usepackage[dvipsnames]{xcolor}
\usepackage{hyperref}
\definecolor{candyapplered}{rgb}{1.0, 0.03, 0.0}
\definecolor{mediumblue}{rgb}{0.0, 0.0, 0.8}
\hypersetup{
pdfauthor={William W. Menasco and Margaret Nichols},
pdftitle={Surface embeddings in $\bR^2 \times \bR$},
bookmarksnumbered,
colorlinks=true,
linkcolor=mediumblue,
citecolor=candyapplered,
urlcolor=blue}

\usepackage{caption,subcaption}
\usepackage[ocgcolorlinks]{ocgx2}

\declaretheorem[numberwithin=section]{theorem}
\declaretheorem[sibling=theorem, style=definition]{definition}
\declaretheorem[sibling=theorem]{lemma}
\declaretheorem[sibling=theorem, style=remark]{remark}
\declaretheorem[sibling=theorem]{proposition}

\declaretheorem[sibling=theorem]{problem}

\declaretheorem[sibling=theorem, style=remark]{example}

\begin{document}

\title[Surface embeddings in $\bR^2 \times \bR$]
{Surface embeddings in $\bR^2 \times \bR$}
\author[Menasco and Nichols]{William W. Menasco and Margaret Nichols}

\address{Department of Mathematics, University at Buffalo}
\email{menasco@buffalo.edu}

\address{Fields Institute for Research in Mathematical Sciences}
\email{mnichols@fields.utoronto.ca}


\begin{abstract}
This is an investigation into a classification of embeddings of a surface in Euclidean $3$-space.  Specifically, we consider $\bR^3$ as having the product structure $\bR^2 \times \bR$ and let $\p:\bR^2 \times \bR \rightarrow \bR^2$ be the natural projection map onto the Euclidean plane.  Let $ \e : S_g \into \bR^2 \times \bR$ be a smooth embedding of a closed oriented genus $g$ surface such that the set of critical points for the map $\p \circ \e$ is a smooth (possibly multi-component) $1$-manifold, $\cC \subset S_g$.  We say $\cC$ is the {\em crease set of $\e$} and two embeddings are in the same {\em isotopy class} if there exists an isotopy between them that has $\cC$ being an invariant set.  The case where $\p \circ \e|_\cC$ restricts to an immersion is readily accessible, since the turning number function of a smooth curve in $\bR^2$ supplies us with a natural map of components of $\cC$ into $\bZ$.  The Gauss-Bonnet Theorem beautifully governs the behavior of $\p \circ \e (\cC)$, as it implies $\chi(S_g) = 2 \sum_{\gamma \in \cC} t(\p \circ \e (\gamma))$, where $t$ is the turning number function.  Focusing on when $S_g \cong S^2$, we give a necessary and sufficient condition for when a disjoint collection of curves $\cC \subset S^2$ can be realized as the crease set of an embedding $\e: S^2 \into \bR^2 \times \bR$.
From there, we give the classification of all isotopy classes of embeddings when $\cC \subset S^2$ and $|\cC|=3$---a simple yet enlightening case. 
As a teaser of future work, we give an application to knot projections and discuss directions for further investigation.
\end{abstract}

\maketitle

\section{Introduction}
\label{Section:Introduction}
\subsection{Initial discussion of main results.}
\label{Subsec: main results}

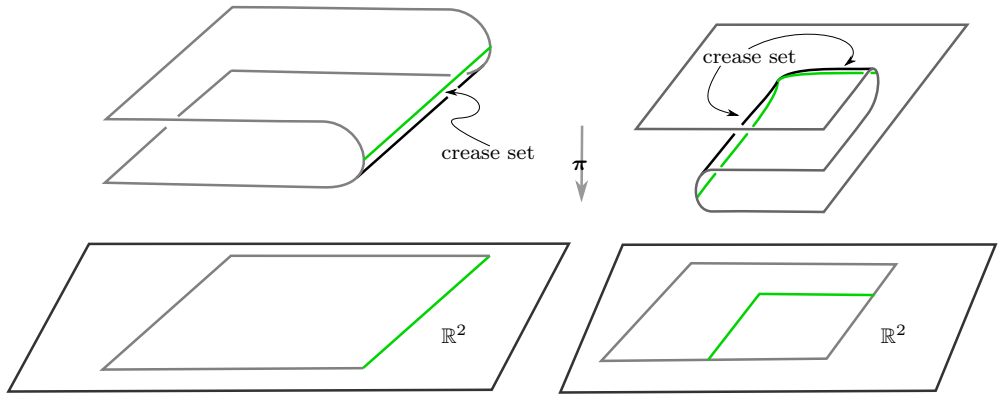
\begin{figure}[ht]

\labellist
\tiny

\pinlabel crease~set [tl] at 480 285
\pinlabel crease~set [t] at 830 395
\pinlabel $\p$ [l] at 623 257

\pinlabel $\bR^2$ [tl] at 480 85

\pinlabel $\bR^2$ [tl] at 970 85

\endlabellist

\centering
\includegraphics[width=.8\linewidth]{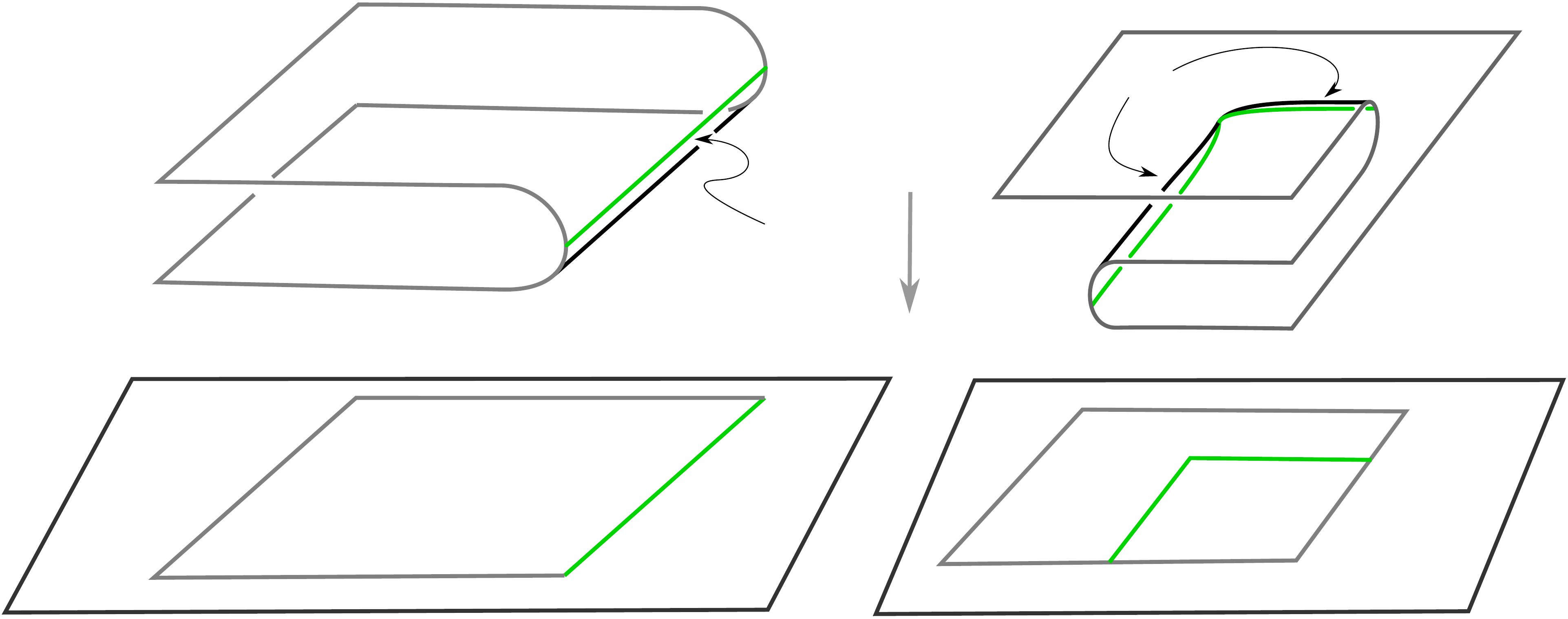}
\caption{The left illustration shows a generic neighborhood of a point in the crease set and its corresponding projection into $\bR^2$. The right illustration shows the neighborhood of a corner-point in the crease set and its corresponding projection into $\bR^2$. }
\label{Fig:creasing}
\end{figure}

We let $\e: S_g \into \bR^2 \times \bR (\cong \bR^3)$ be a smooth embedding of a closed surface of genus $g$.  Let $\p: \bR^2 \times \bR \rightarrow \bR^2$ be the natural projection coming from the product structure.  Let $\cC \subset S_g$ be the critical point set of $\p\circ \e$---the set of points, $x \in S_g$, where the differential, ${d(\p\circ \e)}_x : T_{x}(S_g)) \into T_{\p\circ \e (x)}(\bR^2)$ is not a surjective map.  We will refer to $\cC$ as the {\em crease set of $\e$}.  By general position arguments we can assume $\e(\cC) \subset \bR^2 \times \bR$ is a collection of smooth simple closed curves (s.c.c.'s) with a finite collection of marked {\em corners}---points where the immersion of the crease set into the plane, $(\p\circ \e)(\cC) \subset \bR^2$, fails to be smooth, with a local picture as in Fig.~\ref{Fig:creasing}. Moreover, we may assume that the only multi-point images of the projected crease set are transverse double points.  We refer to such well-behaved embeddings as {\em regular}.

\begin{figure}[ht]

\labellist
\small

\pinlabel bar~that~forces~dimple [l]	at	223.92	0																					

\endlabellist

\centering
\includegraphics[width=.4\linewidth]{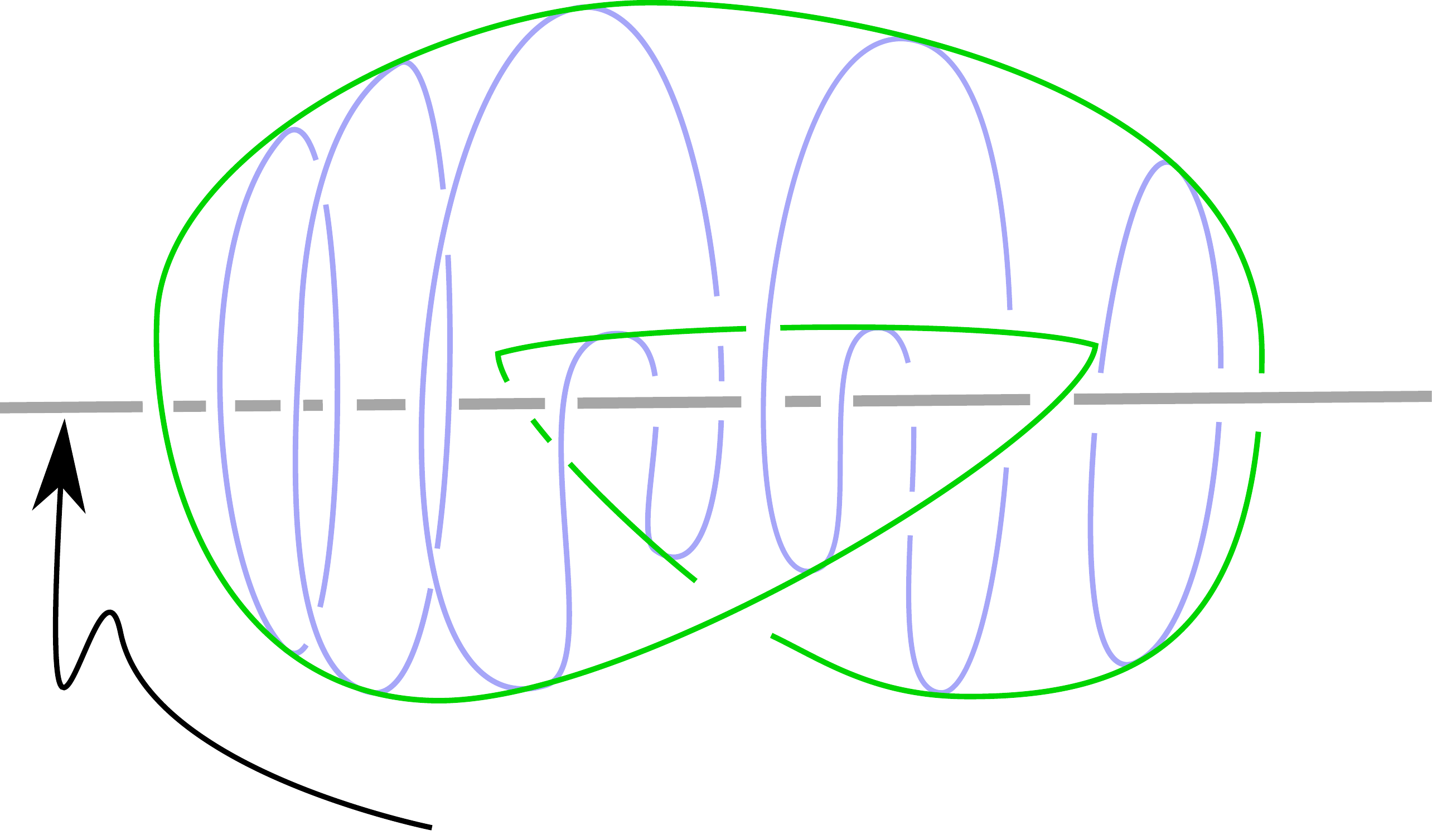}
\caption{The crease curve, shown in {\green green}, projects to a piecewise smooth curve.}
\label{Fig:Cheshire}
\end{figure}

The simplest, yet still interesting, example might be an embedding of the $2$-sphere, $S^2(=S_0)$, that has its crease set a single curve with two corners.  (See Fig.~\ref{Fig:Cheshire}.)  To metaphorically describe how this embedding might be obtained, one starts with the standard unit sphere in $\bR^3$ centered at the origin.  The reader could image forming a dimple in the ``clay ball'' the sphere bounds by taking a tilted rigid bar---say having a core line of \{$y+z = 1, \ x=1.5\}$---and pushing it into the ball via parallel translate from $x=1.5$ to $x=0$. Clay being clay, the bar deforms the round ball so that a dimple is created in the sphere.  The crease curve---the equator of $S^2$ before this deformation---appears deformed by a type-I Reidemeister isotopy move.  The new crease set projects to $\bR^2$ with a single double-point crossing.

This note is an initial investigation into the classification of {\em isotopy classes of regular surface embeddings into $\bR^2 \times \bR$}.   Specifically, we will consider two regular embeddings, $\e_0, \e_1: S_g \into \bR^2 \times \bR$, to be in the same isotopy class if there is a smooth isotopy $\e_t: S_g\to \bR^2 \times \bR, \ 0 \leq t \leq 1$, such that the marked structure of $\cC \subset S_g$, as the critical set of $\e_t$, is invariant for $t \in [0,1]$---we say such an isotopy is {\em regular}.  The number of corners of a crease component is unchanged by such an isotopy, and in particular, the equatorial crease set of the standard unit sphere and the two-corner crease set of the dimpled sphere are in different isotopy classes.

For a given $S_g$ and $\e$, and for a connected component $\gamma \subset \cC$, let $N(\gamma)$ be a closed annular neighborhood of $\gamma$ that is sufficiently small so that $N \cap \cC = \gamma$.  Let $\{ \hat{\gamma}, \hat{\gamma}^\prime \}= \partial N$.  In \S\ref{subsection: corners} we give a scheme for assigning an orientation to each component of $\partial N$.  For such oriented curves we can then define a ``turning number'' for their projections to $\bR^2$ of their embeddings into $\bR^2 \times \bR$.  We will then have a naturally associated $2$-tuple, $(t(\hat{\gamma}), t(\hat{\gamma}^\prime)) \in \bZ^2$---the {\em turning numbers} associated with $\gamma$.    (Please see \S\ref{section: preliminaries} for precise definitions.)  In the sphere-with-dimple example, this $2$-tuple is $(1,1)$ for the single component of $\cC$.

When we restrict a regular isotopy to a smooth oriented curve in $S_g$, its turning number will be invariant.  As such, we will see that our $2$-tuple is invariant within the isotopy class of an embedding.  The starting point for our investigation will be the establishment of relationships between our turning number $2$-tuple and the Euler characteristic of the surface, $\chi(S_g)$.  Specifically,

\begin{equation}\label{Eq: Gauss-Bonnet 1}
\chi(S_g) = \sum_{\gamma \in \cC}  t(\hat{\gamma}) + t(\hat{\gamma}^\prime). 
\end{equation}
This equality is  restated in Theorem~\ref{Theorem: eulerturning}.  It is properly seen as a Gauss-Bonnet type equation that governs the behavior of embeddings of $S_g$ into $\bR^2 \times \bR$.
Thus, we say that a $2$-tuple weight assignment to components of $\cC$ which satisfies the equations of Theorem~\ref{Theorem: eulerturning} corresponds to a {\em Gauss-Bonnet weighting of $\cC$}.

When we restrict to embedding classes of $S^2$ into $\bR^2 \times \bR$, Equ.~\ref{Eq: Gauss-Bonnet 1} becomes $\chi(S^2) = 2$.  We then have the following natural questions.
\begin{itemize}
    \item[1.] Given an collection of disjoint s.c.c.'s with corners, $\cC \subset S^2$, when is there a Gauss-Bonnet weighting of $\cC$?
    \item[2.] Given a collection of s.c.c.'s with corners, $\cC \subset S^2$, that has a Gauss-Bonnet weighting, is there an embedding, $\e: S^2 \into \bR^2 \times \bR$, that realizes $\cC$ as a crease set of $\e$?
\end{itemize}

Focusing on question 1, it is easy to produce collections that do not have any Gauss-Bonnet weightings---three non-concentric circles in $S^2$ for example.  However, when such a weighting exists it is unique.

\begin{theorem}
\label{Theorem: weighting uniqueness}
A Gauss-Bonnet weighting on a collection of disjoint s.c.c.'s with corners in $S^2$ is uniquely determined.
\end{theorem}

Let $(S^2 , \cC)$ and $(S^2, \cC^\prime)$ be two pairs of $2$-sphere/collection of disjoint smooth marked s.c.c.'s.  Assume each pair has a Gauss-Bonnet weighting.  We say the two pairs are {\em equivalent crease set configurations} if there exists a diffeomorphism of pairs, $h : (S^2 , \cC) \into (S^2, \cC^\prime)$.  With this equivalence in mind, Theorem~\ref{Theorem: weighting uniqueness} can be leveraged to prove the following finiteness result.  For its statement, we denote the number of corners of curve $\gamma \in \cC$ as $c(\gamma) (\in \bN)$.

\begin{theorem}
\label{Theorem: configuration finiteness}
Let $n_c, n_t \subset \bN$ be two arbitrary integers.  Then there are only finitely many possible crease set configurations in $S^2$ such that $\abs{\cC} \leq n_t$ and
$\sum_{\gamma \in \cC} c(\gamma) \leq n_c$.
\end{theorem}

For our initial attempt at answering question 2, we will simplify to the special case where $\cC \subset S^2$ is a collection of disjoint s.c.c.'s without any corners.  In this case, for a component, $\gamma \in \cC$, we will have $t(\hat{\gamma}) = t( \hat{\gamma}^\prime)$.

In this simplified case we have the following main result.

\begin{theorem}
\label{Theorem: embedding existence}
For any collection of disjoint smooth s.c.c.'s $\cC \subset S^2$ admitting a Gauss-Bonnet weighting, there exists a regular embedding, $\e: S^2 \into \bR^2 \times \bR$, which realizes $\cC$ as the crease set with the corresponding Gauss-Bonnet weighting.  That is, for each $\gamma \in \cC$, $t(\gamma)$ is equal to its weight.
\end{theorem}

We will establish Theorem~\ref{Theorem: embedding existence} by construction.  Our constructive argument will yield regular embeddings that are in an aesthetically nice form---the curvature function on each component of $\p \circ \e(\cC)$ will never be zero.

Surprisingly (at least to the authors), this aesthetic feature is not always achievable.  In particular, the relatively simple situation when $\cC \subset S^2$ is just three concentric circles without corners there does exist an 
``non-intuitive'' embedding where the curvature function on $\p \circ \e(\cC)$ must have points of zero curvature.  Regardless, it is still possible to perform a calculation that gives a complete classification of the isotopy classes for when $|\cC|=3$.  We view this novel calculation as prescient of what is possible for when $|\cC|$ and genus are higher.

\begin{theorem}
\label{theorem: |cC| = 3}
Up to reflection, there are exactly three isotopy classes of regular $S^2$ embeddings into $\bR^2 \times \bR$ when $\cC$ is just three curves without corners.
\end{theorem}

As a final remark to the section, the authors have conducted a wide literature search and have found no previous investigation into the critical set of $\p \circ \e$---the crease set---with the exception of the recent contribution of Joel Hass \cite{[H]}. Since the Gauss-Bonnet theorem is a classical premier result of differential topology that has produced a significant body of applications and generalizations, this lack of previous interest in the crease set is puzzling.  Regardless, as the arguments in this paper will illustrate the crease set is a source of significant control over the behavior and positioning of arbitrary surface embeddings in $\bR^3$.

\subsection{A link projection application}
\label{subsection: knot application}


Our original motivation for studying surface embeddings into $\bR^2 \times \bR$ comes from knot theory.  Early work of the first author developed a ``normal form'' for representing essential surfaces in $S^3$ link complements with respect to a link projection \cite{[M1]}.

The key construction for this normal form involves positioning a link $L$ to lie in $\bR^2$ (coming from its projection) except at crossings, where the two crossing strands lie on the boundary of a small ``crossing ball''. A surface $S\subset \bR^3\setminus L$ in normal form can be reconstructed entirely from its intersections with $\bR^2$ and these crossing balls. To set notation, let $\mathbf{R}^2_+$ (respectively, $\mathbf{R}^2_-$) be $\bR^2$ with each disk neighborhood of a crossing replaced by the upper (respectively, lower) hemisphere of a crossing ball. The surface intersects  $\mathbf{R}^2_+ \cup \mathbf{R}^2_-$ in a graph, and when in normal form the crease set is a collection of disjoint circuits/curves of this graph.

\begin{figure}[ht]

\labellist
\small

\pinlabel	$T$	[c]	at	375	395

\endlabellist

\centering
\includegraphics[width=.5\linewidth]{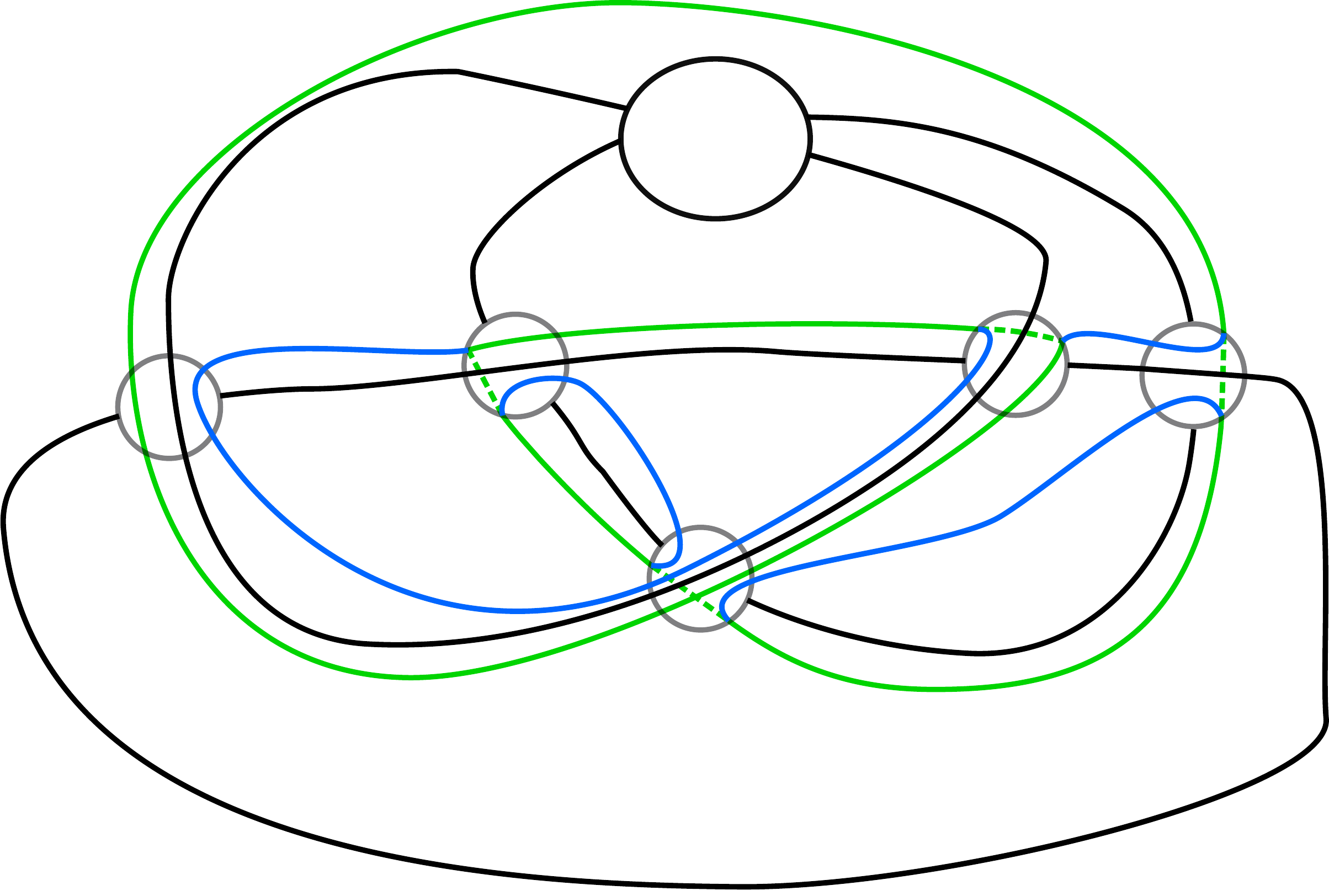}
\caption{The {\blue blue} and {\green green} curves depict $S^2 \cap \mathbf{R}^2_+$ with the green portion indicating which part is contained in the crease set.  There is an arbitrary 2-tangle contained inside the ball labeled $T$.}
\label{Fig:Chee-lik}
\end{figure}

In Fig.~\ref{Fig:Chee-lik} we depict the regular projection (with the crossing ball structure) of a link ``template'' realizing the dimpled $S^2$ of Fig.~\ref{Fig:Cheshire} as a surface in normal form.  The crease set is a subset of the graph $\Gamma = S^2 \cap (\mathbf{R}^2_+ \cap \mathbf{R}^2_-)$ and corresponds to the green curve.  Inside the region of the link template labeled $T$ we can place any $2$-strand tangle.  Thus, there are infinitely many possible link projections that realize the dimpled $S^2$ as a surface in normal form.  This example is illustrative of the general situation (Theorem \ref{theorem: normal form}) which we state here in a more colloquial manner.

\begin{theorem}
\label{theorem: crease set in link projections}
Every isotopy class of regular surface embeddings into $\bR^2 \times \bR$ can be realized as in normal form with respect to infinitely many regular link projections.
\end{theorem}

\subsection{Outline of paper}
\label{subsection: outline}

In \S\ref{section: preliminaries} we give the formal definitions of the crease set and turning number plus some additional concepts of our machinery.  In \S\ref{section: proof of crease set in link projections} we give the proof of Theorem~\ref{theorem: crease set in link projections}.  In \S\ref{section: GB equations} we discuss the Gauss-Bonnet equations and results that are behind Equ.~\ref{Eq: Gauss-Bonnet 1}, which we then use to prove Theorems~\ref{Theorem: weighting uniqueness} and~\ref{Theorem: configuration finiteness}. Then in \S\ref{section: constructing embeddings} we give a construction that establishes Theorem~\ref{Theorem: embedding existence}.  In \S\ref{section: classification set-up} we develop the additional machinery need to prove Theorem \ref{theorem: |cC| = 3}.  In particular, in \S\ref{subsection: branched surface} we observe that for each isotopy class of a regular embedding of a surface into $\bR^2 \times \bR$ there is an associated naturally embedded {\em branched surface with boundary}.  A salient feature of this natural branched surface is a correspondence between its boundary and branching locus curves, and the components of the crease set.  (The reader may correctly suspect that the branched surface associated with the Fig. \ref{Fig:Cheshire} embedding is one containing the Lorenz template \cite{[G-L]}.)
Finally, in \S\ref{Section: problems} we advance directions for further investigation.

\section*{Acknowledgements}
The first author would like to thank Adam Sikora for the brief conversation pre-COVID-19 that was the spark for thinking about the crease set of an embedding. The second author is grateful to the Fields Institute for its support and hospitality.  The authors are grateful to Joel Hass for alerting them to his work.



\section{Preliminaries}
\label{section: preliminaries}

\subsection{Definition of the crease set}
\label{subsection: definition of crease set}
Let $S_g$ be a smooth closed oriented surface of genus $g$. Let $\e:S_g \into \bR^2 \times \bR$ be a smooth embedding and  $\p: \bR^2 \times \bR \rightarrow \bR^2$ be the natural projection onto the first factor in the product structure.

\begin{definition}
The {\it crease set} $\  \cC \subset S_g$ of $\e$ (with respect to $\p$) is the critical point set of $\p\circ \e$.
\end{definition}

We would like to restrict to embeddings where $\cC$ is a nice subset of $S_g$---ideally, a smooth submanifold.

Fix an embedding $\e:S_g \to \bR^3$. Note that for any matrix $A\in{\rm SO}(3)$, post-composing $\e$ by the map $x\mapsto Ax$ gives a new embedding, which we denote $\e_A$.

Let $N: S_g\to S^2$ denote the Gauss map corresponding to $\e(S_g)$, so $N(p)$ is the outward-pointing unit normal vector to $\e(p)$. Observe that $\cC=N^{-1}(Z)$, where $Z\subset S^2$ denotes the equator of $S^2$.

Define $f:S_g\times {\rm SO}(3)\to S^2$ by $f(p,A)=A\cdot N(p)$. This map takes a point $p$ to the outward normal vector at $\e(p)$, then rotates it by $A$. Alternatively, $f$ maps $p$ to its outward unit normal vector in $\e_{A}$; seen this way, $f$ parametrizes the Gauss maps associated to the different embeddings $\e_A$.

We claim that $f$ is a submersion. To see this, observe that $df_{(p,A)}$ restricted to $T_{(p,A)}{\rm SO}(3)$ already surjects: varying $A$ infinitesimally directly corresponds to applying an infinitesimal rotation to $S^2$, perturbing $f(p,A)$ in that direction.

As a submersion, $f$ is consequently transverse to the equator $Z\subset S^2$. By parametric transversality, almost every $f(\,\cdot\,, A)$ is transverse to $Z$. For such $f(\,\cdot\,, A)$, the preimage of $Z$, which is the crease set $\cC$ corresponding to $\e_A$, is a smooth submanifold of $S_g$.

In this case, $\p\circ\e$ restricts to a smooth map on $\cC$. When $d(\p\circ \e|_\cC)$ has rank 0, the tangent line to $\cC$ in $\e(S_g)$ is vertical (with respect to $\p$). These are precisely the points where $\p\circ\e(\cC)$ may fail to be smooth. Notice that unless the local picture of $\e(S_g)$ is like that of the right-side illustration in Fig.~\ref{Fig:creasing}, $\e$ can be perturbed by a generic small rotation to remove this rank-0 point. A rank-0 point cannot be eliminated precisely when it corresponds to a change in the {\em sign} of the crease set (see \S\ref{subsection: Orientation of crease}); we call such a point a {\em corner} of $\cC$.

\begin{definition}
An smooth embedding $\e: S_g\to \bR^2\times\bR$ is {\em regular} with respect to $\p$ if it satisfies the following conditions:
\begin{enumerate}
    \item the Gauss map of $\e(S_g)$ is transverse to the equator of $S^2$;
    \item the differential $d(\p\circ \e|_\cC)$ has rank 1 except for a finite collection of points, which are all corners; and
    \item the immersed, piecewise smooth $\p\circ\e(\cC)$ self-intersects only in transverse double points.
\end{enumerate}
\end{definition}

The preceding discussion shows regular embeddings are generic. Unless otherwise stated, for the remainder of this paper all embeddings will be regular.  We consider $\cC\subset S_g$ to be a collection of smooth s.c.c.'s with marked points corresponding to the corners.

\subsection{Folding sign of a crease curve}
\label{subsection: Orientation of crease}

Since any embedding of an orientable closed surface into Euclidean $3$-space is the boundary of a unique compact $3$-submanifold, $M_S$, we have a well-defined orientation for $\e(S_g)$ coming from the outward pointing normal vector field with respect to this submanifold.  For this natural orientation of $\e(S_g)$, in the neighborhood $N_x \subset S_g$ of a rank-1 point $x \in \cC$ there are two possible embeddings.  As illustrated in Fig.~\ref{Fig:crease orientation}, either the outward pointing normal vectors are ``pointing towards each other'' or ``away from each other''.  For the former situation we say the {\em crease folds negatively (or ``$-$'') along the segment}.  (See left illustration of Fig.~\ref{Fig:crease orientation}.)  For the latter situation the {\em crease folds positively (or ``$+$'') along the segment}.  (See right illustration of Fig.~\ref{Fig:crease orientation}.)

\begin{figure}[ht]
\centering
\includegraphics[width=.5\linewidth]{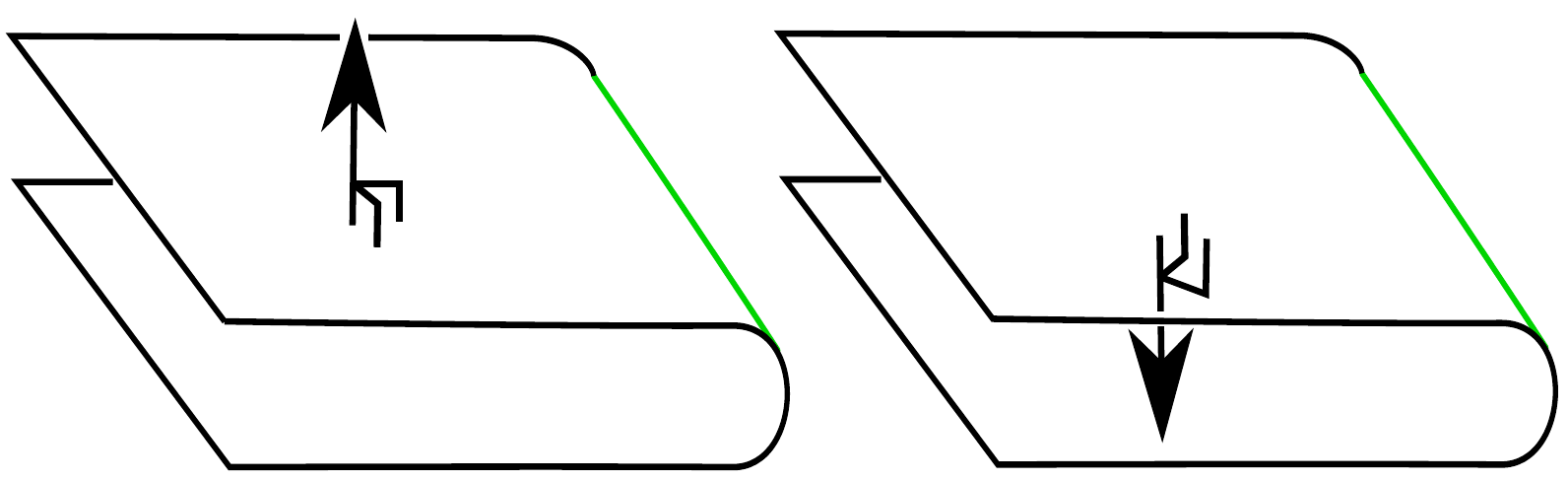}
\caption{On the left, a positive fold, and on the right, a negative fold.}
\label{Fig:crease orientation}
\end{figure}

Alternatively,  we can assume that $N_x \subset S_g$ has been chosen small enough neighborhood of $x \in S_g$ so that $\p\circ \e (\cC) \cap \p\circ \e (N_x \setminus \cC) = \varnothing$.
Let $y \in \p\circ \e( N_x \setminus \cC) $ and $\p^{-1}(y) \cap \e(N_x)  = \{ x_1 , x_2 \} \subset \e(N_x \setminus \cC)$.  Let $\overline{x_1 x_2} \subset \bR^2 \times \bR$ be the compact line segment for which $\p(\overline{x_1 x_2}) = y$.  Then $\e(N_x \cap \cC)$ is a negatively folding segment if $\overline{x_1 x_2} \cap M_S = \{x_1 , x_2 \}$.  If $\overline{x_1 x_2} \cap M_S= \overline{x_1 x_2}$ then $\e(N_x \cap \cC)$ is a positively folding segment.

Note that along arcs of rank-1 points of $\cC$, the folding direction is constant; it can only switch at rank-0 points.  Any component of $\cC$ that is without corners is either a positively folding s.c.c.\ or a negatively folding s.c.c., and more generally a component passes through an even number of corners. We denote the collection of all positive arcs and components $\cC^+$ and similarly all negative arcs and components $\cC^-$.

The reader should observe that the folding sign of either a segment in $\cC$ between two corners, or a corner-free component of $\cC$ is an invariant of an embedding's regular isotopy class.

\subsection{Partial turning numbers and external angles}
\label{subsection: corners}

We start by defining the  ``partial turning number'' of an oriented smooth arc in $\bR^2$.  Let $\alpha : [0,1] \into \bR^2$ be a smooth unit speed arc and define $\alpha^\prime (0) := \lim_{s \rightarrow 0^+} \alpha^\prime (s)$ and $\alpha^\prime (1) := \lim_{s \rightarrow 1^-} \alpha^\prime (s)$.  Then $\alpha^\prime (s) = (\cos(s(t)) , \sin(s(t)))$ for some angle function $s(t)$.   Then the {\em partial turning number} is $$t_p(\alpha) := \frac{s(1) - s(0)}{2 \pi} \in \bR.$$

Note that in the case that $\alpha$ is a smooth closed curve---$\alpha(0) = \alpha(1)$ and $\alpha^\prime (0) = \alpha^\prime (1)$---this definition agrees with the usual definition of the turning number:
$$t(\alpha) = t_p(\alpha)=\frac{s(1) - s(0)}{2 \pi} \in \bZ.$$

We observe that the above definitions are independent of parameterization, and are thus well defined.

Fix an orientation of $\bR^2\subset \bR^2\times\bR$. Let $\K \subset S_g \setminus \cC $ be a connected component.  We consider the closure of $\K$, that is $\bar\K = \K \cup \partial \K$ where $\partial \K \subset \cC$.  (We will be utilizing this association between $\K$ and $\bar\K$ throughout this paper and $\bar\K$ will always be notation for a compact connected surface.) The orientation of $\bR^2$ pulls back under $\p\circ\e$ to give an orientation of $\bar{\K}$; this extends to an orientation on each curve $\gamma\subset\partial \K$. The reader should observe that $\gamma$ is necessarily a boundary component of two distinct components $\K,\K'$ of $S_g \setminus \cC$, and both will induce the same orientation on $\gamma$. Thus we can consider $\cC$ as a collection of oriented curves.

Next, we consider the ``external angle'' of a corner $x \in\gamma\subset\cC$.  Let $\alpha, \beta \subset \p \circ \e(\gamma)$ be adjacent smooth arcs to $\p\circ\e(x)$.   For all such corners we require the technical assumption that the angle between $\alpha$ and $\beta$ at $\p\circ \e (x)$ be $\frac{\pi}{2}$---the angle between their corresponding endpoint tangent vectors.   The reader should observe that this assumption is not limiting since we can perform an isotopy (and, indeed, a regular isotopy) of $\e(S_g)$ in a neighborhood of each corner to obtain this technical assumption.

\begin{figure}[ht]
\centering

\begin{subfigure}{.32\textwidth}

\labellist
\tiny

\pinlabel	external~angle [r] at 174 316

\endlabellist
\centering
\includegraphics[width=\textwidth]{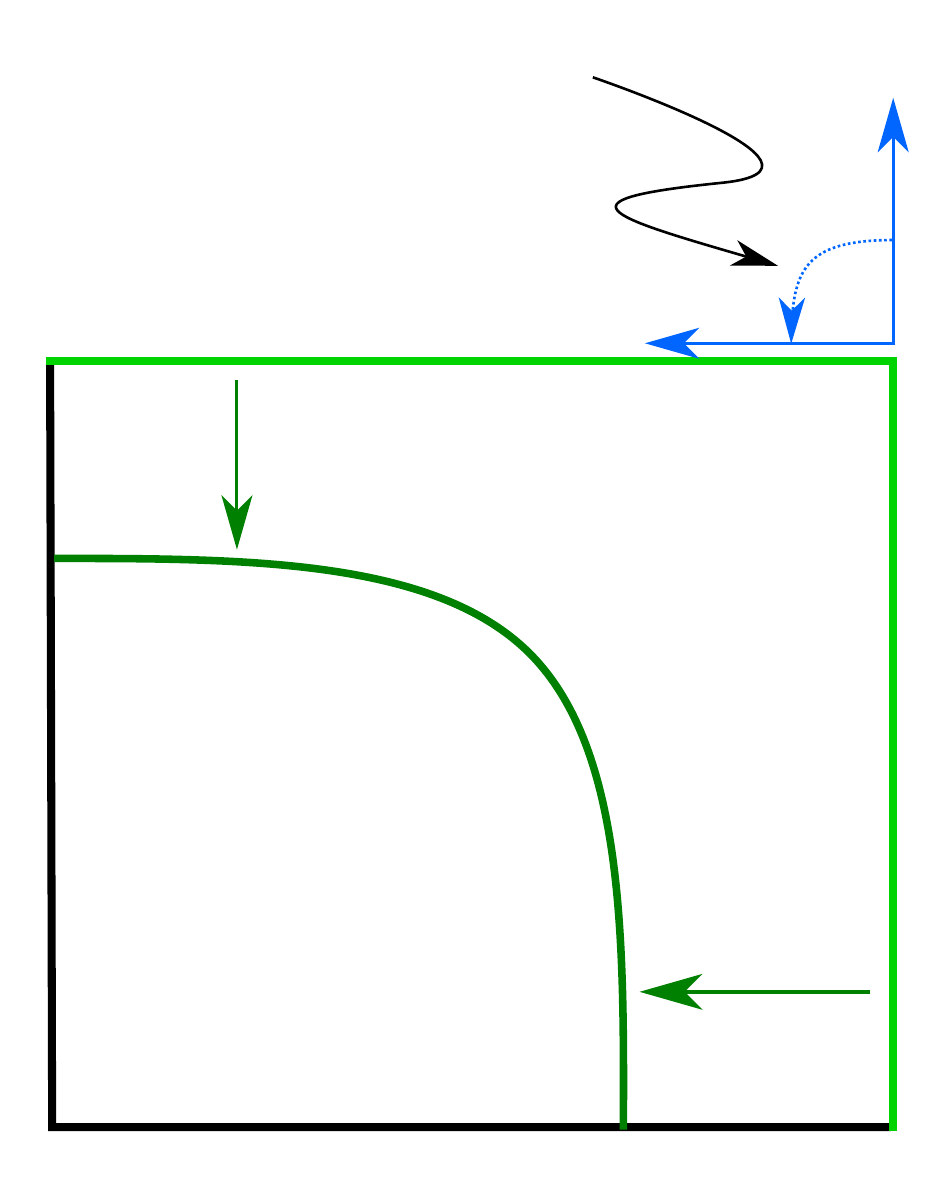}
\caption{An external angle of $\frac{\pi}{2}$.}
\label{Fig: corner-twist1}
\end{subfigure}
\hspace{1em}
\begin{subfigure}{.4\textwidth}

\labellist
\tiny

\pinlabel	external~angle [t] at 192 151

\endlabellist
\centering
\includegraphics[width=\textwidth]{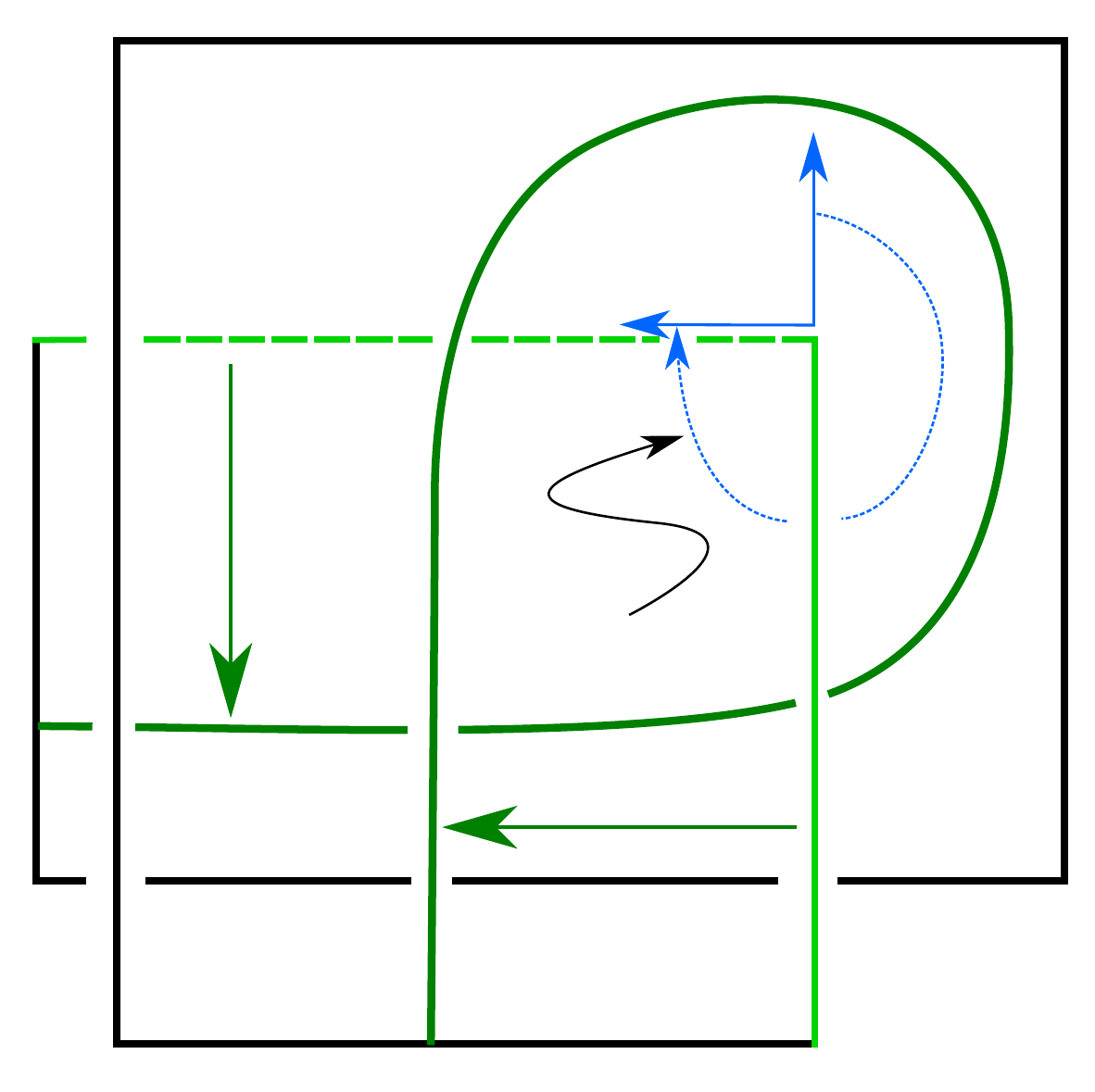}
\caption{An external angle of $\frac{3\pi}{2}.$}
\label{Fig: corner-twist2}
\end{subfigure}
\caption{A top-down view of the two layers of $\e(N_x)$ at a corner $x$. The {\green green} segments depict the portion of $N_x$ that is in $\cC$.  The {\blue blue} vectors correspond to the associated external angles.  The {\dgreen dark green} arcs correspond to the push-offs of $\cC$ into $N_x$.}
\label{Fig: corner-twist}
\end{figure}

We consider the embedding of the two components of $N_x \setminus \cC$ when $x$ is a corner of $\cC$.  In Fig.~\ref{Fig: corner-twist} we have an illustration of the two embedded components of $\e(N_x \setminus \cC)$. We are interested in computing the {\em external angle with respect to $K$}, $\angle_K(x)$, at $\p\circ \e (x)$, where $K$ is the subsurface containing a given component of $N_x\setminus \cC$. 
As shown in Fig.~\ref{Fig: corner-twist}, this is the angular ``swing'' between the tangent vector coming into $\p\circ\e(x)$ and the tangent vector coming out of $\p\circ\e(x)$.
As we have fixed the two arcs meeting at $\p\circ\e(x)$ to meet at an angle of $\frac{\pi}{2}$, there are two possible values for $\angle_K(x)$: $+\frac{\pi}{2}$, as in Fig.~\ref{Fig: corner-twist1}, and $-\frac{3 \pi}{2}$, shown in Fig.~\ref{Fig: corner-twist2}. 
(Here we are making the positive/negative sign assignments in the classical manner---counterclockwise is positive, clockwise is negative.)

An alternative approach to determining the external angle at a corner is to take a push-off of $\e(N_x \cap \cC)$ into the two pieces of $\e(N_x \setminus \cC)$ and compute the partial turning numbers of the projected arcs.  In particular, let $ \alpha^x = N_x \cap \cC$ be a crease segment with corner $x$.  We consider the two smooth push-offs, $\hat{\alpha}^x_{+{1}/{2}}$ and $\hat{\alpha}^x_{-{3}/{2}}$, into the two components of $N_x \setminus \alpha^x$. In both illustrations of Fig.~\ref{Fig: corner-twist}, the oriented brown curve depicts such a push-off, with $\p\circ \e (\hat{\alpha}^x_{+{1}/{2}})$ on the left and and $\p\circ \e (\hat{\alpha}^x_{-{3}/{2}})$ on the right.  Then $t_p(\hat{\alpha}^x_{+{1}/{2}}) = + \frac{1}{4}$, since the tangent vector of $\p\circ e (\hat{\alpha}^x_{+{1}/{2}})$ has a counterclockwise $\frac{1}{4}$-turn; and, the partial turning number is $t_p(\hat{\alpha}^x_{-{3}/{2}}) = -\frac{3}{4}$ since the tangent vector of $\p\circ \e (\hat{\alpha}^x_{+{3}/{2}})$ has a clockwise $\frac{3}{4}$-turn.  The external angle of $x$ with respect to a component of $N_x \cap \cC$ is then $2\pi$ times the associated partial turning number.

\subsection{Weighting of the crease set}
\label{subsection: weighting curves}

Let $\K \subset S_g \setminus \cC $ be a connected component.
Let $\gamma \subset \partial \K$ be a boundary component with corners $\{ x_1, \ldots , x_n \} \subset \gamma$ connecting smooth subarcs $\{\alpha_1 , \ldots , \alpha_n\} \subset \gamma$.   Then we define the {\em Gauss-Bonnet weight with respect to $K$} of $\gamma$ to be
$$t_\K (\gamma) =  \sum_{1 \leq i \leq n} t_p(\p\circ\e(\alpha_i)) + \frac{\angle_K(x_i)}{2\pi}.$$
When $\gamma$ has no corners, this is the turning number of $\p\circ\e(\gamma)$. The reader should observe that Gauss-Bonnet weights are well-defined, independent of our orientation choice for $\bR^2$.

An equivalent method of defining $t_\K(\gamma)$ is to take a smooth push-off of $\gamma$ into $\K$, $\hat{\gamma} \subset \K$, that inherits its orientation from $\gamma$; then $t_\K(\gamma) := t(\p\circ\e(\hat{\gamma}))$. This is effectively applying the alternative definition of $\angle_K$ to each corner along $\gamma$ simultaneously. In particular, this shows that $t_K$ always takes integral values and, moreover, does not depend on our choice to fix the corner angles of $\p\circ\e(\cC)$ as $\frac{\pi}{2}$.

\subsection{The decorated surface}
\label{subsection: further labeling}

We now consolidate the data that we have developed in previous sections into a {\em decorated surface}, $(S_g, \cC, {\bf T})$, associated with an embedding, $\e: S_g \into \bR^2 \times \bR$.  Specifically, ${\bf T}$ denotes the following function.

\noindent
{\bf Turning number function: ${\bf T}: \cC \into \bZ^2$.}  Let $\gamma \subset \cC$ and consider the distinct connected components $\K, \K^\prime \subset S_g \setminus \cC$ such that $\gamma \subset \partial \bar\K$ and $\gamma \subset \partial \bar{\K^\prime}$.  Then ${\bf T} : \gamma \mapsto (t_\K(\gamma),t_{\K'}(\gamma)) \in \bZ^2$, the Gauss-Bonnet weights of $\gamma$ with respect to $\K$ and $\K'$. 
If $\gamma$ has no corners then $t_\K=t_{\K'}$ and we simplify to a $1$-tuple.



\section{Realization coming from link complements}
\label{section: proof of crease set in link projections}

Let $L: S^1 \sqcup \cdots \sqcup S^1 \into \bR^2 \times \bR$ be an embedding of a link such that $(\p\circ L)(S^1 \sqcup \cdots \sqcup S^1) \subset \bR^2$ is a regular projection of a link $L$.  We now augment this link projection by replacing a small disk neighborhood of each crossing with a ``crossing ball'' as shown in Fig.~\ref{Fig: crossing ball1}.  Let $\mathbf{R}^2_+$ (respectively, $\mathbf{R}^2_-$) be $\bR^2$ with each disk neighborhood of a crossing replaced by the upper (respectively, lower) hemisphere of a crossing ball.
Let $\mathbf{R}^3_+$ (respectively, $\mathbf{R}^3_-$)  be upper half-space (respectively, lower half-space) that has $\mathbf{R}^2_+$ (respectively, $\mathbf{R}^2_-$) as its boundary plane.  The salient feature of this construction is that there is now an embedding, $\bar{L}: S^1 \sqcup \cdots \sqcup S^1 \into \mathbf{R}^2_+ \cup \mathbf{R}^2_- \,(\subset \bR^2 \times \bR)$ such that $\p\circ L = \p\circ \bar{L}$.  We will use $\bar{L}$ to denote of the embedding of the link into $\mathbf{R}^2_+ \cup \mathbf{R}^2_-$ for the remainder of this discussion.  The reader should observe that $L$ and $\bar{L}$ are isotopic to each other in $\bR^2 \times \bR$.

\begin{figure}[ht]
\centering
\hfill
\begin{subfigure}{.4\textwidth}

\labellist
\small

\pinlabel	crossing	[bl]	at	-5	116.21
\pinlabel	ball	[tl]	at	-5	116.21
\pinlabel	$L$	[tl]	at	190	18

\endlabellist

\centering
\includegraphics[width=\textwidth]{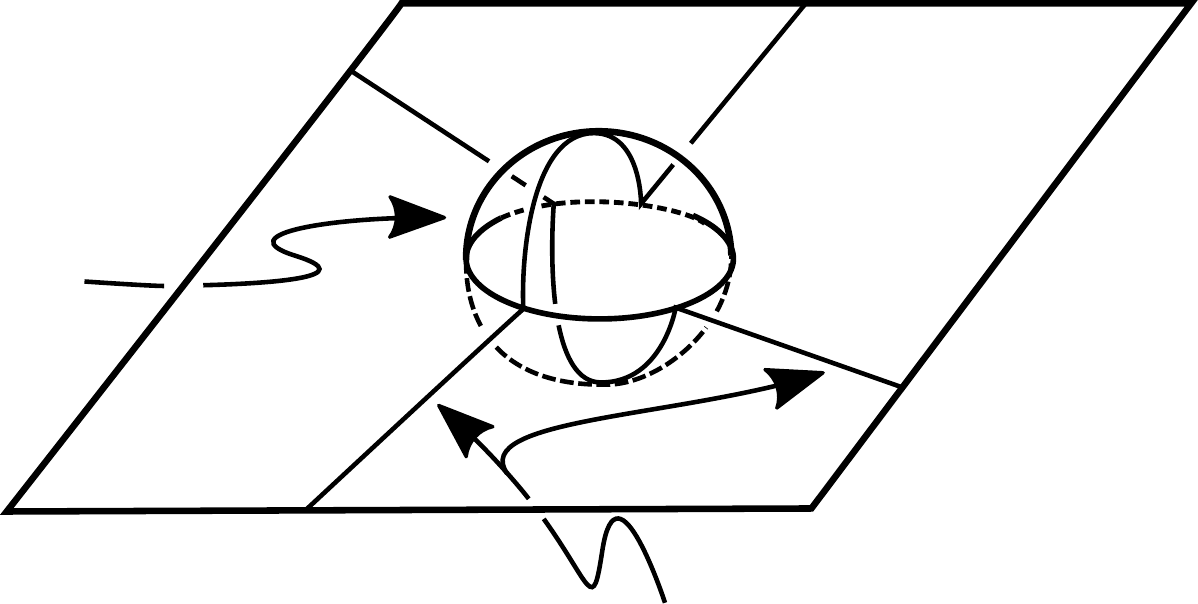}
\caption{A crossing ball.}
\label{Fig: crossing ball1}
\end{subfigure}
\hfill
\begin{subfigure}{.4\textwidth}
\centering
\includegraphics[width=.5\textwidth]{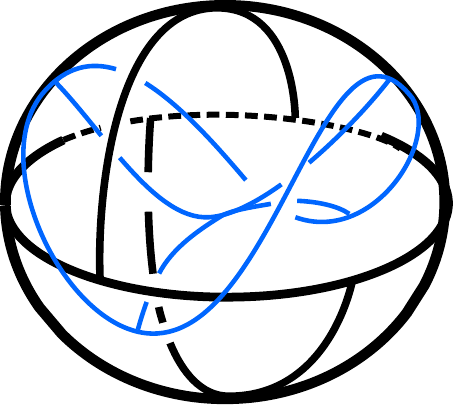}
\caption{A saddle disk, depicted in {\blue blue}.}
\label{Fig: crossing ball2}
\end{subfigure}
\hfill\ 
\caption{How an essential surface lies near a crossing of $L$.}

\label{Fig: crossing ball}

\end{figure}

\subsection{Normal form}
\label{subsection: normal form link}

We consider an essential closed surface, $S \subset \bR^3 \setminus \bar{L}$.  By general position arguments we can assume that $S \cap \mathbf{R}^2_\pm$ is a collection of s.c.c.'s and that any nonempty intersection of $S$ with a crossing ball is a collection of ``saddle'' disks.  (See Fig.~\ref{Fig: crossing ball2}.)  Moreover, by requiring $S$ to be an essential surface we can assume that $S \cap \mathbf{R}^3_+$ and $S \cap \mathbf{R}^3_-$ are collections of disks.  The embedding of $S$ is in {\em normal form} with respect to $\mathbf{R}^2_\pm$ if the following three conditions hold:

\begin{itemize}
    \item[(1)] every s.c.c.\ of $S \cap \mathbf{R}^2_\pm$ intersects at least two distinct crossing balls;
    \item[(2)] for a crossing ball $\mathbf{B}$ and s.c.c.\ $\gamma \in S \cap \mathbf{R}^2_\pm$, $\gamma$ intersects $\mathbf{B}$ in at most 2 arcs.  If $\gamma$ meets some $\mathbf{B}$ in two arcs, they are on the boundary of a common saddle disk in $\mathbf{B}$.
    \item[(3)] $S \setminus (S \cap \mathbf{R}^2_\pm)$ is a collections of open disks in $\mathbf{R}^3_\pm$ and saddle disks in crossing balls.  For any disk component $\D_\gamma \subset S \cap \mathbf{R}^3_\pm$, with boundary curve $\gamma \in S \cap \mathbf{R}^2_\pm$, we require $\p$ restricts to a diffeomorphism from $\D_\gamma$ onto its image, the disk that $\p(\gamma) \,(\subset \bR^2)$ bounds.  (The reader should observe that requiring $S \setminus (S \cap \mathbf{R}^2_\pm)$ to be a collection of disks is natural due to the incompressibility of the essential surface $S$.) 
\end{itemize}
Using general position arguments that are well known in the literature and a number of graduate texts (see for example \cite{[H-T-T1],[H-T-T2],[Li], [M1],[M2]}), one can establish the following theorem.  Thus, due to the repetitive nature of any proof we offer the following result without argument.
\begin{theorem}
\label{theorem: normal form}
Let $L \subset \bR^2 \times \bR$ be a link and $\bar{L}$, $\mathbf{R}^2_\pm$ be the associated spaces as described above.  Let $F \subset (\bR^2 \times \bR) \setminus L$ be an closed incompressible surface in the link's exterior.
Then there exists a closed incompressible surface $\bar{F} \subset (\bR^2 \times \bR) \setminus \bar{L}$ such that:
\begin{itemize}
    \item[(a)] $\bar{F}$ is in normal form with respect to the $\mathbf{R}^2_\pm$;
    \item[(b)] $\chi(F) = \chi(\bar{F})$;
    \item[(c)] if $L$ is a non-split link then $(L , F)$ is pairwise isotopic to $(\bar{L}, \bar{F})$ in $\bR^2 \times \bR$;
\end{itemize}
\end{theorem}

We caution that not all surfaces in normal form with respect to a link projection are incompressible/essential, despite being ``captured'' by a link projection. For examples, projections of {\em hard unknots} (\cite{[G],[K-L]}) will have their non-essential peripheral torus in normal form with respect to their projection.

The salient feature of such a surface in normal form is that the curves of crease set live in the 3-valent graph, $\Gamma = S \cap (\mathbf{R}^2_+ \cup \mathbf{R}^2_-)$.  This can be seen from condition (3) above.  Since for each open disk $\D \subset S \setminus \Gamma$ we have $\p|_{\D}$ is a diffeomorphism onto an open disk in $\bR^2$, we can place a flat structure on the disks of $S \setminus \Gamma$.  Thus, the crease set of a surface in normal form will be exactly the set of edges of $\Gamma$ where the a flat structure of a disk cannot be extended across an edge to an adjacent disk of $S \setminus \Gamma$.

The condition for deciding which edges of $\Gamma$ live in $\cC$ is straight forward.  Let $\alpha \subset S \cap (\mathbf{R}^2_+ \cap \mathbf{R}^2_-) \subset \Gamma$ be an arc which will necessarily lie away from crossing balls.  Then $\alpha$ will be adjacent to two disks $\D_+ \subset S \cap \mathbf{R}^3_+$ and $\D_- \subset S \cap \mathbf{R}^3_-$.  If $\p({\rm int}\,\D_+) \cap \p({\rm int}\,\D_-) \not= \varnothing $ near $\alpha$ then $\alpha$ is in the crease set of $S$.  (Here, ${\rm int}\,\D_\pm$ is the interior of the disks.)  Similarly, consider an arc $\alpha \subset S \cap \partial \mathbf{B}$, where $\mathbf{B}$ is a crossing ball.  Then $\alpha$ is adjacent to a disk $\D \subset S \cap (\mathbf{R}^3_+ \cup \mathbf{R}^3_-)$ and a saddle disk $\D_{\mathbf{B}} \subset \mathbf{B}$.  If $\p({\rm int}\,\D) \cap \p({\rm int}\,\D_{\mathbf{B}}) \not= \varnothing$ near $\alpha$, then $\alpha$ is in the crease set of $S$.

\subsection{Proof of Theorem~\ref{theorem: crease set in link projections}}
\label{subsection: proof of crease set in link projections}

Since $\e(S_g) \subset \bR^2 \times \bR$ is an orientable surface, we choose an orientation and we let ${\bf N}$ be the associate oriented normal bundle.

Next, for each component, $\gamma \in \cC$, we take the two smooth normal push-offs, $\gamma_+$ in the $N$-direction and $\gamma_-$ in the $-N$-direction.  After a small isotopy that places these push-offs in position so that we obtain a regular link projection under $\p$, it is readily confirmed that $\e(S_g)$ is in normal form with respect to $L$ and $\p \circ \e(\cC) \subset \p(\Gamma)= \e(S_g) \cap (\mathbf{R}^2_+ \cup \mathbf{R}^2_-)$.

Once a single link projection $\p(L)$ is obtained such that $\p \circ \e (\cC) \subset \Gamma $, an infinite family of links is obtained by letting strands that intersect a common region of $(\mathbf{R}^2_+ \cup \mathbf{R}^2_-) \setminus \Gamma$ of the link entangle each other.  (See the tangle $T$ in Fig.~\ref{Fig:Chee-lik}.)
 \endproof


\section{Gauss-Bonnet equations}
\label{section: GB equations}

In this section we will show how the classical Gauss-Bonnet Theorem from differential geometry governs the behavior of the turning number function of a decorated surface.

Specifically, let $(F,\mathpzc{g})$ be a compact oriented Riemannian surface with boundary.  (For a development of this topic please see \cite{[L]}.)  We let ${\bf K}$ denote Gaussian curvature and $k_g$ denote geodesic curvature.  We will be considering the case where the components of $\partial F$ are piecewise smooth, so we let $\{v_1, \ldots , v_m \} \subset \partial F$ be the collection of boundary vertices.  We denote the external angle at a vertex by $\theta(v_i), 1 \leq i \leq m$.  Then the Gauss-Bonnet Theorem states:
\begin{equation}
\label{equation: Gauss-Bonnet}
    \int_F {\bf K} \ dA + \int_{\partial F} k_g \ ds + \sum_{v_i \in \partial F} \theta(v_i) = 2 \pi \cdot \chi(F).
\end{equation}

For a regular embedding $\e: S_g \into \bR^2 \times \bR$, we recall the discussion of the subsurface $\bar\K \subset S_g$ from \S\ref{subsection: weighting curves}.  We now adapt Equ.~\ref{equation: Gauss-Bonnet} to our study of regular embeddings of surfaces into $\bR^2 \times \bR$.

\begin{proposition}
\label{proposition: GB for barK}
Given a regular embedding $\e: S_g \into \bR^2 \times \bR$ and an associated connected subsurface, $\K \subset S_g \setminus \cC$, for the restricted regular embedding $\e: \bar\K \into \bR^2 \times \bR$, we have
\begin{equation}
    \label{equation: GB for barK}
    \sum_{\gamma \in \partial \bar\K} t_\K(\gamma) = \chi(\bar\K).
\end{equation}
\end{proposition}

\begin{proof}
We start by giving $\bar\K$ the Riemannian metric that corresponds to pulling back the flat Riemannian metric of $\bR^2$ to $\bar\K$ by the map $(\p\circ \e)^{-1}$.  Then the Gaussian curvature on $\bar\K$ is ${\bf K} \equiv 0$.  Moreover, for any smooth arc, $\alpha \subset \partial \bar\K$, the geodesic curvature, $k_g(x), \ x \in \alpha$, corresponds to the signed curvature, $k_s( \p\circ \e (x))$.  That is, for the smooth arc $\p\circ \e (\alpha) \,(\subset \bR^2)$, the signed curvature $k_s |_{ \p\circ \e (x)}$ is equal to the classical Frenet-Serret curvature times $\pm 1$---it is $+1$ if the Frenet-Serret normal vector equals the inward pointing normal to $\p\circ \e (\bar\K)$ at $ \p\circ \e (x)$, and $-1$ otherwise.

With this convention it is a result of classical differential geometry that $$\int_\alpha k_g ds = \int_{\p\circ \e (\alpha)} k_s ds = 2\pi \cdot t_p(\p\circ\e(\alpha)) .$$
A curve $\gamma\in\partial\bar{K}$ is piecewise smooth with this metric. It fails to be smooth at the corners of $\gamma$, and at such a corner $x$, the external angle is $\theta(\p\circ\e(x))=\angle_K(x)$. It follows for a boundary curve with corners, $\gamma \subset \partial \bar\K$, with smooth arcs $\{\alpha_i\}_{i \in \mathscr{I}}$ and corners $\{ x_i \}_{i \in \mathscr{I}}$,
\begin{equation}
    \begin{split}
        \int_\gamma k_g \ ds + \sum_{i \in \mathscr{I}} \theta(x_i) & = \sum_{i \in \mathscr{I}} \left[ \int_{\p\circ \e (\alpha_i)} k_s ds + \theta(\p\circ \e (x_i)) \right]\\
      \   & = \sum_{i \in \mathscr{I}} [ 2 \pi\cdot t_p(\p\circ\e(\alpha_i)) + \angle_K(x_i) ] \\
      \  & = 2\pi \sum_{i \in \mathscr{I}} \left[t_p(\p\circ\e(\alpha_i)) + \frac{\angle_K(x_i)}{2\pi} \right] \\
      \ & = 2 \pi \cdot t_\K(\gamma).
    \end{split}
\end{equation}

Then setting ${\bf K} \equiv 0 $, it follows from Equ.~\ref{equation: Gauss-Bonnet} that
\begin{equation}
    \begin{split}
        \int_{\partial \bar\K} k_g \ ds + \sum_{i \in \mathscr{I}} \theta(x_i) & =  2 \pi \sum_{\gamma \in \partial \bar\K} t_\K(\gamma) \\
        \ & = 2 \pi \cdot \chi(\bar\K).
    \end{split}
\end{equation}
The equality of our proposition then follows.
\end{proof}

\begin{remark}
\label{remark: GB for barK in terms of hatalpha}
Following the alternative method of computing $t_\K(\gamma)$ described in \S\ref{subsection: weighting curves}, can restate Equ.~\ref{equation: GB for barK}
as
\begin{equation}
    \label{equation: GB for barK in terms of hatalpha}
\sum_{\gamma \in \partial \bar\K} t(\p\circ\e(\hat{\gamma})) = \chi(\bar\K).
\end{equation}
\end{remark}

We are now situated to establish Equ.~\ref{Eq: Gauss-Bonnet 1} mentioned in \S\ref{Section:Introduction}.

\begin{theorem}
\label{Theorem: eulerturning}
Let $\e : S_g \into \bR^2 \times \bR$ be a regular embedding with crease set $\cC$.
Let $\{ \K_1, \ldots , \K_l \} \subset S_g$ be the collection of all the connected subsurface components of $S_g \setminus \cC$.
Then

\begin{equation}
    \label{equation: GB for total surface}
    \sum_{1\le i \le l}\sum_{\gamma\in\partial\bar{K}_i} t_{K_i}(\gamma) = \chi(S_g).
\end{equation}
Moreover, when every component of $\cC$ is without corners then Equ.~\ref{equation: GB for total surface} implies the following equality.
\begin{equation}
    \label{equation: GB for total surface with no corners}
   2 \sum_{\gamma \in \cC} t(\gamma) = \chi(S_g).
\end{equation}
In particular, when $g=0$ and $\cC$ has no corners, twice the sum of all the turning numbers of s.c.c.'s of $\cC$ is equal to $2$.
\end{theorem}

\begin{proof}

From classical topology it is a well known fact that the Euler characteristic is additive in the following sense.  Let $F_1$ and $F_2$ be two orientable compact surfaces with boundary.  Let $A \subset \partial F_1$ and $B \subset \partial F_2$ be two collections of boundary components such that $\abs{A} = \abs{B}$.  For any homeomorphism, $f : A \into B$, consider the surface, $F_1 \cup_f F_2$, obtained by gluing $F_1$ to $F_2$ via $f$.  Then $\chi(F_1 \cup_f F_2) = \chi(F_1) + \chi(F_2)$.

Applying this fact to $S_g$, we have the following sequence of equalities:

\begin{equation}
    \begin{split}
        \chi(S_g) & = \chi(\bar\K_1) + \cdots + \chi(\bar\K_l) \\
        \ &= \sum_{\gamma \in \partial \bar\K_1} t_{\K_1}(\gamma) +
     \cdots + \sum_{\gamma \in \partial \bar\K_l} t_{\K_l}(\gamma) \\
     \ & = \sum_{1\le i \le l}\sum_{\gamma\in\partial\bar{K}_i} t_{K_i}(\gamma) = \chi(S_g).
    \end{split}
\end{equation}
Our first equality then follows.

The claimed equality for the situation where there are no corners follows from the observation that each $\gamma\in\cC$ is a boundary curve of two $K_i$'s, and $\gamma$ has the same Gauss-Bonnet weight for each.
\end{proof}

We include the special case of the sphere in our statement since we will be restricting to that case alone for the remainder of the paper.
 
\subsection{Proof of Theorem~\ref{Theorem: weighting uniqueness}.}
\label{subsection: proof of unique weights}
 
The claim of the theorem can be restated as follows. Given a pair $(S^2, \cC)$ that can be geometrically realized by some regular embedding, there is a unique turning number function ${\bf T}: \cC \into \bZ^2$, that is,  ${\bf T}$ is independent of embedding.
 
We have three key observations that will be combined to establish this result.  The first observation is that every s.c.c.\ on $S^2$ is a separating curve. 
 
The second observation follows from the first, we can determine one of the Gauss-Bonnet weights of a component, $\gamma \in \cC$, that is innermost on $S^2$.  That is, there will always exist a connected subsurface component with boundary, $\bar\K \subset S^2$, that is homeomorphic to a $2$-disk.  With $\gamma = \partial \bar\K$ we will have
 $$+1 = \chi(\bar\K) = t_\K(\gamma),$$
 and thus one of $\gamma {\rm 's}$ Gauss-Bonnet weights.

The third observation is that if we can determine one of the two Gauss-Bonnet weights for a s.c.c., $\gamma \in \cC$, then we can determine the other.  To see this, let $\bar\K, \bar\K^\prime$ be two connected subsurface components with boundary that share a common boundary curve.  That is, we have a curve with corners $\bar\K \cap \bar\K^\prime = \gamma \in \cC$.  
We observe that if we know the value of $t_\K$, we can determine the value of $t_{\K^\prime}$, since a $\frac{\pi}{2}$ corner in $\bar\K$ (respectively, $\frac{3 \pi}{2}$ corner) becomes a $\frac{3 \pi}{2}$ corner (respectively, $\frac{\pi}{2}$ corner) in $\bar\K^\prime$, while the partial turning numbers of smooth arcs remain the same.

Now combining these three observations, suppose there is a component of $\cC$ for which ${\bf T}$ is undetermined.  Again, let $\{ \bar\K_1, \ldots , \bar\K_l \}$ be the collection of all connected subsurface components coming from $S^2 \setminus \cC$.  Suppose there is a subsurface, $\bar\K_i$, such that $t_{\K_i}$ is undetermined for one of its boundary components.  If such a component exists then we can choose $\bar\K_i$ such that it has only one boundary component with $t_{\K_i}$ undetermined.  But, using Equ.~\ref{equation: GB for barK}, we can solve for the ``undetermined'' Gauss-Bonnet weight associated to $\bar\K_i$---we will know $\chi(\bar\K_i)$ and all but one of the associated Gauss-Bonnet weights.

The theorem then follows.

\subsection{Proof of Theorem~\ref{Theorem: configuration finiteness}.}
 
Once the number of components and corners of $\cC$ is bounded by $n_t$ and $n_c$, respectively, the finiteness claim follows from topological bookkeeping arguments.

First, given a collection of disjoint s.c.c.'s, $\cC \subset S^2$, we can naturally associate to it a graph.  Each vertex of the graph will correspond to a connected component of $S^2 \setminus \cC$ and the edges of the graph will correspond to the components of $\cC$---two vertices associated to two subsurface components of $S^2 \setminus \cC$ share a common edge if the curve of $\cC$ associated to that edge is a common boundary curve of the two subsurfaces.  Since every s.c.c.\ in $S^2$ is separating the graph associated with a pair $(S^2 , \cC)$ will be a tree.  It is well known that the number of distinct tree-like graphs having at most $n_t$ edges is a bounded set. (See \cite{[C]}.)

To take into account corners of $\cC$, note that given a collection of disjoint s.c.c.'s, $\cC \subset S^2$, there is only a finite number of ways to distribute no more than $n_c$ marked points along the curves of $\cC$.



\section{Constructing embeddings of $S^2$}
\label{section: constructing embeddings}

In this section we revisit question (2) of \S\ref{Section:Introduction}: \emph{given a collection of s.c.c.'s $\cC \subset S^2$, when is $\cC$ realized as the crease set of a regular embedding into $\bR^2 \times \bR$}?  We will answer this question in the simplified case where the crease set has no corners.  From Theorem~\ref{Theorem: weighting uniqueness} we know that if such an embedding exists, the turning number function, ${\bf T}$, decorating the pair $(S^2,\cC)$ is uniquely determined.  Since such a ${\bf T}$ is totally determined by the topological information of the components of $S^2 \setminus \cC$, one can readily decide whether such a function exists.

Specifically, let $(S^2 , \cC , {\bf T})$ be a triple satisfying the following ($\ast$) conditions:
\begin{itemize}
    \item[($\ast$-i)] $\cC \subset S^2$ is a collection of pairwise disjoint smooth s.c.c.'s.
    \item[($\ast$-ii)] ${\bf T} : \cC \into \mathbb{Z}$ such that for each connected surface $\K \subset S^2 \setminus \cC$, the associated surface with boundary, $\bar\K$, satisfies the relationship $$\sum_{\gamma \in \partial \bar\K} {\bf T}(\gamma) = \chi(\bar\K).$$
\end{itemize}
With the above statement of ($\ast$) in mind, we now give an equivalent restatement of Theorem~\ref{Theorem: embedding existence}, which we will proceed to prove in the remainder of \S\ref{section: constructing embeddings}.

\begin{theorem}
\label{theorem: realization of crease set}
Let $(S^2, \cC , {\bf T})$ be a triple satisfying {\rm($\ast$)}.  Then there exists a regular embedding, $\e : S^2 \into \bR^2 \times \bR$, such that $\cC$ corresponds to the crease set of $\e$ and ${\bf T}$ corresponds to the turning number function associated with $\e$.  In particular, the crease set has no corners.
\end{theorem}

Our constructive argument has two steps.  First, we develop an understanding of representing embeddings of flat planar surfaces in $\bR^2 \times \bR$.  Second, we show how these representations can be ``stacked'' together to construct the needed embedding of $S^2$.

\subsection{Constructing planar surfaces $\bar\K$}
\label{subsection: planar surfaces}

We start with the following definition.

\begin{definition}
An \emph{$n$-turning weight}, $W$, is an $(n+1)$-tuple of odd integers, $W_s =(t_0 , \ldots , t_n )$, such that $$\sum_{i=0}^n t_i = 1-n.$$
The \emph{standard} $n$-turning weight set is 
$$W_s = (1, \underbrace{-1, -1, \ldots , -1}_{n}).$$
A \emph{basis} for an $n$-turning weight, $W = \{t_0 , \ldots , t_n \}$, is a collection of integers, $\{k_0, \ldots, k_n \}$, such that:
\begin{itemize}
    \item[1.] $t_0 = 1 + 2 k_0$,
    \item[2.] $t_i = -1 + 2 k_i, \ {\rm for} \ 1\leq i \leq n.$
    \item[3.] Necessarily, $$\sum_{i=0}^n k_i = 0.$$
\end{itemize}
\end{definition}

Let $\bar\K_{1-n}$ be a smooth compact planar surface having Euler characteristic equal to $1-n$.
Let $W = (t_0 , \ldots , t_n)$ be an $n$-turning weight.
We say an embedding $\e : \bar\K_{1-n} \into \bR^2 \times \bR$ is a \emph{geometric realization of $W$} if:
\begin{itemize}
    \item[i.] every point of ${\rm int}\,\bar\K_{1-n}$ is a regular point.  Thus, the pull-back of $\p \circ \e (\bar\K_{1-n})$ gives a flat structure to $\bar\K_{1-n}$.
    \item[ii.] There is an enumeration of the boundary components, $\partial \bar\K_{1-n} = \{ \gamma_0 , \gamma_1, \ldots , \gamma_n \}$, such that we have $t(\gamma_i) = t_i$ for $ 0 \leq i \leq n $.  That is, the turning numbers of the components of $\partial \bar\K_{1-n}$ realize $W$.
\end{itemize}

Consider a closed disk in $\bR^2 \times \{ {\rm pt.} \}$.  Delete from it $n$ open subdisks with disjoint closures. The resulting ``disk with $n$ holes'' will be an embedding of a $\bar\K_{1-n}$ that is a geometric realization of the standard $n$-turning weight, $W_s$.
We will use the notation $\e_s : \bar\K_{1-n} \into \bR^2 \times \{ {\rm pt.} \}$ for this \emph{standard embedding} of a planar surface of Euler characteristic $1-n$.  Observe that we can easily require that the Frenet-Serret curvature be nowhere zero on $\partial \e_s(\bar\K_{1-n})$.

We now give a procedure for taking the standard embedding, $\e_s(\bar\K_{1-n})$, and altering it so that we obtain a geometric realization of any chosen $n$-turning weight, $W$.  Let $\{ \gamma_0 , \ldots , \gamma_n \}$ be a fixed enumerated labeling of the boundary components $\partial \bar\K$.  Let $\{ k_o , \ldots, k_n \}$ be a basis for $W$.  For each boundary component, $\gamma_i$, we chose $|k_i|$ distinct points, $p_i^1, \ldots , p_i^{|k_i|} \subset \gamma_i$, and label them ``$+2$'' (respectively, ``$-2$'') when $k_i > 0$ (respectively, when $k_i < 0$).  (To keep down the clutter in our figures, we will use a dot, $ \bullet $, for $+2$ label and a circle, $\circ$, for $-2$ label.)  Note that all the labels on a particular boundary component are the same.  See the initial illustration in the sequence in Fig.~\ref{Fig: turning twist}.

Next we choose a collection of properly embedded pairwise disjoint arcs, $\{a_j\} \subset \bar\K_{1-n}$ such that each arc has one endpoint labeled $-2$ and the other endpoint labeled $+2$.  Thus, each arc has a natural orientation from its negative to positive boundary endpoints.  We will refer to such properly embedded oriented arcs as \emph{twisting arcs of $\bar\K$ associated with $W$}.

Three observations are warranted.  First, we must have $$l = \frac{1}{2} \sum_{i=0}^n |k_i|$$ many arcs. Second, since the labels on a particular boundary component are homogeneous, the two endpoints of any twisting arc are on different boundary components.  Finally, there is a great deal of choice in the embedding of the twisting arcs.  Although it is not necessary, for convenience of description we will assume that these arcs are linear intervals.  See the starting illustration in the sequence of Fig.~\ref{Fig: turning twist}.

\begin{figure}[ht]
\centering{\includegraphics[width=\linewidth]{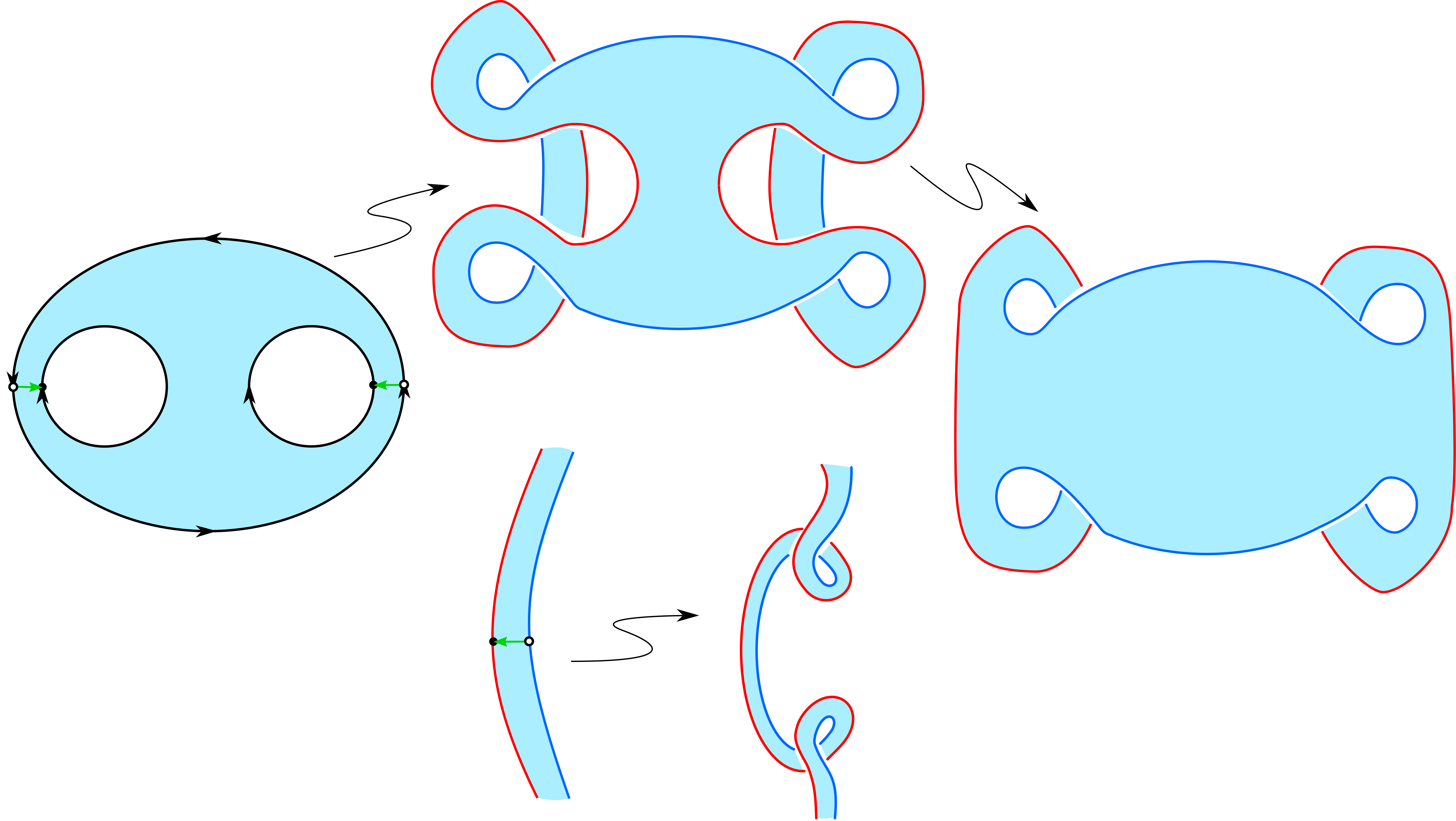}}
\caption{The top sequence illustrates how the twisting operation in a neighborhood of the two oriented arcs takes a planar surface that realizes the standard $2$-turning weight set, $(1, -1, -1)$, to a realization of the $2$-turning weight set, $(-3, 1 , 1)$.  This top sequence continues with an isotopy of the resultant boundary so that $k_s \neq 0$ at every point on the {\red red boundary curves}.  The bottom sequence illustrates our twisting operation in a neighborhood of an oriented arc.}
\label{Fig: turning twist}
\end{figure}

Next, given a choice of of twisting arcs, $\{a_j\}_{1 \leq j \leq l} \subset \bar\K$, in a standard embedded planar surface  we can perform the \emph{twisting operation} illustrated in Fig.~\ref{Fig: turning twist} in a neighborhood of each $a_j$.  The resulting planar surface embedding will be a geometric realization of the associated $n$-turning weight set, $W$.  The top sequence in Fig.~\ref{Fig: turning twist} illustrates this procedure for the taking the standard embedding of a disk-minus-two-holes and obtaining an embedding associated with the $2$-weight set, $(-3, 1 ,1)$.

The reader should observe that the basis used for the $(-3, 1,1)$ geometric realization is $(k_0 ,k_1 , k_2) = (-4, 1, 1)$.  However, one could also utilize the basis $(k_0 ,k_1 , k_2) = (0, -1, 1)$.  Utilizing this ``change of basis'' one will have a single properly embedded arc whose endpoints are on the two $-1$ boundaries of starting planar surface in the top-sequence of Fig.~\ref{Fig: turning twist}.  The $2$-turning weight then becomes $(1,-3,1)$ or $(1,1,-3)$, depending on the ordering of the boundary curves.

The above discussion is captured by the following proposition.

\begin{proposition}
\label{proposition: geometric realization of planar surfaces}
For any basis, $\{k_0, \ldots , k_n \} $, associated with a $n$-turning weight set, $W$, there exists an embedding, $\e : \bar\K_{1-n} \into \bR^2 \times \bR$, that is a geometric realization of $W$ in a manner that respects the basis.  That is, the turning number of the $\gamma_i$ boundary curve of $\bar\K_{1-n}$ is $1+ 2k_0$ when $i=0$ and $-1 + 2k_i$ when $1 \leq i \leq n$.
\end{proposition}

A number of remarks follow. 

\begin{remark}
\label{remark: +1 turning}
The reader should observe that the only possible way a  boundary curve having $+1$ turning number can result from our twisting operation is if the associated basis value is either $k_0 = 0$ or $k_i = 1, \  1 \leq i \leq n$. 
\end{remark}

\begin{remark}
\label{remark: normal position for c_0}
As illustrated in the initial planar surface of the top sequence of \mbox{Fig.~\ref{Fig: turning twist},} we can readily assume that the geometric realization of a standard $n$-turning weighted set has its boundary curves in a ``normal position''.  That is, the curvature $k_s$ is not zero at any point on the boundary.
The reader should observe that the twisting operation preserves this normal position---$k_s \neq 0$---on the $\gamma_0$ boundary curve when the associated basis value $k_0 \geq 0$.
\end{remark}

\begin{remark}
\label{remark: normal position for c_i}
Continuing an analysis of the behavior of the boundary curvature with respect to the twisting operation, for $2 \leq j \leq n$, if we have a basis value, $k_j \leq 0$, then $k_s \neq 0$ on the associated boundary curve after performing the twisting operation. 
\end{remark}

\begin{figure}[ht]
\centering
\includegraphics[width=.7\linewidth]{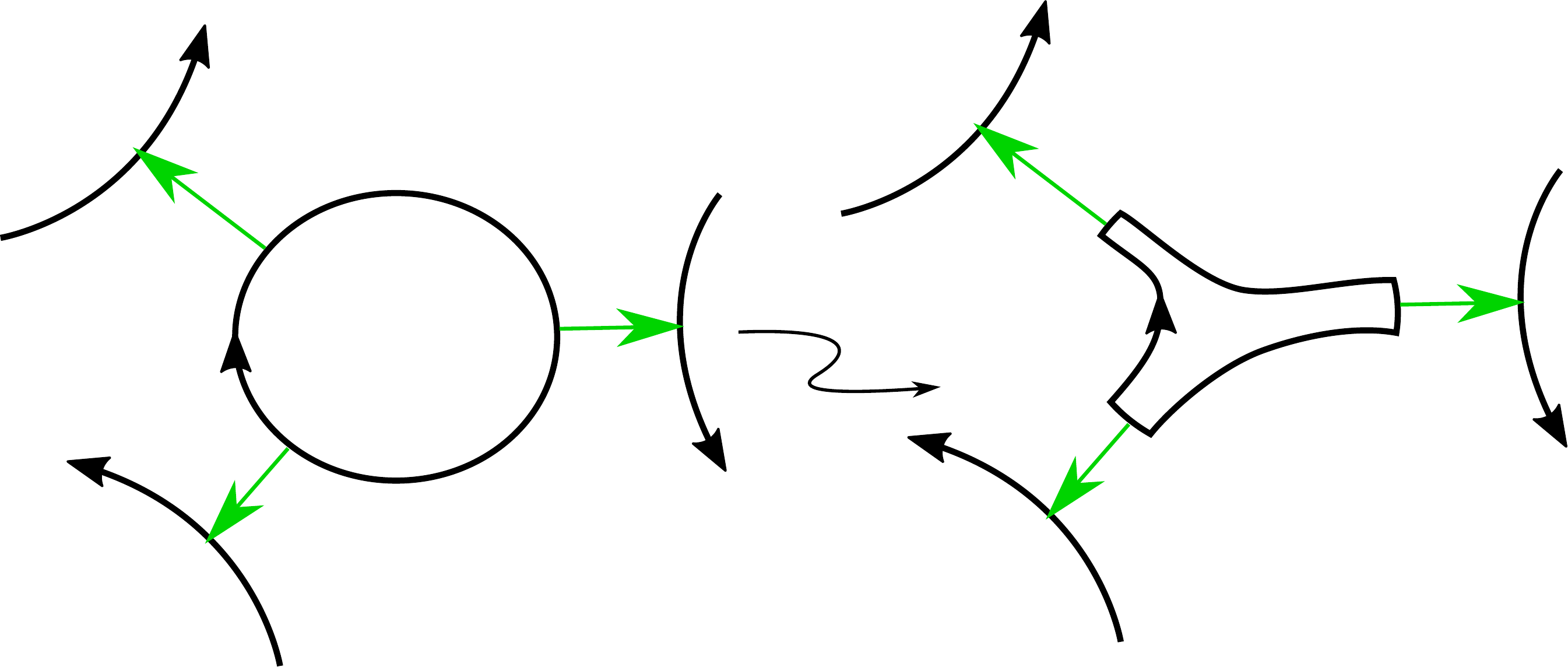}
\caption{A depiction of the isotopy on a boundary curve of a planar component so it will be in normal form once the twist operations are performed. The isotopy fixes a neighborhood of the {\green green} oriented arcs where the twists will be done. In this case, the inner curve has $k_j=-3$, but this generalizes to any $k_j\leq -2$.}
\label{Fig: normal for k_j < -3}
\end{figure}

\begin{remark}
\label{remark: isotopy of c_i to normal position}
For some $ 1 \leq j \leq n$, if we have a basis value $k_j \leq -2$ then we can perform an isotopy that achieves $k_s \neq 0$ on the associated boundary curve after the twisting operation is applied.  (See Fig.~\ref{Fig: normal for k_j < -3}.)
\end{remark}

\begin{figure}[ht]
\centering
\includegraphics[width=.8\linewidth]{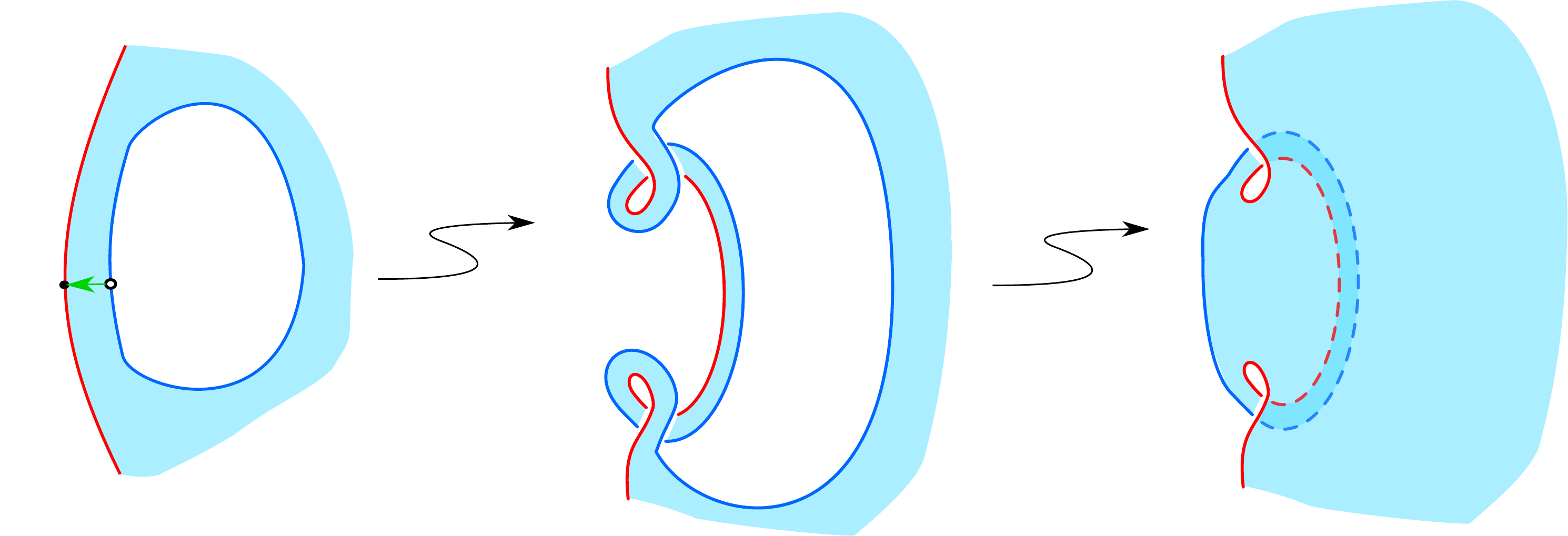}
\caption{The twist operation performed with a $k_j=-1$ curve, followed by an isotopy sliding the top portion of $\bar\K$ to the left, so that that $k_s\neq 0$ along the twisted {\blue blue} boundary curve.}
\label{Fig: k_i= -1 and -2}
\end{figure}

\begin{remark}
\label{remark: cases when k_i = -1 and -2}
To deal with final case for achieving a normal form for the resulting embedding of $\bar\K$, Fig.~\ref{Fig: k_i= -1 and -2} illustrates how the boundary curve initially having $-1$ turning number can achieve a $+1$ turning number after a single twist operation.  This is the case where $k_j = -1$ for some $1 \leq j \leq n$.  The embedding depicted in the middle illustration of the sequence has this $+1$ turning number boundary curve not in normal form since there will be two points where $k_s = 0$.  The right-hand illustration shows this same boundary curve after an isotopy that places it in a position where $k_s >0$.
\end{remark}

\subsection{Constructing embeddings and proof of Theorem~\ref{Theorem: embedding existence}.}
\label{Subsection: proof of embedding existence.}

\subsubsection{Labeled tree graphs.}
\label{subsubsection: LTG}
We assume that we are given, $(S^2 , \cC)$, a $2$-sphere/crease set pair.  That is, $\cC \subset S^2$ is a collection of pairwise disjoint smooth s.c.c.'s which have a Gauss-Bonnet weighting.  We can associate to this pair a \emph{labeled tree graph (LTG)}, $\cG$, as follows:
\begin{itemize}
    \item[1.]  For each connected planar component, $\K \subset S^2 \setminus \cC$, there is a corresponding vertex $v_\K \subset \cG$, with label $\chi(\bar\K)$.
    \item[2.]  Each edge, $e_\gamma \subset \cG$, corresponds to a curve of $\gamma \in \cC$ with the label $t(\gamma)$.
    \item[3.]  An edge, $e_\gamma \subset \cG$, is incident to vertices, $v_\K, v_{\K^\prime} \subset \cG$, if $\gamma = \partial \K \cap \partial \K^\prime$.
\end{itemize}
A LTG being a tree comes from the fact that all s.c.c.'s on a $2$-sphere separate.

\begin{example}[Part 1.]
\label{example: running example}
We offer a running example, $(S^2 , \cC) = (S^2 , \{\gamma_1, \ldots, \gamma_{13} \})$, with: $t(\gamma_i) = 1$, $i \in \{1,2,3,4,10,11,12,13\}$; $t(\gamma_5) = t(\gamma_6) = t(\gamma_8) = t(\gamma_9)= -3$; and $t(\gamma_7)= +5$.  Fig.~\ref{Fig: example 1} illustrates $\cC \subset S^2$ and Fig.~\ref{Fig: example 1 graph} depicts the associated LTG, $\cG$.
\end{example}

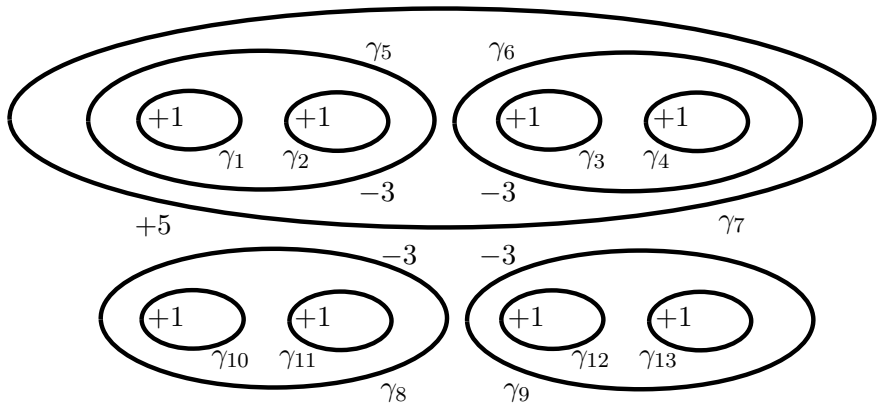
\begin{figure}[ht]

\labellist
\small

\pinlabel $\gamma_1$ [tl] at 140 175.73
\pinlabel $\gamma_2$ [tr] at 210 175.73
\pinlabel $\gamma_3$ [tl] at 385 175.73
\pinlabel $\gamma_4$ [tr] at 455 175.73
\pinlabel $\gamma_5$ [bl] at 240 229.58
\pinlabel $\gamma_6$ [br] at 350 229.58
\pinlabel $\gamma_7$ [t] at 493.18 133.22
\pinlabel $\gamma_8$ [tl] at 250 17.01
\pinlabel $\gamma_9$ [tr] at 360 17.01
\pinlabel $\gamma_{10}$ [tl] at 135 40
\pinlabel $\gamma_{11}$ [tr] at 215 40
\pinlabel $\gamma_{12}$ [tl] at 380 40
\pinlabel $\gamma_{13}$ [tr] at 460 40
\pinlabel $+1$ [l] at 90 194.15
\pinlabel $+1$ [l] at 190 194.15
\pinlabel $+1$ [l] at 333 194.15
\pinlabel $+1$ [l] at 433.66 194.15
\pinlabel $+1$ [l] at 90 58
\pinlabel $+1$ [l] at 190 58
\pinlabel $+1$ [l] at 335 58
\pinlabel $+1$ [l] at 435.66 58
\pinlabel $-3$ [tl] at 235 154.47
\pinlabel $-3$ [tr] at 350 154.47
\pinlabel $-3$ [bl] at 250 89.28
\pinlabel $-3$ [br] at 350 89.28
\pinlabel $+5$ [t] at 99.20 133.22

\endlabellist

\centering
\includegraphics[width=.7\linewidth]{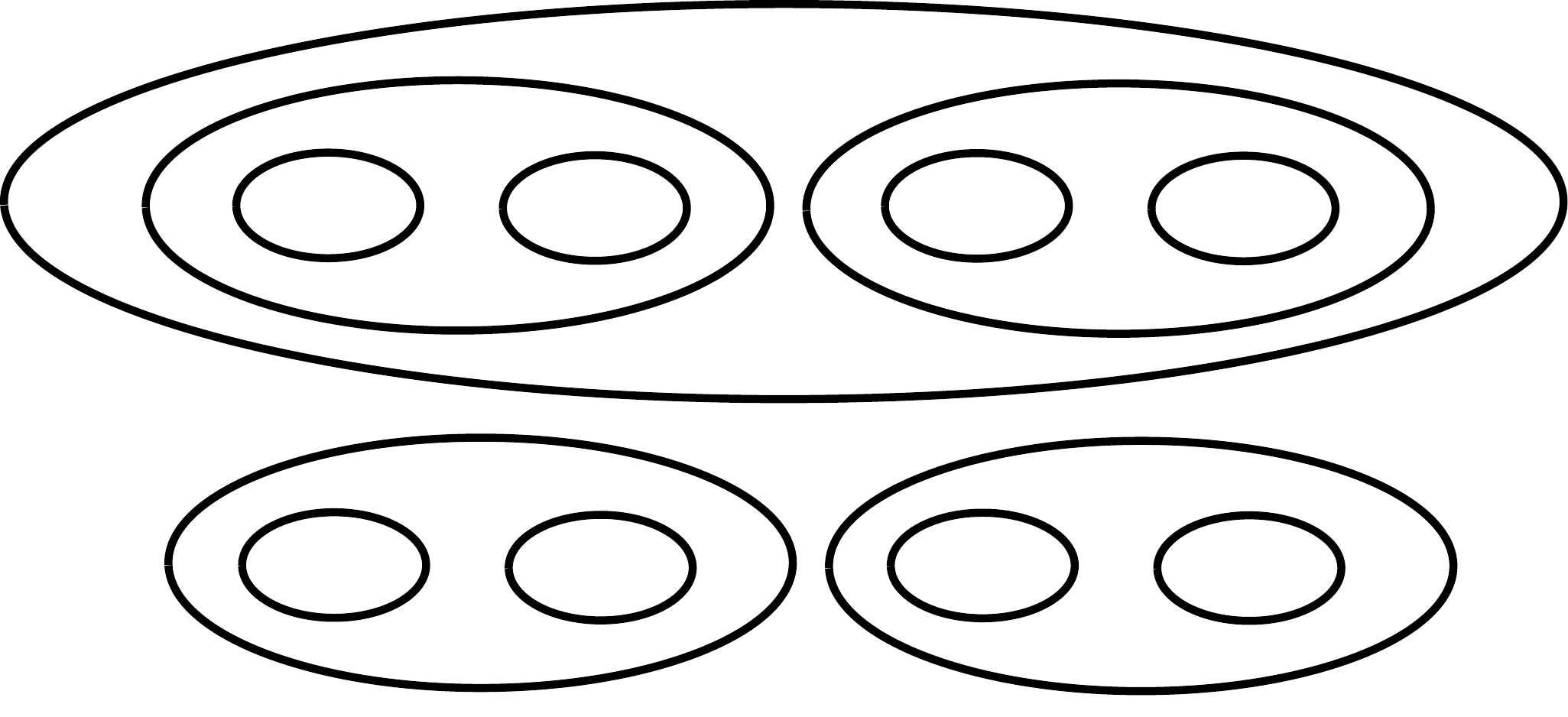}
\caption{A collection of curves on $S^2$ corresponding to $\cC$.  Each curve in $\cC$ is labeled $\gamma_i$ for $1\leq i \leq 13$.  Additionally, each curve is labeled with its determined Gauss-Bonnet turning number weight.}
\label{Fig: example 1}
\end{figure}

\begin{figure}[ht]

\labellist
\tiny

\pinlabel $+1$ [c] at 28.34 928.26
\pinlabel $+1$ [c] at 131.80 928.26
\pinlabel $+1$ [c] at 473.34 928.26
\pinlabel $+1$ [c] at 576.80 928.26
\pinlabel $+1$ [tr] at 75.11 885.74
\pinlabel $+1$ [l] at 144.55 885.74
\pinlabel $-1$ [c] at 144.55 838.98
\pinlabel $-1$ [c] at 589.55 838.98
\pinlabel $+1$ [c] at 259.35 780.87
\pinlabel $+1$ [c] at 704.34 780.87
\pinlabel $+1$ [c] at 359.97 734.10
\pinlabel $+1$ [c] at 807.80 734.10
\pinlabel $+1$ [r] at 243.76 698.67
\pinlabel $+1$ [tl] at 314.62 684.50
\pinlabel $\gamma_3$ [r] at 688.75 698.67
\pinlabel $-1$ [c] at 246.59 641.99
\pinlabel $-1$ [c] at 691.59 641.99
\pinlabel $-3$ [r] at 171.48 573.96
\pinlabel $-3$ [l] at 232.42 573.96
\pinlabel $\gamma_5$ [r] at 616.48 573.96
\pinlabel $-1$ [c] at 201.24 528.61
\pinlabel $-1$ [c] at 646.24 528.61
\pinlabel $+5$ [l] at 205.49 483.26
\pinlabel $-1$ [c] at 201.24 435.08
\pinlabel $-1$ [c] at 646.24 435.08
\pinlabel $-3$ [r] at 155.89 306.11
\pinlabel $-3$ [l] at 226.75 389.73
\pinlabel $\gamma_8$ [r] at 600.89 306.11
\pinlabel $-1$ [c] at 250.84 335.87
\pinlabel $-1$ [c] at 695.84 335.87
\pinlabel $+1$ [r] at 249.43 257.93
\pinlabel $+1$ [bl] at 311.78 283.44
\pinlabel $\gamma_{12}$ [r] at 694.42 257.93
\pinlabel $+1$ [c] at 358.55 232.42
\pinlabel $+1$ [c] at 803.55 232.42
\pinlabel $+1$ [c] at 255.09 179.98
\pinlabel $+1$ [c] at 700.09 179.98
\pinlabel $-1$ [c] at 144.55 133.22
\pinlabel $-1$ [c] at 589.55 133.22
\pinlabel $+1$ [br] at 73.69 86.45
\pinlabel $+1$ [l] at 143.14 80.78
\pinlabel $+1$ [c] at 25.51 28.34
\pinlabel $+1$ [c] at 127.55 28.34
\pinlabel $+1$ [c] at 470.51 28.34
\pinlabel $+1$ [c] at 572.54 28.34
\pinlabel 1 [c] at 430 25.51
\pinlabel 2 [c] at 430 134.63
\pinlabel 3 [c] at 430 178.57
\pinlabel 4 [c] at 430 231.00
\pinlabel 5 [c] at 430 335.87
\pinlabel 6 [c] at 430 437.91
\pinlabel 7 [c] at 430 528.61
\pinlabel 8 [c] at 430 639.15
\pinlabel 9 [c] at 430 729.85
\pinlabel 10 [c] at 430 782.29
\pinlabel 11 [c] at 430 838.98
\pinlabel 12 [c] at 430 933.93

\endlabellist

\centering
\includegraphics[width=.545\linewidth]{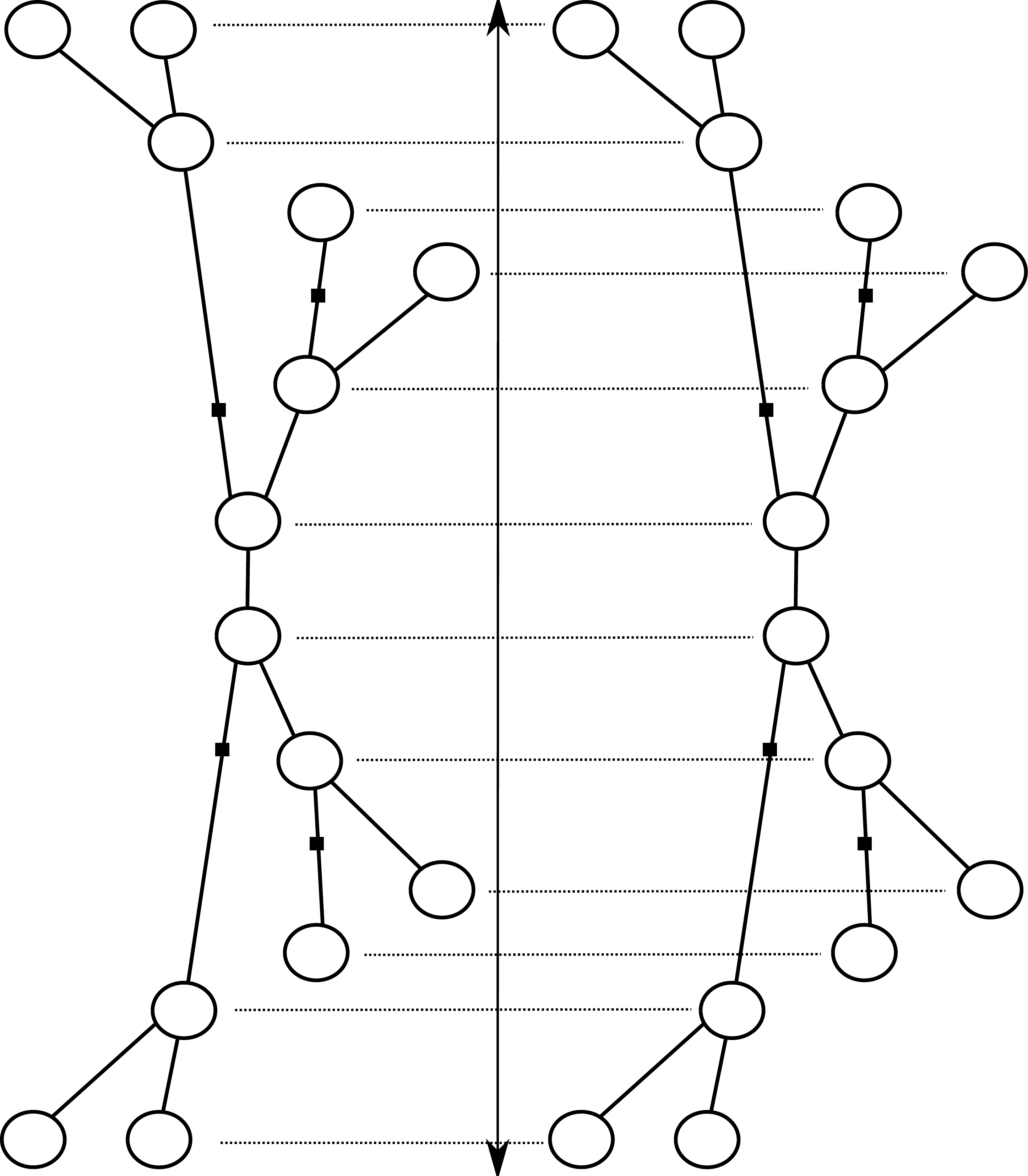}
\caption{The graph $\cG$ associated with Example~\ref{example: running example}.  Labels in the vertices (circles) correspond to Euler characteristic of planar subsurface.  Labels of edges correspond to the turning number of the adjoining curve of $\cC$. The height axis between the two graph depictions gives the height/order for the stacking procedure that is illustrated in Fig.~\ref{Fig: stack example}.}
\label{Fig: example 1 graph}
\end{figure}

For a given $(S^2, \cC)$ configuration we can embed its associated LTG, $\cG \into \bR \times \bR \subset \bR^2 \times \bR$. We use the coordinates $((x,y),z) \in \bR^2 \times \bR$ and $((x,y_0),z) \in \bR \times \bR$, where $y_0$ is an arbitrary, fixed value.  We denote the $z$-height of a vertex $v \subset \cG$ by $z_v \in \bR$, and the $z$-interval support of any edge $e_\gamma \subset \cG$, corresponding to $\gamma \in \cC$, by $I_\gamma \subset \bR$.  As a tree, any LTG can be embedded in $\bR \times \bR$ so as to satisfy the following ($\star$) conditions:
\begin{itemize}
    \item[($\star$)-0] For two distinct vertices, $v , v^\prime \subset \cG$, we have $z_v \neq z_{v^\prime}$.
    \item[($\star$)-1] Each edge is a straight line segment with slope being non-zero, i.e.\ $\left|\frac{dx}{dz}\right| >0$.
    \item[($\star$)-2] Each vertex, $v \in \cG$, having adjacent edges, $\{e_{\gamma_1} , \ldots ,e_{\gamma_p} \}$ for  $\gamma_i \in \cC, \ 1 \leq i \leq p$, satisfies one of the following statements: (a) $\min I_{\gamma_i} \geq z_v $ iff $t(\gamma_i) > 0$; or, (b) $\max I_{\gamma_i} \leq z_v $ iff $t(\gamma_i) > 0$.  (See Fig.~\ref{Fig: vertices of graph}.) 
\end{itemize}

Regarding condition ($\star$)-2, the reader should observe that if $v , v^\prime \subset \cG$ are two vertices adjacent to a common edge then one vertex satisfies ($\star$)-2a and the other satisfies ($\star$)-2b.  Overall, it is a elementary inductive argument on the number of vertices that any LTG has an planar embedding satisfying condition ($\star$).  The reader can check that the LTG associated with our running example in (Fig.~\ref{Fig: example 1 graph}) satisfies ($\star$).

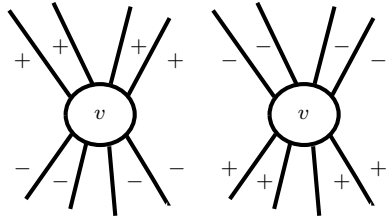
\begin{figure}[ht]

\labellist
\tiny

\pinlabel $+$ [r] at 27 115
\pinlabel $+$ [r] at 55 125
\pinlabel $+$ [r] at 113 125
\pinlabel $+$ [r] at 140 115
\pinlabel $-$ [r] at 27 35
\pinlabel $-$ [r] at 55 25
\pinlabel $-$ [r] at 110 25
\pinlabel $-$ [r] at 141 35
\pinlabel $v$ [c] at 72 75
\pinlabel $v$ [c] at 222.5 75
\pinlabel $-$ [r] at 180 115
\pinlabel $-$ [r] at 205 125
\pinlabel $-$ [r] at 263 125
\pinlabel $-$ [r] at 290 115
\pinlabel $+$ [r] at 180 35
\pinlabel $+$ [r] at 206 25
\pinlabel $+$ [r] at 263 25
\pinlabel $+$ [r] at 290 35

\endlabellist

\centering
\includegraphics[width=.3\linewidth]{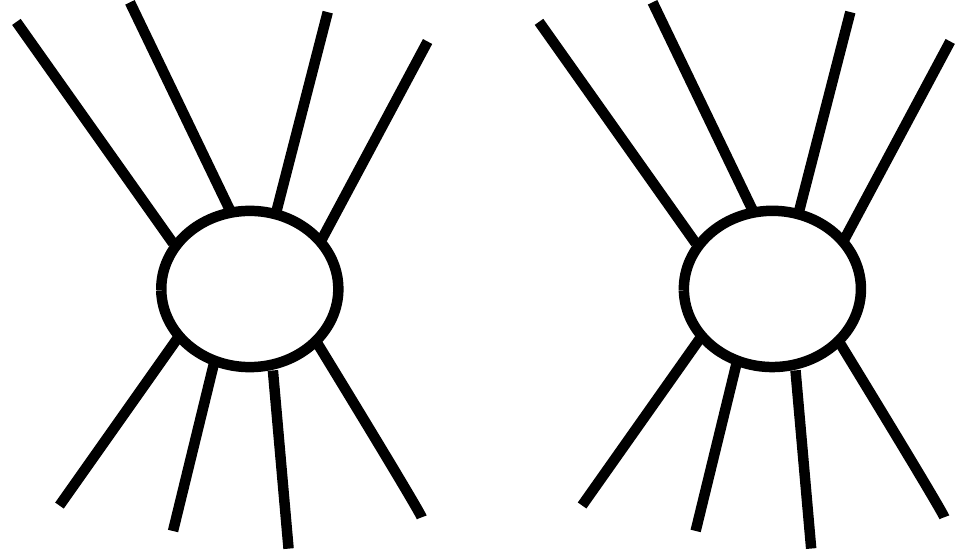}
\caption{The two possible height configuration of edges at a vertex.  The left (respectively, right) corresponds to ($\star$)-2a (respectively, ($\star$)-2b).}
\label{Fig: vertices of graph}
\end{figure}

\subsubsection{Consistent labeling of a LTG}
\label{subsubsection: consistent labeling}
As in the embedding construction establish in Proposition~\ref{proposition: geometric realization of planar surfaces}, we will initially need to designate a collection of twisting arcs in each $\bar\K$, $\chi(\bar\K) \leq 0$, that is associated with a given $(S^2, \cC)$ that has a Gauss-Bonnet weighting.  Such a designation of arcs is non-unique.  As with the case of an individual compact connected planar surface, deciding the placement of twisting arcs in $S^2$ starts with deciding which curves of $\cC$ will correspond to the $+1$ turning number of a standard $n$-turning weight set.  Specifically, for a $(S^2, \cC)$ configuration, let $C=\{\gamma_1, \ldots , \gamma_m\} \subset \cC$ be a subset collection such that for any $\bar\K, \ \chi(\bar\K) \leq 0,$ associated with the configuration, exactly one curve of $C$ lies in $\bar\K$.  For a subcollection $C$ satisfying this condition we say \emph{$C$ is a consistent labeling of $\cC$}.  Reframing this in terms of the LTG, a subcollection of edges, $E(C) \subset \cG$, is a consistent labeling if each vertex of $\cG$ that has a non-positive label is adjacent to exactly one edge in $E(C)$.

In our running example, the subcollection of $\cC$ with a $\blacksquare$ corresponds to a consistent labeling of $\cC$.  See Fig.~\ref{Fig: oriented arcs}.  Similarly, the corresponding edges of the LTG in our running example have a $\blacksquare$.  See Fig.~\ref{Fig: example 1 graph}.

The salient feature of such a labeling is that if $\gamma \in \cC$ is a crease curve that is a boundary curve of the two planar components, $\K, \K^\prime$, then the $\pm 2$-labeling on $\gamma$ when viewed as a boundary curve of $\K$ is the same as the $\pm 2$-labeling on $\gamma$ when viewed as a boundary curve of $\K^\prime$.

\begin{figure}[ht]

\labellist
\small

\pinlabel $\gamma_1$ [tl] at 140 175.73
\pinlabel $\gamma_2$ [tr] at 203 175.73
\pinlabel $\gamma_3$ [tl] at 387 172.73
\pinlabel $\gamma_4$ [tr] at 455 175.73
\pinlabel $\gamma_5$ [bl] at 240 229.58
\pinlabel $\gamma_6$ [br] at 350 229.58
\pinlabel $\gamma_7$ [t] at 493.18 133.22
\pinlabel $\gamma_8$ [tl] at 250 17.01
\pinlabel $\gamma_9$ [tr] at 360 17.01
\pinlabel $\gamma_{10}$ [tl] at 135 40
\pinlabel $\gamma_{11}$ [tr] at 215 40
\pinlabel $\gamma_{12}$ [tl] at 380 40
\pinlabel $\gamma_{13}$ [tr] at 460 40
\pinlabel $+1$ [l] at 90 194.15
\pinlabel $+1$ [l] at 190 194.15
\pinlabel $+1$ [l] at 333 194.15
\pinlabel $+1$ [l] at 433.66 194.15
\pinlabel $+1$ [l] at 90 58
\pinlabel $+1$ [l] at 190 58
\pinlabel $+1$ [l] at 335 58
\pinlabel $+1$ [l] at 435.66 58
\pinlabel $-3$ [tl] at 235 154.47
\pinlabel $-3$ [tr] at 350 154.47
\pinlabel $-3$ [bl] at 250 90
\pinlabel $-3$ [br] at 350 90
\pinlabel $+5$ [t] at 99.20 133.22

\endlabellist

\centering
\includegraphics[width=.8\linewidth]{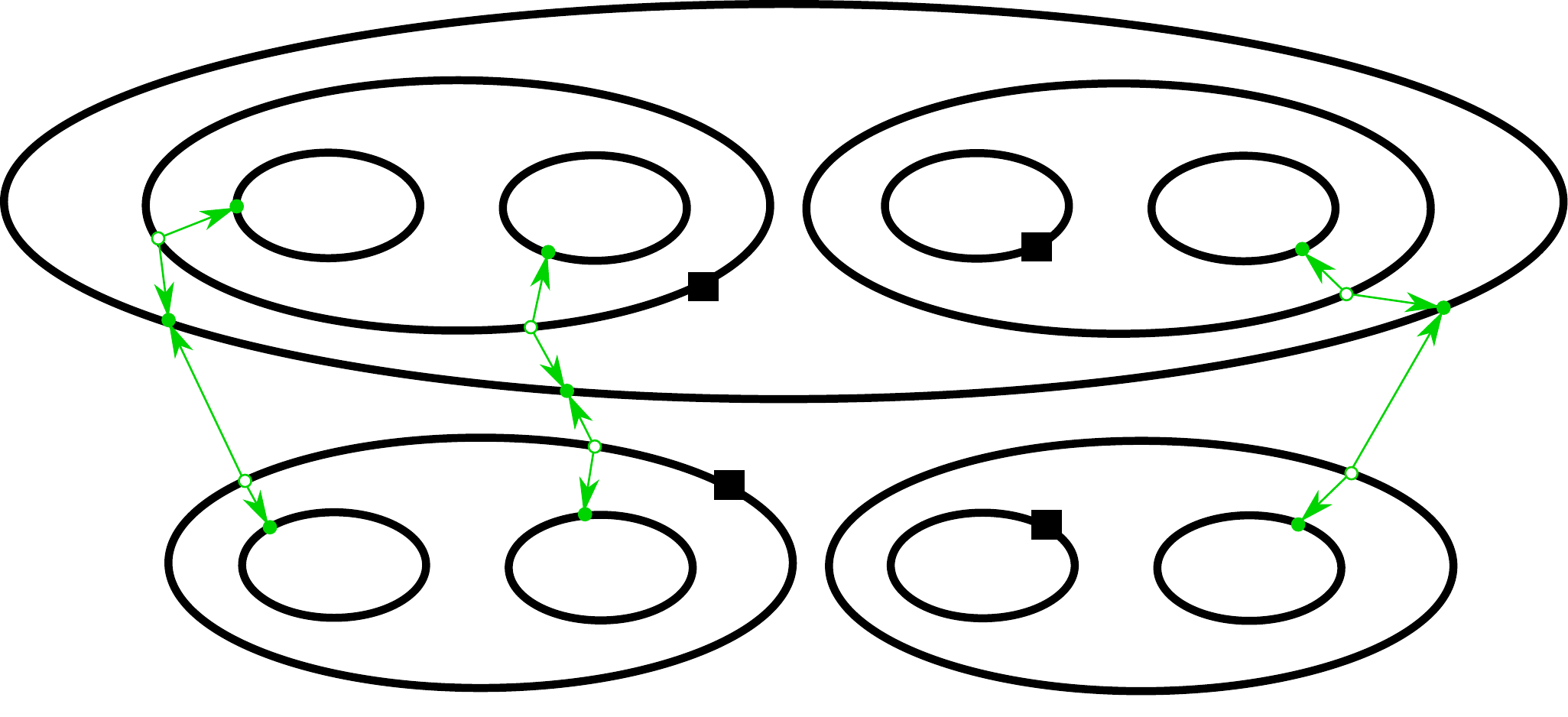}
\caption{The crease set, $\cC \subset S^2$, of the running example with the enhancements of black-squares and {\green green} edge-paths of oriented arcs.}
\label{Fig: oriented arcs}
\end{figure}
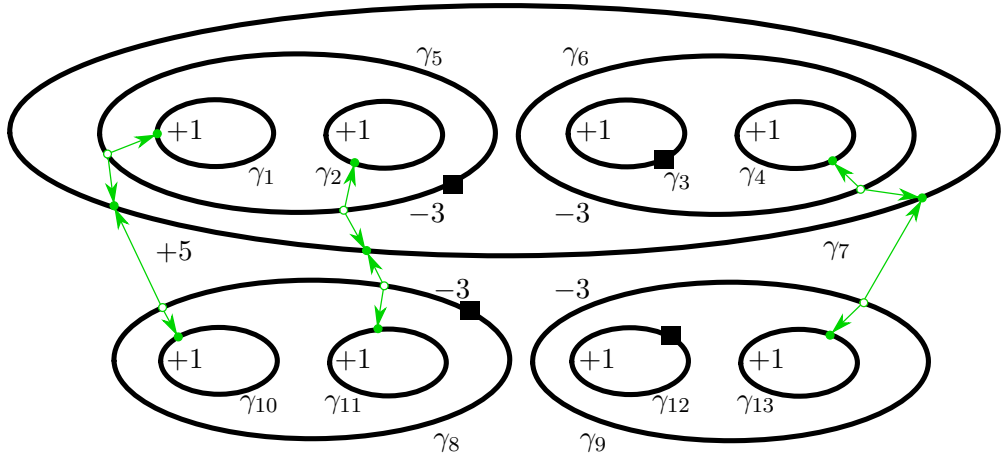


That label consistency is readily achievable follows from the {\em black-square tree decomposition} procedure that we now describe.  Let $v \in \cG$ be a vertex that is adjacent to a collection of edges, all but one of which carrying a ``$+1$'' label.  Note that the exceptional edge will have a negative integer as its label.  We make an initial choice of edge for placing a black-square, $\blacksquare$: either a ``$+1$'' edge receives the square or the exceptional edge receives it.  Next, we choose an ``alternating'' edge-path, ${\bf E}(v,v^\prime) \subset \cG$, between $v$ and another $v^\prime \in \cG$ that is also only adjacent to a collection of ``$+1$'' labeled edges and a single negative labeled edge.  The edge-path is alternating in the sense that, referring to Fig.~\ref{Fig: vertices of graph}, each interior vertex of the edge-path is adjacent to one positive and one negative edge of ${\bf E}(v, v^\prime)$.

Now, if ${\bf E}(v,v^\prime)$ does not include our initial choice of the black-square edge, $e$, we extend it to include $e$.  An extension would mean that the edge-path now goes between a $+1$ labeled vertex and $v^\prime$, and we relabel this $+1$ vertex as our $v$.  If our initial choice of a black-square edge is the exceptional negative labeled edge then there is nothing to do.

With this possible extension modification in place, we now traverse ${\bf E}(v, v^\prime)$ starting at $v$ and traveling across the first edge which will have a black-square.  We then place a  on every other edge in the edge-path ${\bf E}(v , v^\prime)$.  When we finally come to the ending vertex $v^\prime$, if it black-square is not adjacent to an edge with a black-square then we add an extending edge that is adjacent to a $+1$ vertex and label this extending edge with a black-square.
If this final extension is necessary, we again adjust the labeling of vertices so that $v^\prime$ is now a $+1$ labeled vertex. 

Now let $e_1, \ldots , e_k \subset \cG$ be the collection of edges that are adjacent to ${\bf E}(v, v^\prime)$, and $T_1 \subset \cG \setminus ({\bf E}(v , v^\prime) \cup \bigcup_i e_i)$ is a collection of trees---here ${\bf E}(v , v^\prime) \cup \bigcup_i e_i$ is the {\em link of ${\bf E}(v,v^\prime)$}.  Any component of $T_1$ that is a single vertex will necessarily have a $+1$ label.  For any component of $T_1$ that is not a single vertex we repeat the above procedure of choosing a black-square edge-path with the following proviso.

Observe that each vertex of $T_1$ inherits the positive/negative edge feature of Fig.~\ref{Fig: vertices of graph}.  In our choice for $v$ and $v^\prime$ we allow for the possibility that an edge path can end or begin at a vertex that has only negative/positive edge labels adjacent to it.  With this in mind, we iterate the choice of a black-square edge-path in each component of $T_1$.  Then $T_2$ is obtained by removing the link of each black-square edge-path in each component of $T_1$. We iterate this removal of the link of black-square edge-paths until what remains is a collection of vertices.  We now reassemble $\cG$ to obtain a consisting labeling of $\cG$.

Once we have a consistent labeling of $\cC \,(\subset S^2)$ by black-squares we can label its curves with $+2$ green-dots and $-2$ green-circles.  Specifically, for a curve $\gamma \in \cC$ with a black-square having Gauss-Bonnet weight $w_\gamma$, if $w_\gamma < 0$ we place $k_{\green{\circ}}$ green-circles along $\gamma$ such that $ w_\gamma = -1 + -2 \cdot k_{\green{\circ}}$.  Whereas if $w_\gamma < 0$, we place $k_{\green{\circ}}$ green-dots along $\gamma$ such that $ w_\gamma = 1 + 2 \cdot k_{\green{\bullet}}$.  Finally, in each $\bar\K$ component having both green-dots and green-circles we make a choice of disjoint twisting arcs connecting the dots to circles.

The reader should observe that the twisting arcs in $S^2$ define a collection of edge-paths in $\cG$, and each edge-path has its endpoints on s.c.c.'s in $\cC$ that have Gauss-Bonnet weight $+1$.  We return to our running example.

\begin{example}[Part 2]
\label{Example: part 2}
Applying the above described procedure to the graph in Fig.~\ref{Fig: vertices of graph}, we can choose black-square designations for the subcollection, $\{\gamma_3, \gamma_5, \gamma_8, \gamma_{12}\}$, which is a consistent labeling of $\cC$.  Here the initial ${\bf E}$-edge-path could be the one that goes from $\gamma_5$ to $\gamma_8$.  Once the link of this edge-path is removed from the graph, we make  $\blacksquare$-labeling choices of $\gamma_3$ and $\gamma_{12}$.  We then transfer this labeling to the configuration in Fig.~\ref{Fig: oriented arcs}.  Now referring to Fig.~\ref{Fig: oriented arcs}, to achieve a $+1$ Gauss-Bonnet weighting on curves $\gamma_1, \gamma_2, \gamma_4, \gamma_{10}, \gamma_{11}, \gamma_{12}$, we place a single $+2$ green-circle on each.  There is no need for additional labeling of $\gamma_3$ and $\gamma_{12}$ since their $\blacksquare$-label implies their Gauss-Bonnet weights are already $+1$. To achieve a $-3$ Gauss-Bonnet weight on $\gamma_5$ and $\gamma_8$ we label each of them with two $-2$ green-circles.  Whereas, we need only label $\gamma_6$ and $\gamma_9$ with a single $-2$ green-circle.  Finally, to achieve the $+5$ Gauss-Bonnet weight on $\gamma_7$ we must label it with three $+2$ green-dots.  Now connecting green-circle to green-dots in each planar component, we obtain three edge-paths.
\end{example}

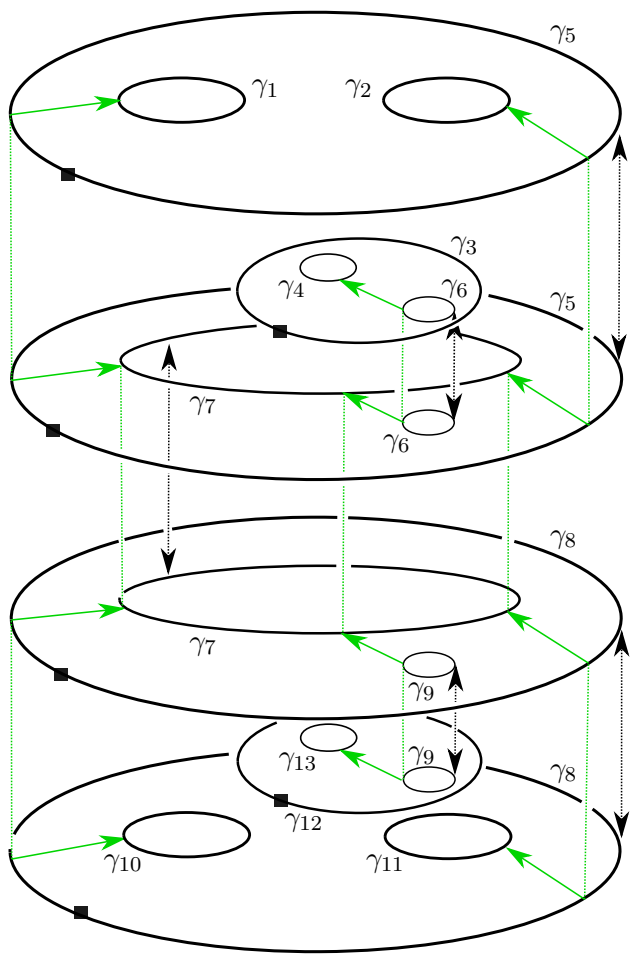
\begin{figure}[ht]

\labellist
\small

\pinlabel $\gamma_1$ [bl] at 204.08 730
\pinlabel $\gamma_2$ [br] at 317.45 730
\pinlabel $\gamma_3$ [bl] at 375 595
\pinlabel $\gamma_4$ [tr] at 260 583
\pinlabel $\gamma_5$ [bl] at 460 776.62
\pinlabel $\gamma_5$ [bl] at 460 547.03
\pinlabel $\gamma_6$ [bl] at 367 558
\pinlabel $\gamma_6$ [tr] at 350 450
\pinlabel $\gamma_7$ [t] at 167.23 485
\pinlabel $\gamma_7$ [t] at 167.23 278
\pinlabel $\gamma_8$ [bl] at 460 342.96
\pinlabel $\gamma_8$ [bl] at 460 141.72
\pinlabel $\gamma_9$ [t] at 355 239
\pinlabel $\gamma_9$ [b] at 355 155
\pinlabel $\gamma_{10}$ [tr] at 120 90
\pinlabel $\gamma_{11}$ [tr] at 345 90
\pinlabel $\gamma_{12}$ [tr] at 275 125
\pinlabel $\gamma_{13}$ [tr] at 270 177

\endlabellist

\centering
\includegraphics[width=.5\linewidth]{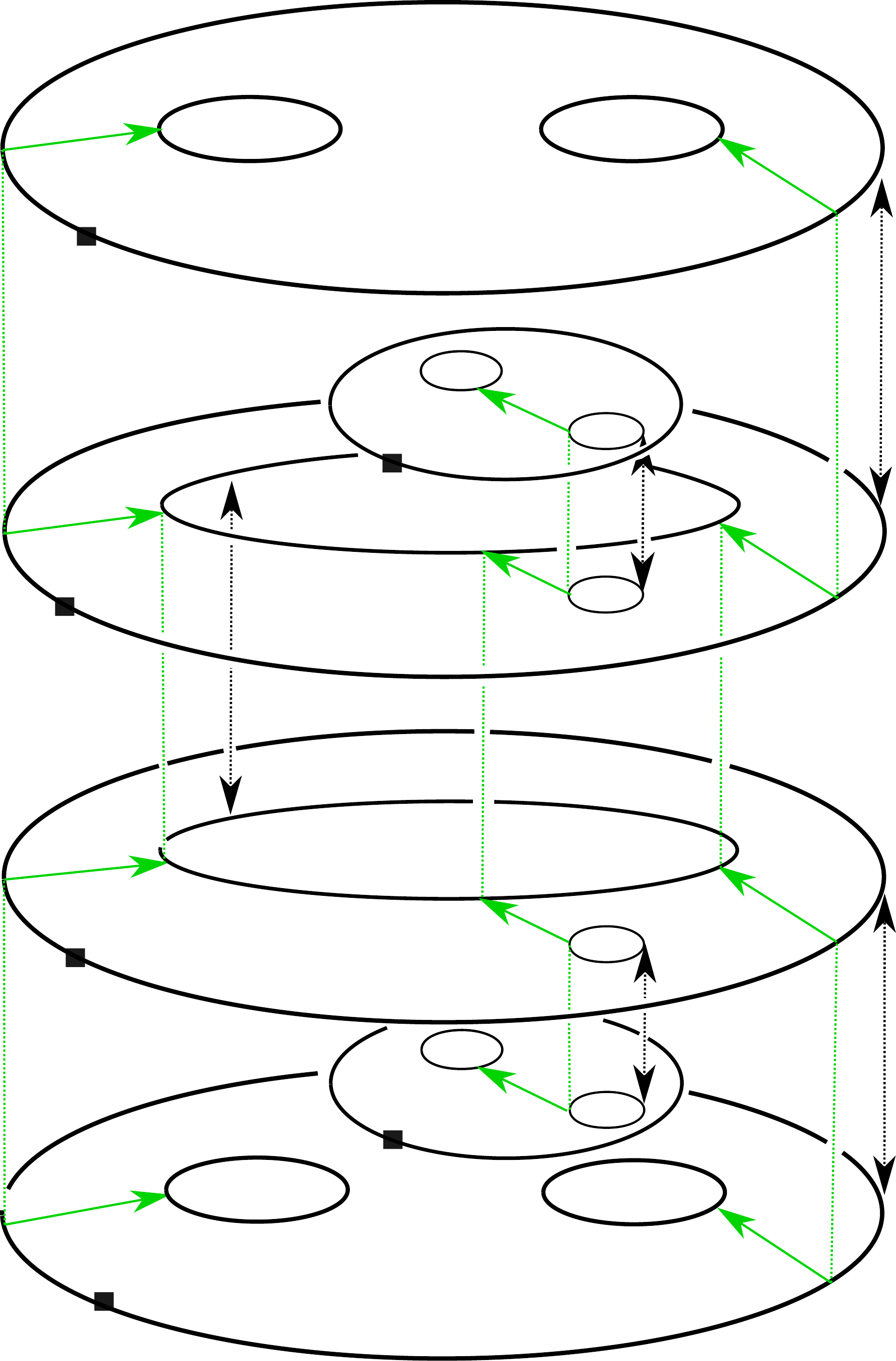}
\caption{A stacking of the planar surfaces from the running example which is consistent with the height/order of the associated vertices in Fig.~\ref{Fig: example 1 graph}.}
\label{Fig: stack example}
\end{figure}

\subsubsection{The height of vertices of a LTG}
\label{subsubsection: The height of vertices}
For a LTG embedding, $\cG \into \bR \times \bR$, satisfying conditions ($\star$), there is still some ambiguity in establishing the heights of the vertices of $\cG$.  To address this issue we utilize our black-square tree decomposition to make the height assignments of the vertices of $\cG$.

We first remark that a black-square decomposition results in having each edge of $\cG$ being in either a black-square edge-path of the decomposition or adjacent to two such edge-paths in the decomposition. 

With that observation in mind, we consider the initial black-square edge-path ${\bf E}(v,v^\prime)$ of a decomposition.  We can readily assign heights to its vertices that correspond to the order in which one encounters them as one traverses the path starting at $v$ and ending at $v^\prime$.  This is due to the condition that the edge-path is alternating---height $0$, height $1$, and so on.  Next, let $e_\gamma \in \cG$ and ${\bf E}^\prime \subset \cG$ be an edge and black-square edge-path coming from the decomposition with ${\bf E}(v,v^\prime) \cap {\bf E}^\prime = \varnothing$ such that $e_\gamma$ is adjacent to both ${\bf E}(v, v^\prime)$ and ${\bf E}^\prime$.  Let $\bar{v}, e_{\bar{\gamma}} \in {\bf E}(v,v^\prime)$ be a vertex and an edge that satisfy the following: $\bar{v}$ is adjacent to both $e_\gamma$ and $e_{\bar{\gamma}}$, and, $e_\gamma$ and $e_{\bar{\gamma}}$ have the same $\pm$ label.  We then place the height assignment for the vertices of ${\bf E}^\prime$ within the $z$-interval support, $I_\gamma$.  Within $I_\gamma$ we again assign heights to the vertices in ${\bf E}^\prime$ that correspond to the order in which they are encountered as ${\bf E}$ is traversed.

In general, for two black-square edge-paths, ${\bf E}_1$ and ${\bf E}_2$---their indices corresponding to the order in which they are chosen in the decomposition---and $e_{\bar{\gamma}}$ being the commonly adjacent edge, we will place the height assignments of the vertices of ${\bf E}_2$ within the $z$-interval support of the edge of ${\bf E}_1$ that shares a vertex with $e_{\bar{\gamma}}$ and has the same $\pm$ label.

\subsubsection{Stacking planar surfaces}
\label{subsubsection:Stacking planar surfaces}
We now utilize the embedding of the graph, $\cG$, in $\bR \times \bR$ to guide a ``stacking'' of components of $S^2 \setminus \cC$ in $\bR^2 \times \bR$.  The setup is that we again have a configuration, $(S^2 ,\cC)$, that has a Gauss-Bonnet weighting, and an associated embedding of its LTG, $\cG \into \bR \times \bR \,(\subset \bR^2 \times \bR)$, that satisfies conditions ($\star$).  Let $C \subset \cC$ be a subcollection that corresponds to a consistent labeling of $\cC$.  Let $\bar\K$ be a compact connected planar surface coming from a connected component of $S^2 \setminus \cC$ with $\chi(\bar\K) \leq 0$.  Finally, let $v_\K \subset \cG$ be the vertex associated with $\bar\K$.

With the above setup we take $\e_s : \bar\K \into \bR^2 \times \{ z_{v_{\K}} \}$ to be a standard embedding such that for the curve, $\gamma = C \cap \partial \bar\K$, we have $t(\e_s(\gamma)) = +1$.  We have such a standard embedding for each planar component where $\chi(\bar\K) \leq 0$.  By condition ($\star$)-0 we know each plane $\bR^2 \times \{{\rm pt.}\}$ contains at most one planar surface.  

Now suppose two distinct vertices, $v_\K , v_{\K^\prime} \in \cG$, are adjacent to an edge $e \in \cG$.  By the definition of $\cG$, this implies $\bar\K$ and $\bar\K^\prime$ share a common boundary curve, $\gamma_e \in \cC$.  We use $\e_s$  (respectively, $\e^\prime_s$) as notation for the standard embedding associated with $\bar\K$ (respectively, $\bar\K^\prime$).  We now require that the $\bar\K$-planar components associated with $(S^2,\cC)$ satisfy the following positioning conditions.

\noindent
{\bf P1}---For two planar components sharing some $\gamma_e$ as a common boundary curve, we position two surfaces in their respective $\bR \times \{{\rm pt.}\}$ planes so that $\p \circ \e_s (\gamma_e) =\p \circ \e^\prime_s (\gamma_e)$.  Additionally, we position $\e_s(\gamma_e)$ and $\e^\prime_s(\gamma_e)$ so that they are bijections on the set of green-circle/dots in $\gamma_e$.

\noindent
{\bf P2}---For two planar components sharing some $\gamma_e$ as a common boundary curve, let $p \subset \gamma_e (= \partial \bar\K \cap \partial \bar\K^\prime)$ be a point that is either a green-circle or green-dot.  Let $a \subset \bar\K$ and $a^\prime \subset \bar\K^\prime$ be two twisting arcs adjacent to $p$.  We further position $\e_s(\gamma_e)$ and $\e^\prime_s(\gamma_e)$ so that $\pi \circ \e_s (a) = \pi \circ \e^\prime_s (a^\prime)$.

We claim that conditions P1 and P2 can be achieved simultaneously for all $\bar\K$-planar components in their respective copies of $\bR^2 \times \{ {\rm pt.} \}$.  Indeed, if we allow our stacking of planar surfaces to mirror the black-square tree decomposition of $\cG$, these two conditions are naturally met.  That is, first stack only the $\bar\K {\rm 's}$ that correspond to vertices in the initial black-square edge-path ${\bf E}(v, v^\prime)$.  For such a ``linear'' stacking of planar surfaces P1 and P2 are easily achieved.  For the next edge-path, ${\bf E}^\prime$, within the $z$-interval support of $I_\gamma$, since it is also a linear stacking, again P1 and P2 are easily achieved.  Iterating this mirroring of the decomposition, we obtain a stacking of planar surfaces satisfying P1 and P2.

\begin{example}[Part 3]
\label{Example: part 3}
Returning to our running example, in Fig.~\ref{Fig: stack example} we show a stacking of $\bar\K, \ \chi(\bar\K) \leq 0,$ components such that conditions P1 and P2 are satisfied.  In particular, the double-headed arrows indicate how the projected image of the boundary curves along with the twisting arcs will coincide under the $\p$ projection map.
\end{example}
\begin{figure}[ht]

\labellist
\small

\pinlabel	$\p$	[r]	at	145	235
\pinlabel   $\bR^2$ [r] at  155 80

\endlabellist

\centering
\includegraphics[width=.8\linewidth]{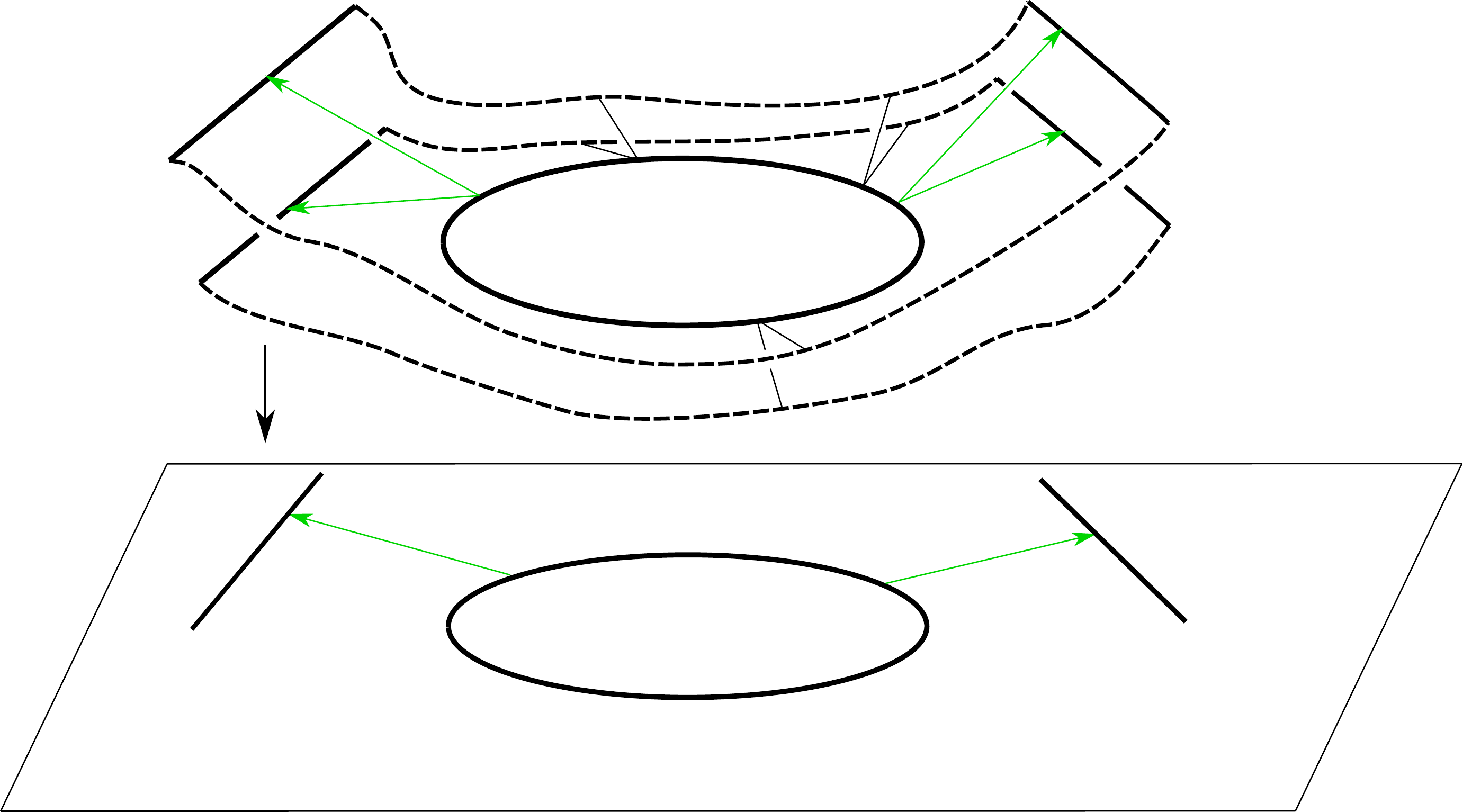}
\caption{A gluing of two stacked planar components along a shared boundary curve.  Note that this creates the stacking of {\green green} oriented arcs whose projection into the $\bR^2$-plane is a single oriented arc.}
\label{Fig: gluing stacked curves}
\end{figure}

\subsubsection{Gluing stacked curves.}
\label{subsubsection: gluing stacked curves}
With the $\bar\K$ components ($\chi(\bar\K) \leq 0$) associated with $(S^2, \cC)$ initially embedded in $\bR^2 \times \bR$ as describe in \S\ref{subsubsection:Stacking planar surfaces}, we next glue together the two copies of each embedded curve of $\cC$.  In particular, for any s.c.c.\ $\gamma \in \cC$ with Gauss-Bonnet weight $t(\gamma) \not= +1$, there will be two planar components, $\bar\K$ and $\bar\K^\prime$ for which $\gamma = \bar\K \cap \bar\K^\prime$.  As a result of our positioning conditions, {\bf P1} and {\bf P2}, we can perform an isotopy in $\bR^2 \times [z_{v_{\bar\K}}, z_{v_{\bar\K^\prime}}]$ that glues $\e_s(\gamma)$ and $\e^\prime_s(\gamma)$ together.  See 
Fig.~\ref{Fig: gluing stacked curves}.

\begin{figure}[ht]

\labellist
\small

\pinlabel	$\p$	[r]	at	330	165

\endlabellist

\centering
\includegraphics[width=.5\linewidth]{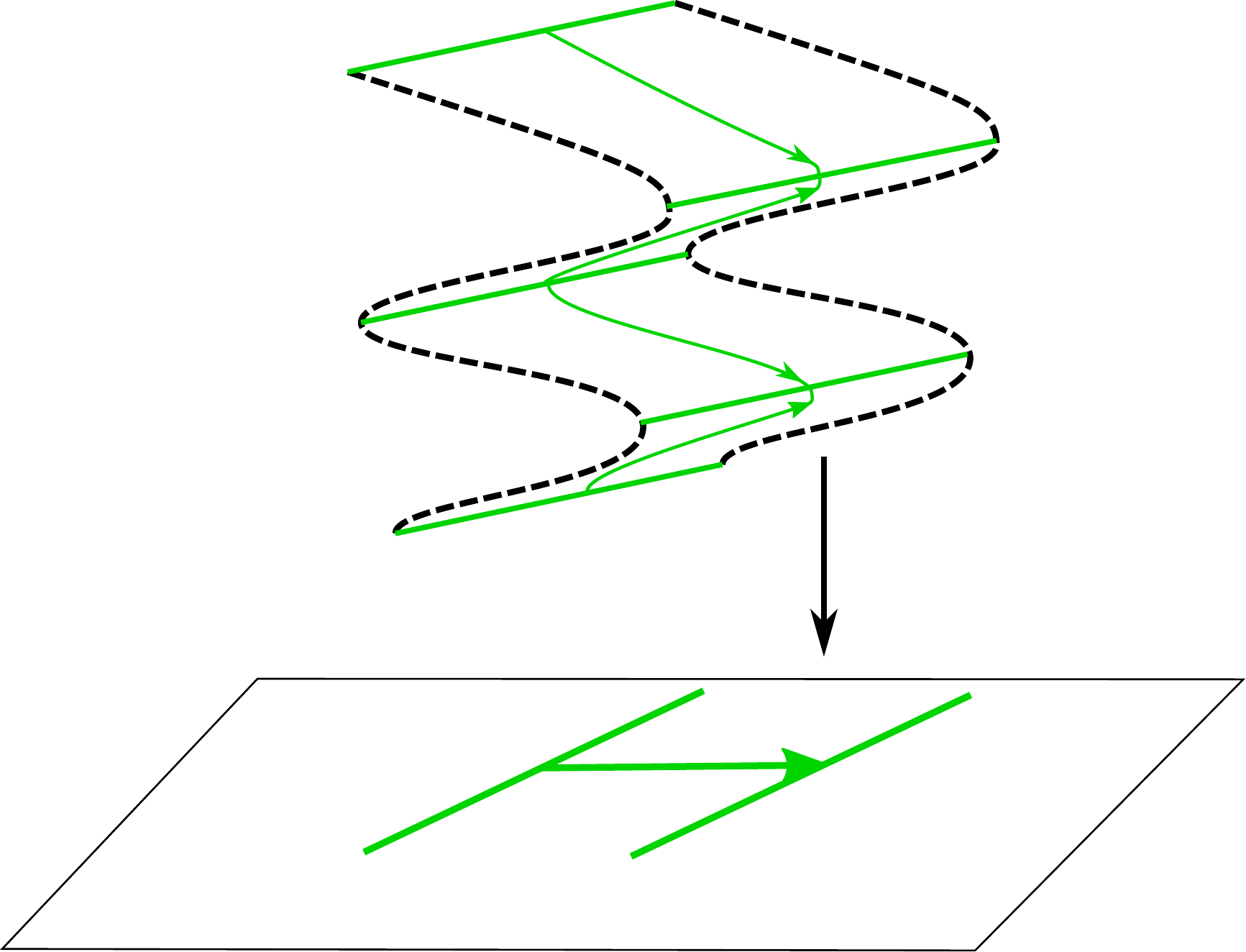}
\caption{The stack of oriented arcs projects onto a single oriented arc in the $\bR^2$-plane.}
\label{Fig: stacking of arcs}
\end{figure}

\begin{figure}[ht]
\centering
\includegraphics[width=.6\linewidth]{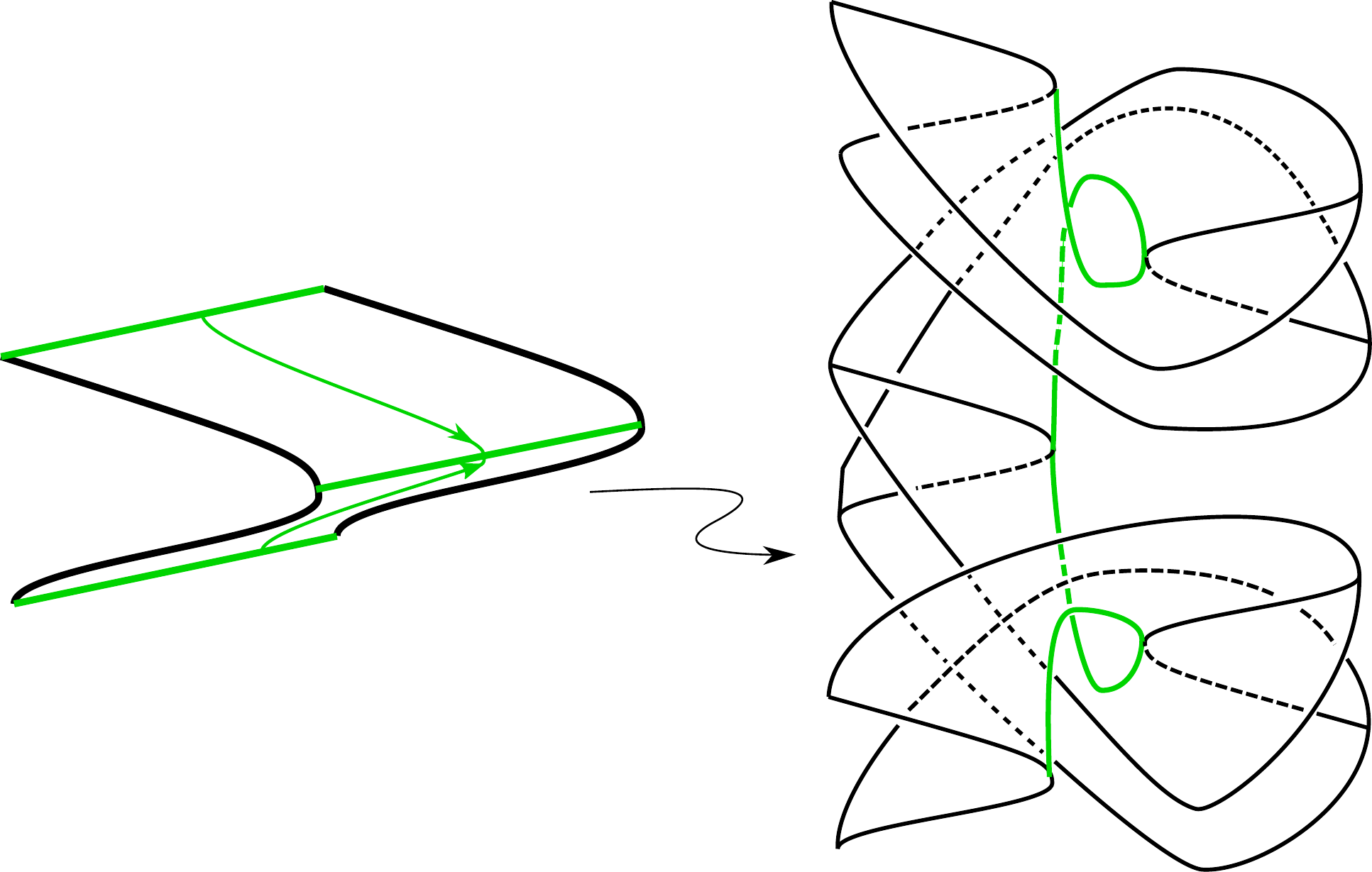}
\caption{Performing a twist on the stacked arcs.}
\label{Fig: twisting stack}
\end{figure}

\subsubsection{Twisting the stacked planar surfaces.}
\label{subsubsection: twisting a stack of arcs}
Once a gluing of all $(\e_s(\gamma) , \e^\prime_s(\gamma))$ pairs of curves have been performed, we observe that condition ($\star$)-2 implies that there is an accordion-like-folding neighborhood of each edge-path of twisting arcs.  See Fig.~\ref{Fig: stacking of arcs}.  Finally, we can perform our twisting operation to this entire edge-path neighborhood as depicted in Fig.~\ref{Fig: twisting stack}.  After performing these twisting operations the image of each s.c.c.\ of $\cC$ will realize its assigned Gauss-Bonnet weighting.  In particular, all curves having $+1$ turning number can now be placed in normal position---nowhere zero curvature---as depicted in Fig.~\ref{Fig: k_i= -1 and -2} and then capped off with a disk.  The interior of such disks will have empty intersection with the crease set.



\section{Embeddings and branched surfaces}
\label{section: classification set-up}
The main thrust of this section is to develop the machinery needed for the classification of isotopy classes for regular embeddings of $S^2$ into $\bR^2 \times \bR$ having a crease set of three curves without corners.  Once our machinery has been developed we will carry out the classification calculation in \S\ref{section: when |cC|=3}.

We start by formalize our meaning of isotopy class with the following definition.
\begin{definition}
An isotopy $\e_t: S\to \bR^2\times \bR$, $0\le t\le 1$, is {\em regular} if the critical set of $\p\circ \e_t$ is $\cC$ for all $t\in[0,1]$.
\end{definition}

Such an isotopy restricts to each subsurface $\bar\K$ as defined in \S\ref{subsection: corners}. Note that a regular isotopy preserves the local data of the embedding, in that it cannot add or remove corners, nor change the sign of a crease curve, as defined in \S\ref{subsection: Orientation of crease}. However, a regular isotopy may pass through non-regular embeddings, as the assumption of transverse self-intersections of $\p\circ \e (\cC)$ may be violated.

\subsection{The natural branched surface}
\label{subsection: branched surface}
There is a natural branched surface with boundary, $\B \subset \bR^2 \times \bR$, associated with any regular embedding, $\e : S_g \rightarrow \bR^2 \times \bR$.  This is due to the fact that the projection, $\p: \bR^2 \times \bR \into \bR^2$, induces an $I$-bundle fibration on $M_{S}$, the 3-manifold bound by $\e(S)$, where the fibers are the components of $\p^{-1} (p) \cap M_{S}$, for any $p \in \bR^2$.  $\B$ is then the quotient space, $M_S / \{I - {\rm fibers } \}$.  We will use $q_{\e}: \e(S) \rightarrow \B \, \ (\subset \bR^2 \times \bR)$ to denote the image of $S$ under this quotient map.

\begin{figure}[ht]
\centering

\begin{subfigure}{.3\textwidth}
\centering
\includegraphics[width=\textwidth]{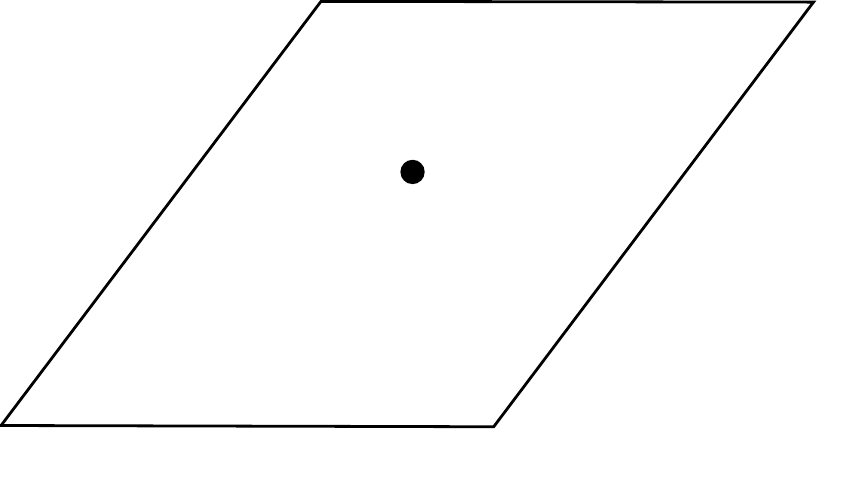}
\caption{Away from branching.}
\label{Fig: neighborhoods of B1}
\end{subfigure}
\begin{subfigure}{.3\textwidth}
\centering
\includegraphics[width=\textwidth]{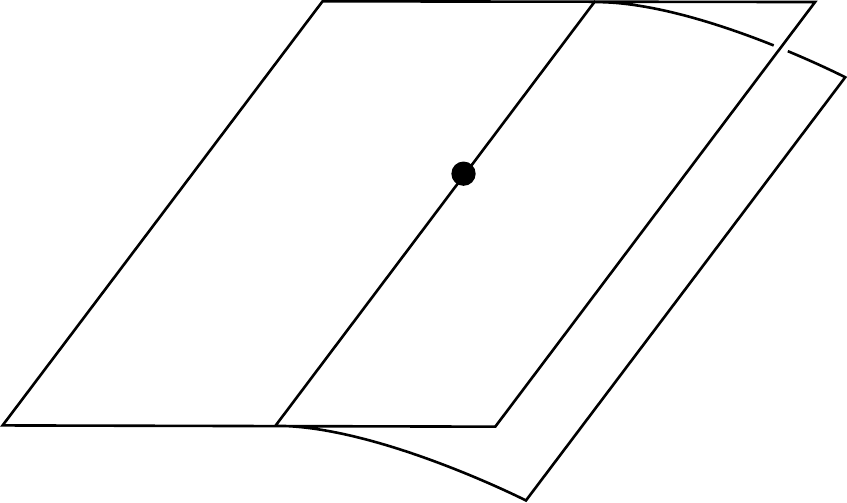}
\caption{A branching point.}
\label{Fig: neighborhoods of B2}
\end{subfigure}
\begin{subfigure}{.3\textwidth}
\centering
\includegraphics[width=\textwidth]{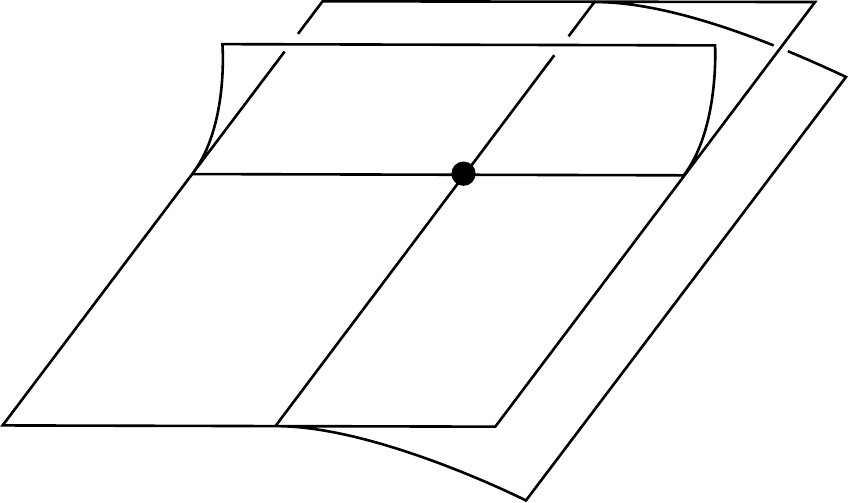}
\caption{A double point of branching.}
\label{Fig: neighborhoods of B3}
\end{subfigure}
\caption{Neighborhoods of different points of $\B$.}
\label{Fig: neighborhoods of B}
\end{figure}

Taking a brief aside for readers who are not familiar with branched surfaces, our $\B$ will be a smooth surface with a codimension-$1$ branching locus singular set.  Fig.~\ref{Fig: neighborhoods of B} depicts the three neighborhood models for an interior point in $\B$.  The boundary of $\B$ is a compact $1$-submanifold which may contain endpoints of the of the branching locus.  
For our discussion on the classification of embeddings of $2$-sphere with three crease components, the branching locus and boundary of $\B$ will be disjoint.  The key point for our discussion is that at every interior point, $p \in \B \setminus \partial \B$,  there is a well defined tangent plane $T_p \B$ for which the differential, 
$d\p: T_p \B \rightarrow \bR^2 \,(\subset \bR^2 \times \bR)$,
will be an isomorphism.  This interior tangent 
bundle smoothly extends to a tangent bundle of the boundary for which $d\p$ is again an isomorphism.

Although it is possible to connect our discussion of $\B$ with the historical development of branched surfaces in the literature---specifying the horizontal and vertical boundary of the $I$-bundle structure of $M_S$ and their interplay with $\cC$---it is more straight forward to develop $\B$ from our current machinery.  To that end let, $\cC = \cC^+ \cup \cC^-$, be the decomposition of the crease set into its the positive folding and negative folding components.  Coming from the alternative way of defining the folding direction, we observe that there is a natural embedding and identification, $\cC^+ \longleftrightarrow \partial \B \,(\subset \B)$, coming from the fact that a point, $p_0 \subset \partial \B$, corresponds to a point component of $\p^{-1} (p) \cap M_{S}$ for some $p \in \bR^2$.  Thus, we will abuse notion by having $\cC^+ = \partial \B$.

Similarly, let $\LL$ denote the branching locus of $\B$. We now observe that there is a natural immersion, $\cC^- \into \LL \,(\subset \B)$.  This comes from the fact that a point, $p \in \LL$, corresponds to a fiber component, $I \subset \p^{-1} (\p(p)) \cap M_{S}$, for which ${\rm int}\,{I} \cap \e(\cC^-) \not= \varnothing$.  (Here the ${\rm int}\,{I}$ is the interior of the fiber.)  Since $\p \circ \e(\cC)$ is a $4$-valent graph, ${\rm int}\,{I} \cap \e(\cC^-)$ is either one or two points.  If it is just one point then $p$ lies in a neighborhood modeled on Fig.~\ref{Fig: neighborhoods of B2}.  If there are two points then the neighborhood model corresponds to Fig.~\ref{Fig: neighborhoods of B3}, a double point in $\LL$.  Again, we may abuse notion by writing $\cC^- = \LL \subset \B$. We refer readers interested in a fuller development of branched surfaces to \cite{[Le]}.

A primary use of branched surfaces is their 
``book-keeping'' function for carrying embedded surfaces in a fibered neighborhood---surfaces that correspond to assigning non-negative 
weights to the components of $\B \setminus \LL$.  (Again, see \cite{[Le]} for a detailed treatment.)  

For $\B$ in our setting, there is a natural way to see the embedding of the components, $\{ \bar{K}_1, \cdots , \bar{K}_l \} \subset S_g$, ``carried'' by $\B$ if we allow both $\partial \B$ and $\LL$ to serve as boundary.  To start, it is natural to decompose the $\bar{K}_i {\rm 's}$ into two sets.  Recall that each $\bar{K}_i$ inherits a normal vector field coming from the outward pointing normal of $M_S$. As such, for a given open $\K_i$, its normal vectors project onto $\lambda \mathbf{j}$, where $\mathbf{j}$ is the positively oriented unit associated with the $z$-axis---the $\bR$ factor of $\bR^2 \times \bR$.  Let $\bK^+$ (respectively, $\bK^-$) be the sub-collection of $\bar{K}_i {\rm 's}$ for which $\lambda >0$ (respectively, $\lambda <0$).

\begin{figure}[ht]

\labellist
\small
\pinlabel $\bK^+$ [r] at 185 158
\pinlabel $\bK^-$ [l] at 435 158
\endlabellist

\centering
\includegraphics[width=.6\linewidth]{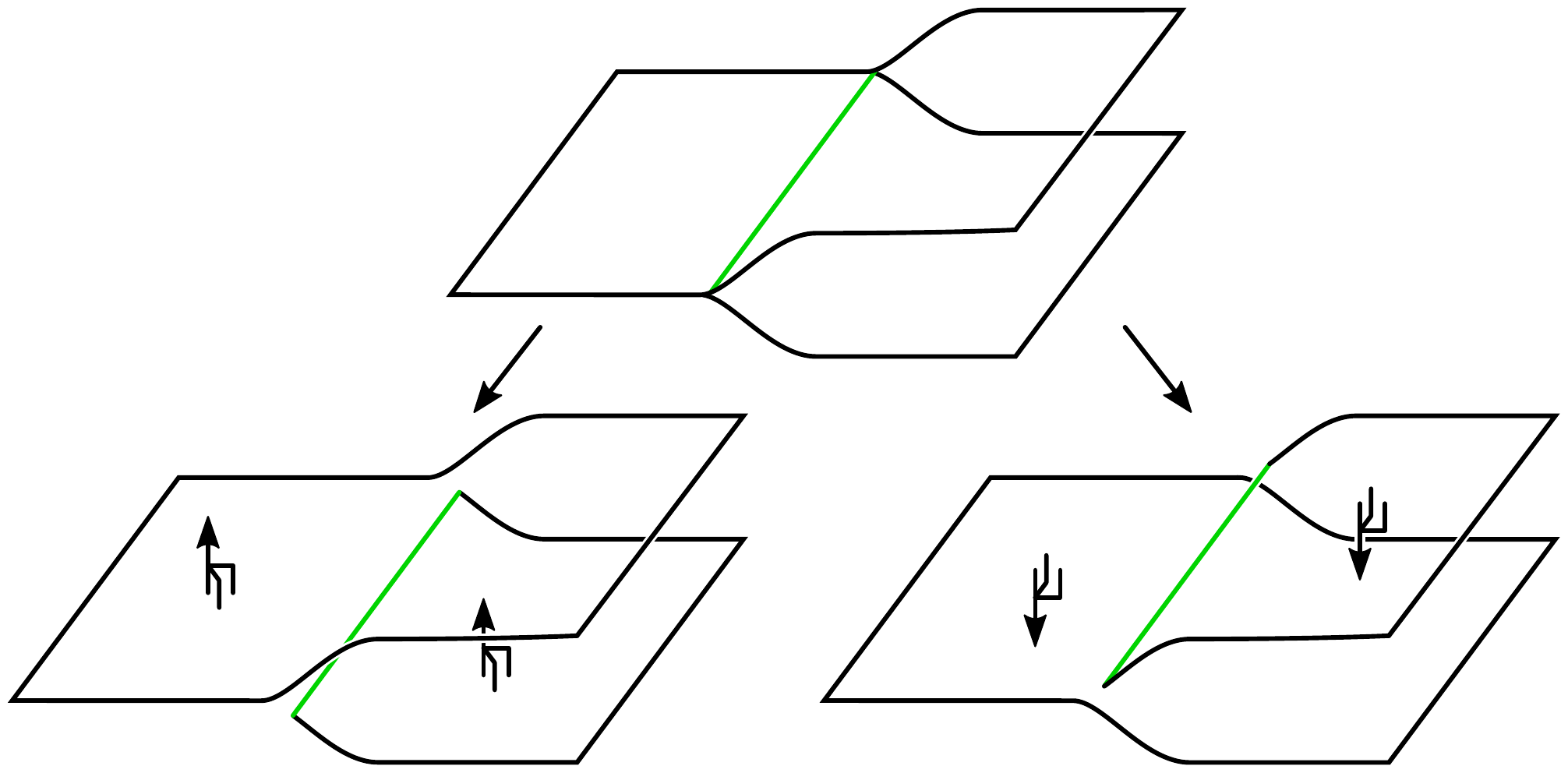}
\caption{The branched surface carries $S_g$ by splitting appropriately at $\LL$.}
\label{Fig: carries}
\end{figure}

With the above in place the reader should observe that a weight assignment of $+1$ to each component of $\B \setminus \LL$ corresponds to $\B$ carrying either $\bK^+$ or $\bK^-$.  Specifically, in Fig.~\ref{Fig: carries} we locally depict how $\bK^+$ or $\bK^-$ are configured near $\partial \B$ and $\LL$.  The reader should observe that we treat $\LL$ as lying in the boundary of the resulting surface, so that the weighting satisfies the branching equations.

\subsection{Regular isotopies and the branched surface.}
\label{subsection: isotopies of B}
Given a regular isotopy, $\e_t : S \into \bR^2 \times \bR$, $0 \leq t \leq 1$, there is an associated family of branched surfaces $\bar{\e}_t(S)=q_{\e_t}\circ\e_t(S)$, $0\le t\le 1$. When all self-intersections of $q_{\e_t}\circ \e_t (\cC)$ are transverse, the topological type of $\bar\e_t(S)$ is fixed, so $\bar\e_t$ is an isotopy of the associated branched surface $\B$ such that all points of $\p (\B \setminus (\partial \B \sqcup \LL))$ are regular values; in this case we abuse notation and write simply $\bar\e_t(\B)$.  To account for those isolated values of $t\in[0,1]$ where $q_{\e_t}\circ\e_t(\cC)$ has non-transverse intersections, we expand the meaning of a regular isotopy $\bar{\e}_t$ of $\B$ to include two classical operations on branched surfaces, {\em pinching} and its inverse {\em splitting}. Fig.~\ref{Fig: branched surface isotopy} depicts these local operations along a neighborhood of the branching locus, $\LL \subset \B$.  Both operations can result in the introduction or removal of double points.

\begin{figure}[ht]

\centering
\includegraphics[width=.9\linewidth]{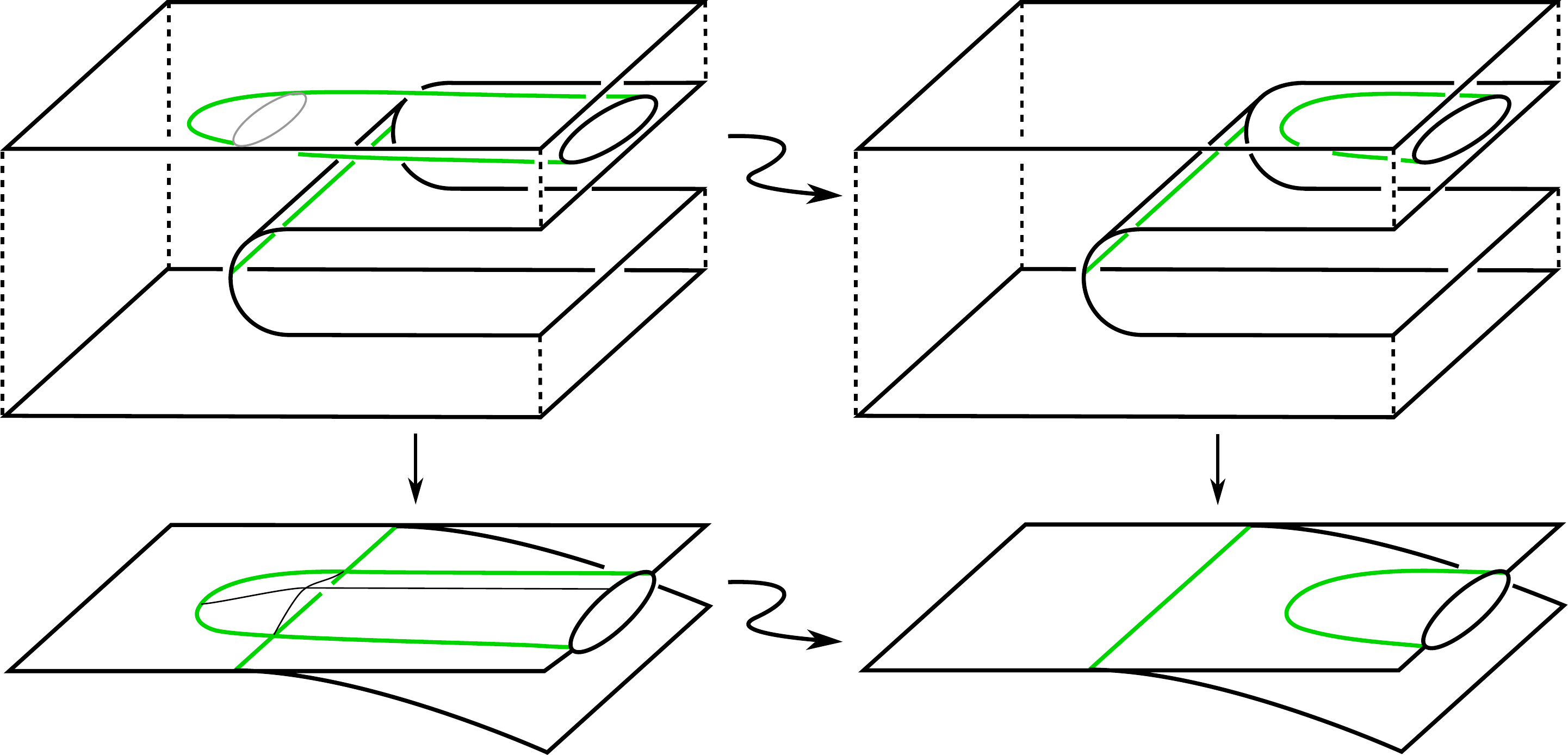}
\caption{A regular isotopy of $\e(S)$ which descends to a pinching move on $\B$.}
\label{Fig: branched surface isotopy}
\end{figure}

\subsubsection{Boundary based isotopies}
\label{subsubsection: boundary based isotopies}
One primary source of regular isotopies comes from moving $\partial \B$ along a portion of ${\rm int}\,{\B}$.  Specifically, we say $R \subset \B \setminus \LL$ is a {\em boundary annular region} if $R$ is an annulus such that $\partial R = a \cup b$ with $a \subset \partial \B$ and $b \subset {\rm int}\,{\B} \setminus \LL$.  We say $R \subset \B \setminus \LL$ is a {\em boundary half-disk} if $\partial R = a \cup b$ where $a \subset \partial \B$ is an arc and $b \subset \B \setminus \LL$ is a properly embedded arc---$\partial b = \partial a \subset \partial B$---and $(\partial \B \setminus {\rm int}\,{a} ) \cup b$ is a smooth $1$-manifold.  We then have the following lemma.

\begin{lemma}
\label{lemma: isotopies across R}
Let $R \subset \B \setminus \LL$ be either a boundary annular region or a boundary half-disk region.  Then there is a regular isotopy $\bar{\e}_t (\B), \ 0 \leq t \leq 1$, such that $R \cap \bar{\e}_1 (\B) = b$ and $\bar{\e}_t (\B \setminus R) = \mathbf{ id}, \ 0 \leq t \leq 1$.
\end{lemma}

\begin{proof}
Let $\phi_t$, $0 \leq t \leq 1$, be an smooth flow coming from extending the inward pointing normal vector field along $a\, (\subset \partial R)$ to all of $R$ such that along $b$ the vector field is an outward (with respect to $R$) pointing normal. Viewing $R\cong a\times[0,1]$, with $a\times\{1\}$ identified with $a$, we take $\bar{\e}_t (\B)$ to be $\phi_t(a,s)s$ on $R$ and $\bar{\e}_t(\B \setminus R) \cong \mathbf{id}$, $0 \leq t \leq 1$. 
\end{proof}

\subsubsection{Pinching and double-cusp disks}
\label{subsubsection: pinching and double-cusp disks}
Next, we develop the tools for pinching a branched surface $\B$.  This pinching will correspond to a regular isotopy of the embedding $\e$, thus preserving $\cC^-$ (which corresponds to $\LL$) and $\cC^+$ (which corresponds to $\partial \B$).  
Let $d^+ , d^- \subset \B$ be two simple arcs such that:
\begin{itemize}
    \item ${\rm int}\,{d}^\pm \subset \bK^\pm$;
    \item $\partial d^+=\partial d^-=d^+ \cap d^- \subset \LL$; and,
    \item $\p(d^+) = \p(d^-)$.
\end{itemize}

Then $d^+ \cup d^-$ bounds a {\em double cusp disk}, $\td$, for which $\p( \td) = \p( d^\pm)$.  Necessarily, $\partial \td \cap \LL$ will contain the two {\em cusp points} of $d^+ \cap d^-$ (along with all points of ${\rm int}\,{d}^\pm \cap \LL$).

The double cusp disks that will be of interest to us are ones that are associated with a {\em pinching $3$-ball}.  Specifically, suppose $d^\pm \subset \bK^\pm$ split off disks $\delta^\pm \subset \bK^\pm$ such that
$$\Sigma \ = \ \td \cup  \delta^+ \cup \delta^- $$ is a $2$-sphere bounding a $3$-ball, $B_\Sigma$.  To parse out the gluing: $\td$ and $\delta^+$ are glued along $d^+$; $\td$ and $\delta^+$ are glued along $d^+$; and, $\delta^+$ and $\delta^-$ are glued along a segment $\ell \subset \LL$ for which $\partial \ell = \partial d^+ = \partial d^-$.  We observe that $\p (B_\Sigma) = \p ( \delta^+) = \p (\delta^-)$.  As such, for any point $p \in \p (B_\Sigma)$ the segment components of its pre-image give us an $I$-bundle structure for $B_\Sigma$. Specifically,  $$B_\Sigma = \bigsqcup_{t \in [-1,1]} D_t,$$ where 
\begin{itemize}
    \item $D_{-1} = \delta^+$ and $D_{1} = \delta^-$;
    \item for all $t \in [-1,1]$, $\p(D_t) = \p(\delta^\pm)$;
    \item for all $t \in [-1,1]$ and any $p \in {\rm int}\,{D}_t$, $d\p_p: T_p D_t \rightarrow \bR^2$ is an isomorphism;
    \item for all $t \in [-1,1]$, $\partial D_t = \ell \cup d_t$ where $d_t \subset \td$; and
    \item $\td = \bigsqcup_{t \in [-1,1]} \  d_t$ with $d_{-1} = d^+$ and $d_1 = d^-$.
\end{itemize}

Observe that when ${\rm int}\,\td \cap \B = \varnothing$, then ${\rm int}\,B_\Sigma \cap \B = \varnothing$.  We can then use the $I$-bundle structure of $B_\Sigma$ to pinch $\B$ along $\ell$, with the isotopy moving $\ell$ through the disk $\delta^+ \,(\subset B_\Sigma)$ to the arc $d^+ \,(\subset \ \td)$.

When we utilize the above pinching procedure in \S\ref{section: when |cC|=3}, we will be concerned with, first, identifying such double-cusp disks and, second, ensuring that an associated ${\rm int}\,{B}_\Sigma$ has empty intersection with $\B$. We return to this second point in \S\ref{subsubsection: emptying B-Sigma}.

\subsubsection{Pinching and $\partial$-close branching locus segments}
\label{subsubsection: pinching boundary close}
For a given, $\e: S_g \rightarrow \bR^2 \times \bR$, and associated branched surface, $\B$, let $\ell \subset \LL$ be a closed segment with $\partial \ell = \{p_1 , p_1\}$.  Let $R \subset \bK^+ \ {\rm or} \ \bK^-$ be a rectangular region (for simplicity, we take $R \subset \bK^+$), with $\partial R = \ell \cup r_1 \cup b \cup r_2$ where:
\begin{itemize}
    \item $r_i$, $i \in \{0,1\}$, are arcs in $\bK^+$;
    \item $b = R \cap \partial \B \subset \bK^+$ is a segment with $\partial b = \{ q_1 , q_2 \}$;
    \item $r_i$ is glued to $\ell$ at $p_i$, $i \in \{0,1\}$; and
    \item $r_i$ is glued to $b$ at $q_i$, $i \in \{0,1\}$.
\end{itemize}
We say $\ell \,(\subset \LL)$ is $\partial$-close with respect to $\bK^+$, or just $\partial$-close, if $R$ exists such that ${\rm int}\,R \cap \LL = \varnothing$.

Similarly, let $\ell \subset \LL$ be a circle component.  Let $\ell \subset \partial \bar{K}$ where $\bar{K} \subset \bK^+ \ {\rm or} \ \bK^-$ is an annulus.  If ${\rm int}\,\bar{K} \cap \LL = \varnothing$ then we say $\ell$ is $\partial$-close with respect to $\bK^+$, or just $\partial$-close.

The salient feature of $\partial$-close is, in the annular case, we have a regular isotopy (by Lemma \ref{lemma: isotopies across R}) that positions $\bar{K} \cap \partial \B$ arbitrarily close to $\ell$.  In the case where $R$ is a rectangular region we can take a half-disk, $R^\prime \subset R$, for which $\partial R^ \prime \setminus b$ is arbitrarily close to $r_1 \cup \ell \cup r_2$.  Then an application of Lemma \ref{lemma: isotopies across R} to $R^\prime$ positions $\partial \B$ arbitrarily close to $\ell$.

\begin{figure}[ht]

\labellist
\small
\pinlabel $\ell$ [r] at 130 105
\pinlabel $R$ [l] at 195 81
\endlabellist

\centering
\includegraphics[width=1.0\linewidth]{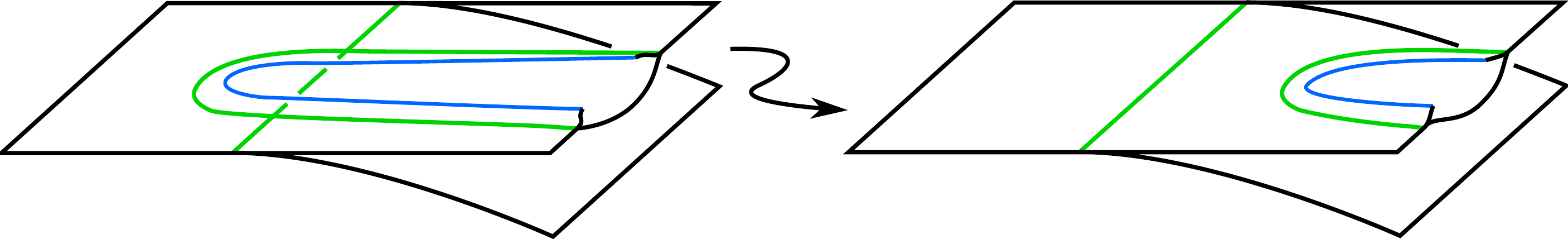}
\caption{Pinching along a portion of $R$ to remove two double points along the {\green green} $\partial$-close arc $\ell$.}
\label{Fig: boundary close pinching}
\end{figure}

Now, let $R \subset \bar{K} \,(\subset \bK^\pm)$ be either a disk or an annulus.  Suppose $R$ is an annulus with $\partial R = \ell \cup b$, where $\ell \subset \LL$ is a circle component and $b \subset {\rm int}\,{\bar{K}}$; we say $R$ is an {\em $\LL$-annulus}.  For $R$ a disk, suppose we may write $\partial R = \ell \cup b $ where $\ell \subset \partial \LL$ is an segment and $b \subset \bar{K}$ is a properly embedded arc---$\partial b = \partial \ell \subset \partial \LL$---and moreover, $(\LL \setminus {\rm int}\,{\ell} ) \cup b$ is a smooth $1$-manifold.  Then we say $R$ is an {\em $\LL$-disk}.  We then have the following lemma.

\begin{lemma}
\label{lemma: boundary close pinching}
Let $R \subset \bar{K} \,(\subset \bK^\pm)$ be either an $\LL$-annulus or $\LL$-disk.  Let $\ell \subset \partial R$ be its associated boundary circle or segment in $\LL$.  If $\ell$ is $\partial$-close with respect to $\bK^\mp$ then we can pinch $\B$ along $\ell$ so as to replace $\ell$ in $\LL$ with $b$.  (See Fig.~\ref{Fig: boundary close pinching}.)
\end{lemma}

\begin{proof}
For simplicity of notation, assume $\ell \subset \bar{K} \subset \bK^+$.  Then $\ell$ is $\partial$-close with respect to $\bK^-$ and we can perform a regular isotopy of $\B$ such that $\ell$ is arbitrarily close to either a component of $\partial \B$ (when $R$ is an annulus) or a segment of $\partial \B$ (when $R$ is a disk).  We can then pinch $\LL$ through $R$ and extend the isotopy to the portion of $\bK^-$ that is close to $\ell$; importantly, since this is an arbitrarily small strip of $\bK^-$, we do not need to worry about interactions with the remainder of $\B$.  Fig.~\ref{Fig: boundary close pinching} illustrates the pinching operation.
\end{proof}

\subsubsection{Emptying out a $B_\Sigma$ ball}
\label{subsubsection: emptying B-Sigma}
The final operation we now discuss will allow us to assume that the interior of the $B_\Sigma$ $3$-balls introduced in \S\ref{subsubsection: pinching and double-cusp disks} have empty intersection with $\B$.  In that discussion we were given a double cusp disk, $\td$, and we assumed that there was an associated pinching $3$-ball.  If we restrict to embeddings of $2$-spheres, then the bounded component of $\bR^2 \times \bR \setminus (\e(S^2) \cup \td)$ is always a $3$-ball.  When we consider the situation where $\cC$ has only three components, as we will in \S \ref{section: classification set-up}, the condition that $\p(B_\Sigma) = \p(\delta^+) = \p(\delta^-)$ will be almost immediate.  As such, we have the following lemma.

\begin{lemma}
\label{lemma: emptying B-Sigma}
Let $B_\Sigma$ be a pinching $3$-ball as described in \S\ref{subsubsection: pinching and double-cusp disks}.  Let $\td$ be the associated double cusp disk and $\Sigma$ be the associated $2$-sphere.  Then by a regular isotopy of $\B$ we can assume that ${\rm int}\,{B_\Sigma} \cap \B = \varnothing$.
\end{lemma}

\begin{proof}
As a technical point we take an small collar neighborhood of $N(\td) = \td \times (-\varepsilon,\varepsilon)$ where $N(\td) \cap B_\Sigma$ is supported only over the $[0,\varepsilon)$ interval.  We then extend the $D_t$-disk bundle structure of $B_\Sigma$ to $B^\prime = B_\Sigma \cup N(\td)$.  We will use the notation $$B^\prime = \bigsqcup_{t \in [-1,1]} D^\prime_t$$ for the resulting bundle structure.  Similarly extending the notation of \S \ref{subsubsection: pinching and double-cusp disks}, we have $\partial D^\prime_t = \ell^\prime \cup d^\prime_t$ where $\ell \subset \ell^\prime \subset \LL$---an extension of $\ell$---and $d^\prime_t \subset \td \times \{-1\}$.

Then we choose a non-zero vector field flow, $\phi_s, \ s \in (-1, 1)$, on ${\rm int}\,{B}^\prime$ such that each restricted flow $\phi_s|_{{\rm int}\,{D}^\prime_t},\ t \in [-1,1]$ has an extension to a flow $\phi_{s,t}$ such that on $\ell^\prime$, $\phi_{s,t}$ is an inward pointing normal vector field and on $d^\prime_t$, it is an outward pointing normal vector field.

With this setup, we use $\phi_s$ to perform a smooth isotopy moving $\B \cap {\rm int}\,{B}_\Sigma$ to lie in the $(-\varepsilon,0)$ interval support of $N(\td)$.  Such a isotopy will necessarily preserve the regular points of $\p : \B \rightarrow \bR^2$.
\end{proof}

\section{Classification of $S^2$ embeddings when $|\cC| =3$}
\label{section: when |cC|=3}

In this section we give the classification of isotopy classes of regular embeddings of $S^2$ into $\bR^2 \times \bR$ having $\cC = \{ \gamma_i , \gamma_m , \gamma_o \}$---three smooth curves without corners.  Seemingly a simple case, we will see that the subtleties of what can occur are instructive for what may happen in the general setting.

To begin, it is readily seen that there are topologically two possible configurations of three disjoint curves on $S^2$: one where the curves split $S^2$ into three disks plus a pair-of-pants; and, one where the curves split $S^2$ into two disks and two annuli.  The former case is not allowable since, as previously observed, there is no Gauss-Bonnet weighting of $\cC$.  For the latter case, we let $\gamma_i$ and $\gamma_o$ (``$i$'' for ``inner'' and ``$o$'' for ``outer'') be the two curves bounding disks $\D_i$ and $\D_o$, respectively.  Since both curves bound disks, by Proposition~\ref{proposition: GB for barK}, their Gauss-Bonnet weightings most both be $+1$.  Additionally, we let $\gamma_m$ (``$m$'' for ``middle'') be the curve that co-bounds an annulus, $A_{i,m}$, with $\gamma_i$ and co-bounds an annulus $A_{o,m}$, with $\gamma_o$. Again, applying Proposition~\ref{proposition: GB for barK} we have $-1$ for the Gauss-Bonnet weight for $\gamma_m$.

\begin{theorem}
\label{theorem: mushroom-saucer-toric}
Let $\e: S^2 \into \bR^2 \times \bR$ be a regular embedding such that the crease set $\cC$ has three components with no corners.  Then $\e(S^2)$ is regularly isotopic to one of the following, up to the obvious (reflection) symmetries: the ``saucer'' (Fig.~\ref{Fig:saucervshroom1}), the ``mushroom'' (Fig.~\ref{Fig:saucervshroom2}), or the ``toric'' sphere (Fig.~\ref{Fig:saucervshroom3}).
\end{theorem}

\begin{figure}[ht]
\centering

\begin{subfigure}[b]{.3\textwidth}
\centering
\includegraphics[width=.9\textwidth]{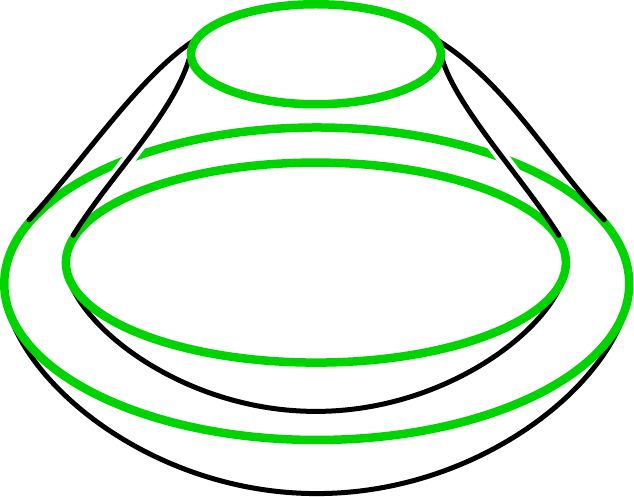}
\caption{Saucer.}
\label{Fig:saucervshroom1}
\end{subfigure}
\begin{subfigure}[b]{.3\textwidth}
\centering
\includegraphics[width=.65\textwidth]{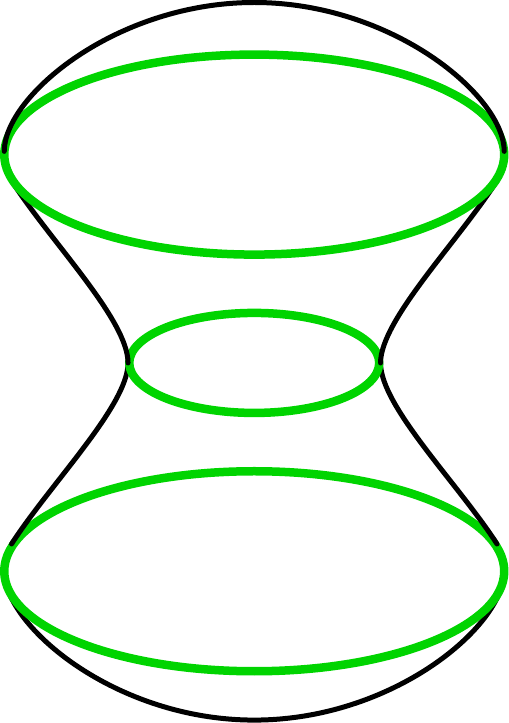}
\caption{Mushroom.}
\label{Fig:saucervshroom2}
\end{subfigure}
\begin{subfigure}[b]{.3\textwidth}
\centering
\includegraphics[width=\textwidth]{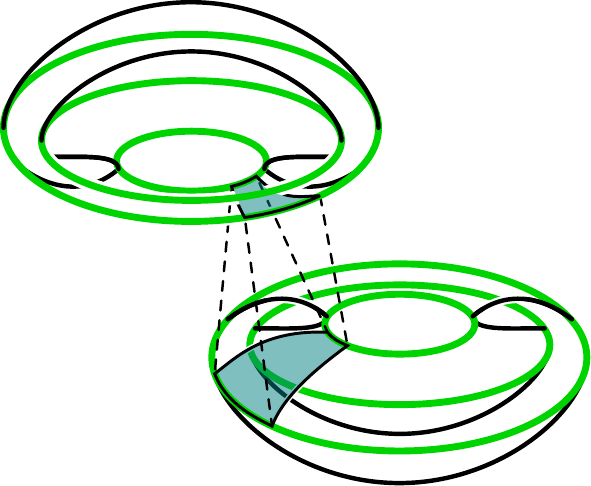}
\caption{Toric.}
\label{Fig:saucervshroom3}
\end{subfigure}
\caption{Model representatives of the three regular isotopy classes of embeddings of $S^2$ with three corner-less crease curves. The toric sphere is realized as a plumbing of two saucer spheres.}
\label{Fig:saucervshroom}
\end{figure}

To be precise about the meaning of obvious symmetry, we allow pre-composition by a homeomorphism of pairs $h:(S^2,\cC)\to(S^2,\cC)$ and post-composition by an isometry $i:\bR^2\times\bR\to\bR^2\times\bR$ satisfying $\p\circ i=i|_{\bR^2\times\{0\}}\circ \p$. For instance, the saucer and its reflection through the $\bR^2\times\{0\}$-plane are not regularly isotopic, but are related by an obvious symmetry.

\subsection{Proof approach}
\label{subsection: sign of cC}

As remarked earlier, the folding data of the crease set are a regular isotopy invariant. Theorem~\ref{theorem: mushroom-saucer-toric} may be understood as saying that in the case $|\cC|=3$, this is a complete invariant---at most one isotopy class occurs per folding direction assignment---and only three folding direction assignments actually occur.

Initially the number of possible assignments for the folding signs is $2^{|\cC|} = 2^3$.  However, since $\partial \B \not= \varnothing$ and $\partial \B$ corresponds to $\cC^+$, some component of $\cC$ must be a $+$ crease.  Similarly, if there is no $-$ crease, then $\LL = \varnothing$ and $\B$ is in fact an embedded planar surface with boundary.  Since $\B$ must carry $\e(S^2)$, $\B$ is necessarily a disk.  But, then $|\cC| =1$, not $3$, a contradiction.  Thus, both $\cC^+$ and $\cC^-$ are nonempty. 

Next we observe that we can interchange of roles of $\gamma_i$ and $\gamma_o$, an allowed symmetry. This leaves us with four bijective maps of ordered sets that correspond to a folding assignment.
\begin{enumerate}
    \item[i.] $f_1 : [\gamma_i , \gamma_m , \gamma_o] \longleftrightarrow [+,+,-] $.
    \item[ii.] $f_2 : [\gamma_i , \gamma_m , \gamma_o] \longleftrightarrow [+,-,+] $.
    \item[iii.] $f_3 : [\gamma_i , \gamma_m , \gamma_o] \longleftrightarrow [-,+,-] $.
    \item[iv.] $f_4 : [\gamma_i , \gamma_m , \gamma_o] \longleftrightarrow [+,-,-] $.
\end{enumerate}

\begin{lemma}\label{Lem: eliminating f_4}
The folding assignment $f_4$ cannot occur.
\end{lemma}

\begin{proof}
Suppose $\e:S^2\rightarrow \bR^2\times\bR$ realizes this folding assignment. Observe that $\gamma_i$, as the single crease curve with positive folding, bounds an embedded disk $\D$ in $\B$, and moreover $\partial \B=\partial \D$. Then $\B\setminus \D$ is a branched subsurface of $\B$. The boundary of $\B\setminus\D$ is precisely $\LL\cap\D$. In particular, since $\D$ is a disk, every loop $\ell\subset\partial(\B\setminus\D)$ can be realized as the image of the boundary of some map $f: D^2\to \D$.

Suppose $\ell$ is homotopically nontrivial in $\partial(\B\setminus\D)$. If $\ell$ is nullhomotopic in $\B\setminus \D$, then there is a second map $f':D^2\to \overline{\B\setminus\D}$ mapping the boundary to $\ell$. Together these two maps define $f\cup f': S^2\to \B$, which is evidently non-nullhomotopic.

If $\ell$ is essential in $\B\setminus\D$, it is homotopic to another loop $\ell'\subset\partial(\B\setminus\D)$. Let $g:S^1\times[0,1]\to\overline{\B\setminus\D}$ realize this homotopy. We also have $f':D^2\to \D$ mapping $\partial D^2$ to $\ell'\subset \D$. Then $f\cup g \cup f':S^2\to \B$ represents a nontrivial class in $\pi_2(\B)$.

Since $\B$ is the quotient of a 3-ball, $\B$ is contractible. However, we have shown $\pi_2(\B)\ne 1$, a contradiction.
\end{proof}

Our approach to proving Theorem~\ref{theorem: mushroom-saucer-toric} is to show, in each of the remaining three cases of folding assignments, the branched surface associated to any embedding with given folding data may be taken by a regular isotopy to one of the following three models.

\subsubsection*{Saucer--folding assignment $f_1: [\gamma_i , \gamma_m , \gamma_o] \longleftrightarrow [+,+,-]$}

The branched surface model associated with the saucer embedding has $\p \circ \e(\gamma_i \sqcup \gamma_m  \sqcup \gamma_o) \ (= \p (\partial \B \sqcup \LL))$ as three concentric circles in $\bR^2$ with $\p \circ \e(\gamma_i)$ the largest radius and $\p \circ \e (\gamma_m)$ the smallest radius.  

We take $\bK^+$ to be the images of $\{ A_{i,m} , \D_o \}$ and for convenience of notation relabel their images $\{A^+, \D^+ \} \subset \B$, respectively. Similarly, we take $\bK^-$ to be the images of $\{ A_{m,i}, \D_i\}$ and relabel their images $\{ A^- , \D^-\} \subset \B$, respectively. To satisfy the folding assignment, we take the $z$-support of $A^+$ and $\D^\pm$ in $\bR^2 \times \bR$ to be $\{0\}$, and the $z$-support of ${\rm int}\,{A}^-$ to be $(0, \frac{1}{2})$. 

\subsubsection*{Mushroom--folding assignment $f_2: [\gamma_i , \gamma_m , \gamma_o] \longleftrightarrow [+,-,+]$}

The branched surface model associated with the mushroom embedding will have $\p \circ \e(\gamma_i \sqcup \gamma_m  \sqcup \gamma_o) = \p(\partial \B \sqcup \LL) \subset \bR^2$ again be three concentric circles with $\p \circ \e(\gamma_i)$ the largest radius and $\p \circ \e (\gamma_m)$ the smallest radius.  

We take $\bK^+$ to be the images of $\{ A_{i,m} , \D_o \}$ and relabel their images $\{A^+, \D^+ \} \subset \B$, respectively. Similarly, we take $\bK^-$ to be the images of $\{ A_{m,i}, \D_i$\} and relabel their images $\{ A^- , \D^-\} \subset \B$, respectively. Again, in order to satisfy the folding assignment, we take the $z$-support of $A^+ $ and $\D^-$ to be $\{0\}$, and the $z$-support of ${\rm int}\,{A}^-$ to be $(0, \frac{1}{2})$.

\subsubsection*{Toric--folding assignment $f_3: [\gamma_i , \gamma_m , \gamma_o] \longleftrightarrow [-,+,-]$}

To construct the branched surface model associated with this toric embedding,  we start with $S^1 \vee S^1 \subset \bR^2 \times \bR$ such that  $\p(S^1 \vee S^1) \subset \bR^2$ looks like the standard two crossing projection of the Hopf link except that one crossing is now the common point of the wedge sum; this choice, as well as the choice of over-strand for the intact crossing, is a matter of symmetry.  Next, let $T \subset \bR^2 \times \bR$ be a torus-minus-open-disk which contains and deformation retracts onto our $S^1\vee S^1$.  Further, we require that $\p({\rm int}\,T) \subset \bR^2$ be an immersion.  Then, adapting Proposition \ref{proposition: GB for barK}, observe the turning number of $\partial T$ must be $-1$.  As such we let $\partial T = \e(\gamma_m)$.  Finally, to construct our $\B$, we attach to $T$ two disks, $\D^+$ and $\D^-$, one along each of the $S^1 {\rm 's}$ in our wedge sum, so that $\p(\D^\pm) \subset \bR^2$ is an embedding.  To specify, we identify $\partial \D^+$ (respectively, $\partial \D^-$) with the under-crossing (respectively, over-crossing) $S^1$-factor of the remaining crossing in the Hopf link projection. From this construction we have that $\partial \D^+$ and $\partial \D^-$ are the images of $\gamma_i$ and $\gamma_o$ in $\B$. The boundary and branch locus of $\B$ are shown in Fig.~\ref{Fig: plumbing}. 

\begin{figure}[ht]
\centering

\includegraphics[width=.35\textwidth]{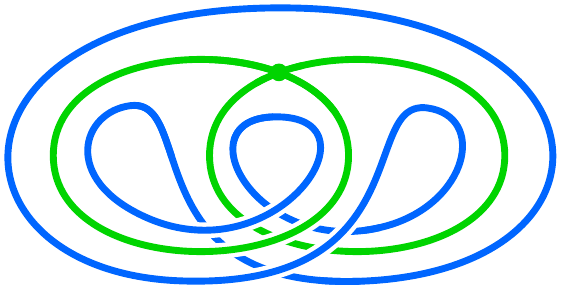}
\caption{The crease set of a toric sphere as identified in the branched surface. The {\green green} $\cC^-$ curves have been identified at a single point in $\LL$.}
\label{Fig: plumbing}
\end{figure}

\subsection*{Summary of the 3 models.}
We summarize the ``geography'' of each of our three models.  Referring back to Fig.~\ref{Fig: carries}, a choice of splitting into $\bK^+$ or $\bK^-$ produces a $2$-disk and an annulus, independent of model.  A salient feature of any such splitting will be the ``membership'' of the boundary components of the resulting disk/annulus pair---a boundary component will either come from a component of $\partial \B$ or a component of $\LL$.  Table \ref{Table: geography} categorizes the membership of each possible $K^\pm$ surface.  Superscripts of $\pm$ correspond to the superscripts associated with the splitting and superscripts, where used, correspond to the membership of the boundary component.

\begin{table}[ht]
\begin{minipage}{\textwidth}
    \centering
    
    \caption{Geography of the three models.}
    \label{Table: geography}
\begin{tabular}{@{}clll@{}}
\toprule
\ $\bK^\pm$\  & \ $f_1$ (saucer)\  & \ $ f_2$ (mushroom)\     & \ $f_3$ (toric)\  \\
\midrule
\ $\Delta^+$\  & \ $\partial \Delta^+ \subset \LL$\  & \ $\partial \Delta^+ \subset \partial \B$\  &  \ $\partial \Delta^+ \subset \LL$\ \\

\ $A^-$\  &  \ $\partial A^- = \partial_\B^- \cup \partial_\LL^-$\  &  \ $\partial A^- = \partial_\B^- \cup  \partial_\LL^-$ \  &  \ $\partial A^- = \partial_\B^- \cup \partial_\LL^-$\ \ \\

\ $A^+$\  &  \ $\partial A^+ \subset \partial \B$\  &  \ $\partial A^+ = \partial_\B^+ \cup  \partial_\LL^+$\   &  \ $\partial A^+ = \partial_\B^+ \cup \partial_\LL^+$\ \\

\ $\Delta^-$\  &  \ $\partial \Delta^- \subset \partial \B$\  &  \ $\partial \Delta^- \subset \partial \B$\  &  \ $\partial \Delta^- \subset \LL$\ \\
\bottomrule
\end{tabular}
\end{minipage}

\end{table}

\subsection{Proof of Theorem \ref{theorem: |cC| = 3}}
\label{subsection: proof of |cC| = 3 theorem}

We will establish that any regular embedding $\e(S^2)$ with $|\cC| = 3$ is regular isotopic to either the saucer, mushroom, or toric embedding (up to symmetry) by showing its associated branched surface, $\B \subset \bR^2 \times \bR$, is equivalent to one of our three model branched surfaces through a sequence of boundary-based regular isotopies (see \S \ref{subsubsection: boundary based isotopies}), $\partial$-close pinching (see \S \ref{subsubsection: pinching boundary close}), and pinching double-cusp disks (see \S \ref{subsubsection: pinching and double-cusp disks}). Applying appropriate symmetries, we make the following assumptions on a given embedding to match our models: in all cases, $\D_i$ will always correspond to $\D^-$; any $f_1$-folding embedding will have $\partial\D^+\subset\LL$; and, in any $f_3$-folding embedding, the signed intersection number of the image in $\B$ of the ordered pair $(\gamma_i .\gamma_o)$ is $-1$. 

To give justification for why, in an $f_3$-folding, $\LL$ must always contain at least one double point, we consider the surface with boundary $T = \B \setminus \{ {\rm int}\, \Delta^+ \cup {\rm int}\, \Delta^- \}$.  Since $\partial T = \partial \B$, the turning number of its boundary is $-1$.  Moreover, since every point in ${\rm int}\, T$ is regular with respect to $\p: T \rightarrow \bR^2$, the crease set of $T$ is empty.  Thus, $T$ is a torus minus a disk.  But, the two components of $\LL$ are naturally contained in ${\rm int}\, T$ as homotopically non-trivial curves, with each component bounding an embedded disk---the disks $\Delta^\pm$.  Such curves will have algebraic intersection $\pm1$.

To power our argument we first define a complexity measure.  The setup is as follows.  When $\e(S^2)$ is associated with the $f_1$ folding assignment, we let $\bX = \{A^- , \D^+\}$.  For folding assignment $f_2$, we let $\bX = \{ A^+ , A^- \}$. And, for $f_3$ we let $\bX = \{ A^+, A^-, \D^+ , \D^- \} $. In each case, $\bX$ is comprised of those components with some boundary component in $\LL$.

With this in place, we consider the graph $G = \bigsqcup_{X \in \bX} X \cap \LL$, where $X \in \bX$ is a single component. (We allow for the possibility that $G$ has $S^1$-``edges'', which occur as isolated components of $G$; these occur when some component of $\LL$ has no double points.) Let $D(\LL)$ denote the set of double points of $\LL$ and $V_G$ the vertex set of $\LL$. Each $v\in V_G$ corresponds to a point of $D(\LL)$, and by our choices of $\bX$, any double points of $\LL$ will be captured as vertices in $G$. Referring back to Fig.~\ref{Fig: neighborhoods of B} and \ref{Fig: carries}, observe that each component of $\LL\setminus D(\LL)$ appears as an edge at most four times in $G$ while each double point of $\LL$ appears as a vertex at most six times in $G$. Both bounds are realized when considering all four components of $\K^\pm$. 

For a component $X \in \bX$, we say a vertex $v \in V_G\cap \partial X$ is {\em $\partial$-close in $X$} if there exists a segment neighborhood, $\ell \subset \partial X \cap \LL$, such that $v \in \ell$ and $\ell$ is $\partial$-close as defined in \S \ref{subsubsection: pinching boundary close}.  Let $N_2$ be the number of vertices of $V_G$ that are $\partial$-close in some component of $\bX$ and $N_1$ the number of vertices of $V_G$ which are not.  Then our complexity measure for a branched surface, $\B \subset \bR^2 \times \bR$, is the lexicographically ordered $2$-tuple,  $\chi(\B) = (N_1, N_2)$.

Since $D(\LL) = \varnothing$ for our saucer and mushroom standard branched surface models, the minimal possible complexity for folding assignments $f_1$ and $f_2$ is $(0,0)$ with $G = \partial \bX \cap \LL$.

The complexity for the toric model is $(6,0)$, which we calculate for each component of $\bX$ as follows.  $\D^\pm \cap \LL$ is just the $\partial \D^\pm$ with each containing a single vertex of $G$ and, since $\D^\pm \cap \partial \B = \varnothing$, these single vertices will not be $\partial$-close in $D^\pm$.  $A^\pm \cap \LL$ will be subgraphs having three edges, $e_1^\pm , e_2^\pm , e_3^\pm$ and two vertices, $v_1^\pm , v_2^\pm$, with: $e_1^\pm \cup v_1^\pm = \partial A^\pm \cap \LL $, $e_2^\pm \cup v_2^\pm$ a loop bounding a disk in ${\rm int}\,{A}^\pm$, and $e_3^\pm$ having endpoints $v_1^\pm$ and $v_2^\pm$.  As such, $v_1^\pm$ and $v_2^\pm$ are not $\partial$-close in $A^\pm$.

This is indeed the minimum possible complexity for $f_3$, as there must be at least one double point in $\LL$, which is necessarily not $\partial$-close anywhere it appears in $G$.

We note that the subgraphs in $A^\pm$ in the toric model are informative.  Having an edge loop in $G$ implies that the number of components of $\LL$ is greater than $1$.  Similarly, having $e_3^\pm$ as ``cut'' edges implies that $|\LL| \geq 1$---$e_3^\pm$ share a vertex with $e_1^\pm$ and $e_2^\pm$ so in the quotient they cannot be identified with either edge. 

\subsection*{Case 1: $\chi(\B)$ is minimal.}
For the all three folding assignments, we can perform a boundary based isotopy that positions $\partial \B$ arbitrarily close to $\LL$.

For the saucer and mushroom cases, we pick a point $p$ in the interior of the disk component of $\B \setminus \LL$.  We can then choose a small enough neighborhood, $N(p)$, such that $C_r =\p(\partial N(p)) \subset \bR^2 \times \{0\}$ is circle of some fixed radius.

Now, pick a vector flow on the closure of disk component of $\B \setminus \LL$ that flows inward from $\LL$ to $C_r$.  We then use this flow to isotopy the neighborhood of $\LL$ that contains the annular components of $\B \setminus \LL$ to a neighborhood of $C_r$.  After this isotopy we will have the projection of the resultant branched surface, $\p(\B^\prime)$, be one that corresponds the that of our saucer or mushroom models.  That is, $\p(\partial \B^\prime \sqcup \LL)$ will correspond to three concentric circles.  We can then employ a linear isotopy that positions $\B^\prime$ so as to satisfy the $z$-support conditions of our saucer and mushroom models.  If $\D^-$ projects to $\p(\D^-)$ with $\partial \p(\D^-)$ being the outermost concentric circle, point-wise we could use $H_t(p) = (1-t) [p] + t[\p(p)]$, $p \in \D^-$, $0\leq t \leq 1$, and extend to the rest of $\B^\prime$.  Being linear, it is evident that $dH_t : T_p \B^\prime \to \bR^2 $, $0 \leq t \leq 1$ is an isomorphism on $p \in \B \setminus \partial \B$.

For the toric case we pick two arbitrarily small circles $S^+ \subset \D^+$ and $S^- \subset \D^-$ such that $S^+ \cap S^- = D(\LL)$ and $\p(S_1 \cup S_2)$ corresponds to a ``nice Hopf link'' projection with a wedge point and a single crossing point.

Next, we choose vectors flows on $\D^\pm$ that flow inward from $\partial \D^\pm \,(\subset \LL)$ to $S^\pm$.  We use these two flows to perform an isotopy moving $\LL$ to $S^+ \cup S^-$.  This isotopy extends to one moving the neighborhood of $\LL$ containing the annular component of $\B \setminus \LL$ to a small neighborhood of $S^+ \cup S^-$. This describes a regular isotopy of $\B$ to $\B^\prime$ such that the associated branching locus, $\LL^\prime$, has a nice Hopf link type projection.
Finally, let $T^\prime=\B^\prime \setminus ({\rm int}\,\D^{\prime+} \cup {\rm int}\,\D^{\prime -}$. (Here we are extending the ``prime'' notation to all of the component pieces of $\B^\prime$.) As in the saucer and mushroom cases, we can use a linear isotopy as a regular isotopy to take $T^\prime$ to the $T$ of the toric model, extending to take $\B^\prime$ to the toric model $\B$.

\subsection*{Case 2: $\chi(\B)$ is not minimal.}
Consider an arc $\ell \subset \cC^-$ whose image connects two double points of $\LL$, or $\ell$ an entire component of $\cC^-$.  By our definition of $\bX$ there will be two components, $X^+, X^- \in \bX$ such that $\ell \subset \partial X^\pm \cap \LL$.  (By way of example, in the mushroom case we necessarily have $\ell \subset \partial A^+ \cap \LL$ and $\ell \subset \partial A^- \cap \LL$.)  Let $v^\pm_1, \ldots , v^\pm_n \in \partial X^\pm \cap V_G$ be those vertices which lie along $\ell$ in $X^\pm$.  If one or both of $X^\pm$ is annular, then there is the possibility that $\ell$ is $\partial$-close in that component---though by the assumption that $\ell$ has double points, it cannot be $\partial$-close in both. We break into two cases: when $\ell$ is $\partial$-close in one of $X^\pm$, and when $\ell$ is not $\partial$-close in either. These correspond exactly to when we can use each of the pinching moves discussed in \S \ref{subsection: isotopies of B}.

\subsubsection*{Pinching $\partial$-close as in \S \ref{subsubsection: pinching boundary close}.} 
Suppose that $\ell$ is $\partial$-close in, say, $X^+$, but not in $X^-$---there exists a connected subgraph, $J \subset G$, such that $J \cap \partial X^- = \{ v^-_1, \ldots , v^-_n \}$ and either $X^-=\D^-$ or $J\cap {\rm int}\, X^-$ is nonempty. We further suppose $J\cap {\rm int}\, X^-$ is connected, possibly by passing to a subarc of $\ell$, and that $\ell$ is maximal in the sense that either $G\cap X^-\setminus\{v^-_1,v^-_n\}$ is disconnected or $\ell\cong S^1$. Then a neighborhood $R(J) \subset X^-$ defines either an $\LL$-disk or $\LL$-annulus. We then perform the pinching procedure of \S \ref{subsubsection: pinching boundary close}, replacing $\B$ with the resulting $\B^\prime$. This eliminates the $2n$ points $v_1^\pm,\ldots,v_n^\pm$ from $D(\LL)$, reducing $N_1$ and $N_2$ each by $n$. Other points of $D(\LL)$ may become $\partial$-close (decreasing $N_1$ while increasing $N_2$). However, the conditions on $\ell$ guarantee no new double points are introduced, and any $\partial$-close double point is either eliminated or stays $\partial$-close. As a result, $\chi(\B^\prime) \leq \chi(\B)$. This same operation may be performed with $X^+$ and $X^-$ interchanged, so that $\ell$ is $\partial$-close in $X^-$ and not in $X^+$.

\subsubsection*{Pinching double-cusp disks as in \S \ref{subsubsection: pinching and double-cusp disks}.} 
Now we assume that $\ell$ is not $\partial$-close in either $X^+$ or $X^-$.  Similar to before, we assume $\ell$ is maximal in the sense that either $G\cap X^\pm\setminus\{v^\pm_1,v^\pm_n\}$ are both disconnected or $\ell\cong S^1$, and that no subarc of $\ell$ is maximal.

The reader should observe that if there is no such $\ell$, either every double point of $\LL$ is $\partial$-close and we can always apply the pinching procedure of \S\ref{subsubsection: pinching boundary close} to eliminate points in $D(\LL)$, or there is a single point $p\in D(\LL)$ which is not $\partial$-close in either $X^\pm$---the toric case. In either case, if $\LL$ has at least two double points, some pair must be connected by a $\partial$-close arc.

Suppose $\ell \subset \LL$ is such a maximal arc. Let $J^\pm \subset G$ be connected subgraphs such that ${J}^\pm \setminus \ell \subset {\rm int}\,{X}^\pm$ and $J^\pm \cap \partial X^\pm =\ell$. At least one of $J^\pm$ must contain an edge outside of $\ell$.

Next, using the product structure of our space, $\bR^2 \times \bR$,  we project $J^+$ (respectively, $J^-$) onto $X^-$ (respectively, $X^+$).  Let $\p_{X^\pm} (J^\mp) \cup J^\pm \subset X^\pm$ denote the resulting graphs, taking the (transverse) intersections of $\p_{X^\pm} (J^\mp) \cup J^\pm$ as additional vertices.
Observe that $$\p(\p_{X^+} (J^-) \cup J^+) = \p(\p_{X^-} (J^+) \cup J^-).$$

Now, let $E^+ \subset \p_{X^+} (J^-) \cup J^+ \subset X^+$ be a simple edge-path that starts and ends in $\ell \subset \partial X^\pm$ with ${\rm int}\,{E}^+ \subset {\rm int}\,{X}^+$. From the common projection, we can lift a dual edge-path $E^- \subset \p_{X^-} (J^+) \cup J^- \subset X^-$ such that $\p(E^+) = \p(E^-)$.  Thus, $E^+ \cup E^-$ is the boundary of a double-cusp disk, $\td$.  We then take $\delta^\pm$ to be the half-disks that $E^\pm$ split off in $X^\pm$.  Thus, we have the setup for Lemma \ref{lemma: emptying B-Sigma} and can now assume ${\rm int}\,{\td} \cap \B = \varnothing$. With this assumption in place we can performing the pinching procedure described in \S \ref{subsubsection: pinching and double-cusp disks}.

In the resulting branched surface, $\B^\prime$, we have reduced $N_1$, so $\chi(\B^\prime) < \chi(\B)$.  We caution the reader that a decrease in $N_1$ may result in an increase in $N_2$.  However, the combined procedures of \S\ref{subsubsection: pinching boundary close} and \S\ref{subsubsection: pinching and double-cusp disks} will always terminate, at which point either $|D(\LL)| = 0$---the saucer and mushroom---or $D(\LL) = 1$---the toric case.  We are then in the case where $\chi(\B)$ is minimal.
\endproof

\begin{example}[Mushroom with five double points] It is helpful to see how the argument of \S \ref{subsection: proof of |cC| = 3 theorem} can be carried out in actual examples.  In Fig.~\ref{Fig: double-points example1}, we offer an example of a branched surface that is equivalent to that of a mushroom $S^2$ embedding.  There are five double points in the branching locus. We draw the reader's attention to the green curve, the unique component of $\LL$.  We have marked a point with ``$\times$''.  Starting at this point one can produce a ``Gaussian notation'' sequence by labeling the double points as they are encountered when traversing the $\LL$.  Specifically, we have the numeric sequence $$\times, 1,2,3,3,4,2,5,5,4,1,\times.$$
We also draw the reader's attention to a single crossing in the projection of $\LL$ which does not make a contribution to this numeric sequence.

To finish out the narrative for Fig.~\ref{Fig: double-points example1}, $\partial \B$ consists of the two blue curves.  The light-blue ``train-tracks'' illustrate how $\B$ branches at $\LL$.  And the regions with the light-blue swirls correspond to portions of $\D^+$. Here $\D^-$ is not visible, as it is the underside of this depiction of $\B$. 

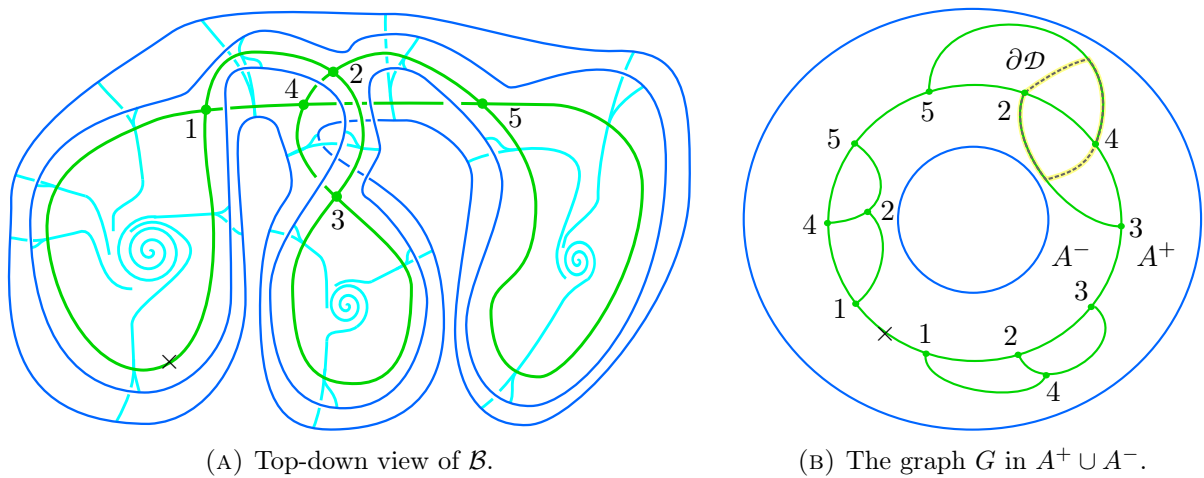
\begin{figure}[ht]
\centering

\begin{subfigure}[b]{.55\textwidth}

\labellist
\small
\pinlabel $1$ [c] at 150 247
\pinlabel $2$ [c] at 285 292
\pinlabel $3$ [c] at 270 173
\pinlabel $4$ [c] at 232 280
\pinlabel $5$ [c] at 415 256

\pinlabel $\times$ [c] at 132 56
\endlabellist

\centering
\includegraphics[width=\textwidth]{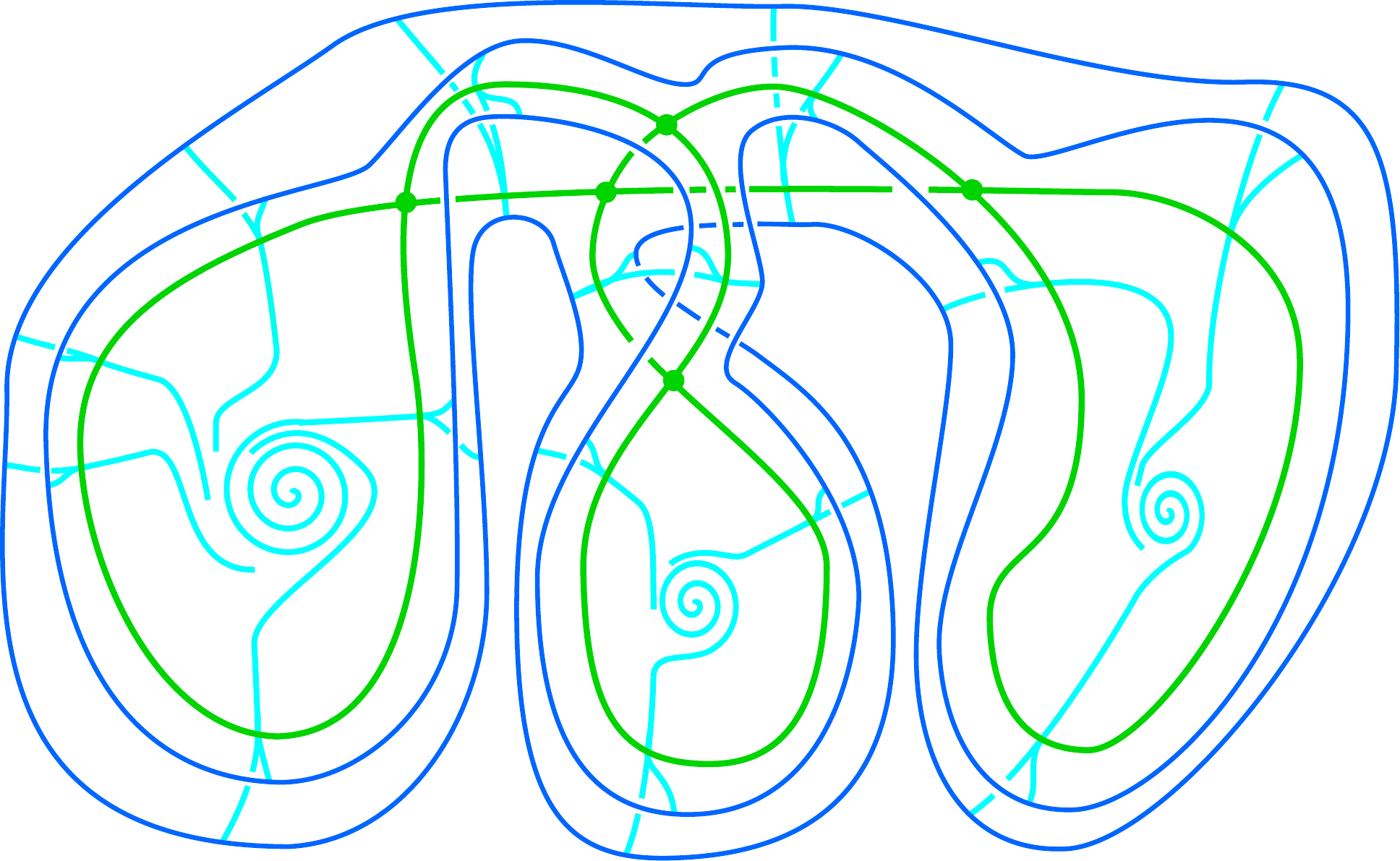}
\caption{Top-down view of $\B$.}
\label{Fig: double-points example1}
\end{subfigure}
\hspace{1em}
\begin{subfigure}[b]{.37\textwidth}

\labellist
\small

\pinlabel $1$ [c] at 150 80
\pinlabel $2$ [c] at 220 78
\pinlabel $4$ [c] at 255 32
\pinlabel $3$ [c] at 278 115
\pinlabel $3$ [c] at 323 167
\pinlabel $4$ [c] at 303 240
\pinlabel $2$ [c] at 215 263
\pinlabel $5$ [c] at 152 262

\pinlabel $4$ [c] at 55 169
\pinlabel $5$ [c] at 75 240
\pinlabel $1$ [c] at 80 100
\pinlabel $2$ [c] at 119 181

\pinlabel $\partial \td$ [c] at 232 305

\pinlabel $A^+$ [c] at 340 145
\pinlabel $A^-$ [c] at 270 145

\pinlabel $\times$ [c] at 117 80
\endlabellist

\centering
\includegraphics[width=\textwidth]{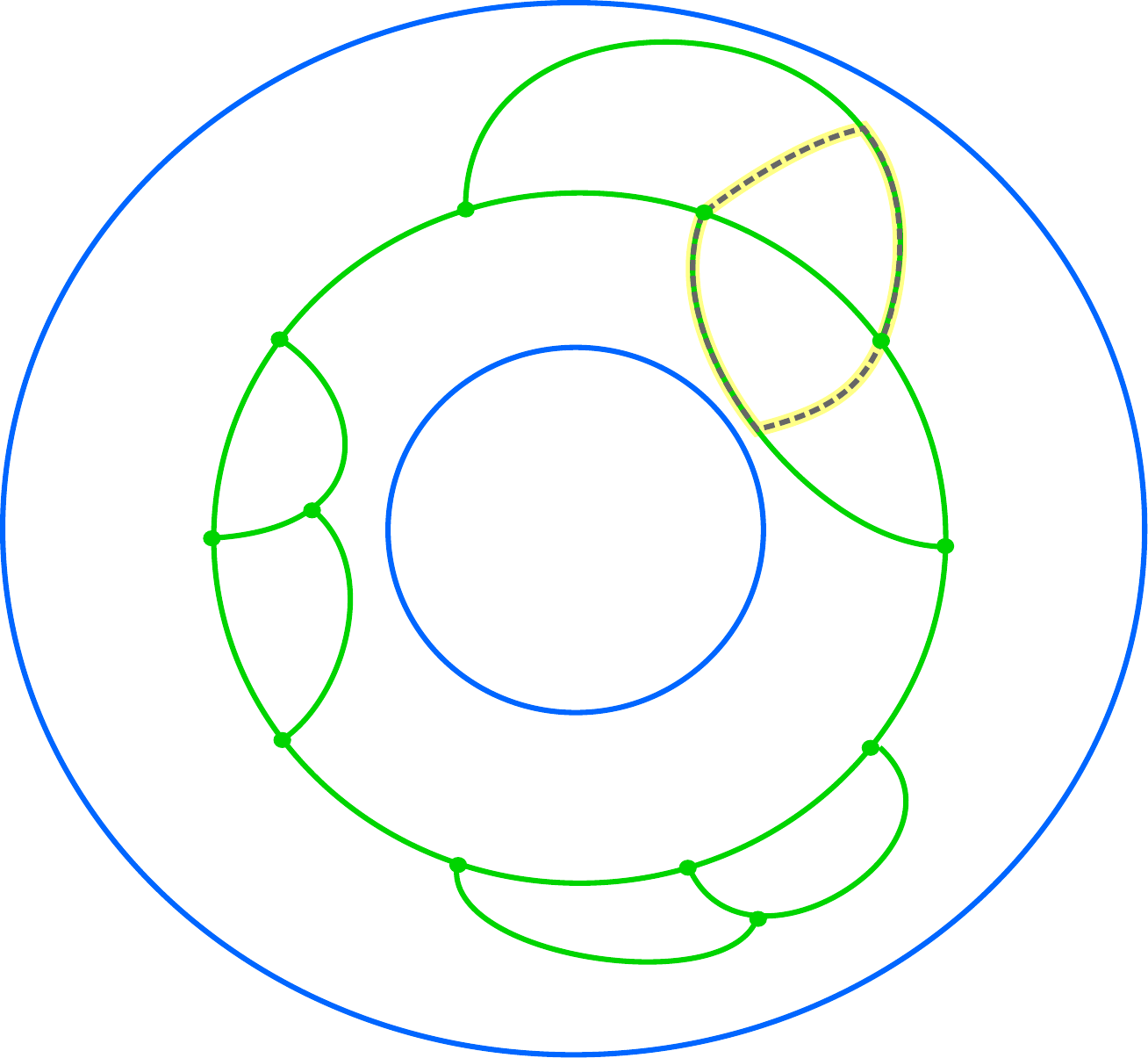}
\caption{The graph $G$ in $A^+\cup A^-$.}
\label{Fig: double-points example2}
\end{subfigure}

\caption{A mushroom-type branched surface $\B$ and its associated graph $G$. The green curve corresponds to $\LL$ and the blue curves correspond to the two components of $\partial \B$.  The three ``swirling'' regions of $\B$ lie in $\Delta^+$. The dotted loop on the right is the boundary of a double-cusp disk.}
\label{Fig: double-points example}
\end{figure}

Fig.~\ref{Fig: double-points example2} depicts $A^+$ and $A^-$, which are $X^+$ and $X^-$ in the mushroom case.  For convenience,  we adjoin them together along their common branching boundary component---the core green circle. Note that the vertices along this curve appear twice in $G$---once in the $A^+$ component and once in the $A^-$ component. In ${\rm int}\,{A}^+$ and ${\rm int}\,{A}^-$, we have the subgraphs of $G$ arising from $A^+ \cap \LL$ and $A^- \cap \LL$, respectively.

Initially, $\chi(\B) = (14, 8)$.  To see this  with respect to $A^+$, when we traverse counterclockwise the core green circle---$\partial A^+ \cap \LL$---starting at the $\times$-point, the first occurrences of $1$, $2$, $3$, $4$, and $5$  are not $\partial$-close.  Additionally, the second occurrence of $2$ and the interior vertex $4$ are not $\partial$-close.  However, the second occurrences of $3$, $5$, $4$, and $1$ are all $\partial$-close.  Thus, $A^+ \cap G$ contributes a count of $7$ to $N_1$ and $4$ to $N_2$.  

Replicating this count with respect to $A^-$, the second occurrences of $3$, $2$, $5$, $4$, and $1$ along with the first occurrence of $4$ and the interior $2$-vertex are not $\partial$-close.  But, the first occurrences of $1$, $2$, $3$, and $5$ are all $\partial$-close.  Thus, $A^- \cap G$ also contributes $7$ to $N_1$ and $4$ to $N_2$.  
 
We have a number of ways to use subgraphs to eliminate double points of $\LL$, thereby reducing $\chi(\B)$.  We refer the reader to Fig.~\ref{Fig: double-points example2}.  To perform a $\partial$-close pinching, we consider the subgraph $J^+$ that consists the vertex labeled $4$ in ${\rm int}\, A^+$, its three adjacent vertices on the green core circle, that in our numeric sequence correspond to $\times, 1,2,3$, and their adjoining edges. With respect to $A^-$, the arc through the vertices in the same partial numeric sequence, $\times,1,2,3$, is $\partial$-close.  Then we take a neighborhood $N(J^+) \subset A^+$ as the needed $\LL$-disk to perform a $\partial$-close pinching. There is a similar possible $\partial$-close pinching coming from the subgraph $J^- \subset A^-$ containing the $2$-labeled vertex in ${\rm int}\, A^-$, its three adjacent vertices in $\partial A^-$, and their adjoining edges.

Performing the $\partial$-close pinching associated with the above $J^+$ eliminates the double points $1$, $2$, and $3$, leaving two double points, $4$ and $5$.  The resulting numeric sequence for the new branched surface, $\B^\prime$, is $$ \times, 4,5,5,4,\times.$$  Specifically, the first occurrences of $4$ and $5$ are $\partial$-close with respect to $A^-$, but the second occurrences of $5$ and $4$ are not.  This pattern flips with respect to $A^+$---the first occurrences of $4$ and $5$ are not $\partial$-close with respect to $A^+$, but the second occurrences of $5$ and $4$ will be.  Thus, the new complexity measure is $\chi(\B^\prime) = (4,4)$.

Alternatively, the arc $\ell\subset \LL$ containing the subsequence $3,4,2,5$ is not $\partial$-close in either $A^+$ or $A^-$, as it cannot be entirely ``seen'' from $A^+\cap\partial\B$ or $A^-\cap\partial\B$. So one can employ pinching. To do so, we first identify the boundary of a double-cusp disk.

To obtain $J^-$, we add to $\ell\subset A^-$ the edge in ${\rm int}\,A^-$ with endpoints $4$ and $5$, and to obtain $J^+$, we take $\ell\subset A^+$ and the edge in ${\rm int}\,A^+$ with endpoints $3$ and $2$.  We then project $J^-$ into $A^+$ and $J^+$ into $A^-$.  The dotted loop $\partial \td$, shown in Fig.~\ref{Fig: double-points example2}, is a boundary of the required double cusp disk.  The ``rectangular'' region $\partial \td$ bounds corresponds to $d^+ \cup d^-$.  Performing the associated pinching will increase $|D(\LL)|$ from $5$ to $6$.

To see this move in Fig.~\ref{Fig: double-points example1}, the reader should imagine sliding the $2$-labeled double point along $\LL$ towards-and-past the double point labeled $4$.  This will create two new double points, one in the $\overline{14}$ segment of $\LL$---label it $6$---and another in the $\overline{43}$ segment---label it $7$.  In Fig.~\ref{Fig: double-points example2}, this slide corresponds to shrinking all three edges having point labels $(2,4)$.  In the resulting branched surface, we drop the use of $2$ as a double point label, yielding the numeric sequence, $$\times, 1,6,7,3,3,7,4,5,5,4,6,1,\times.$$

To focus on the complexity change, $N_1$ drops from $14$ to $12$ while $N_2$ increases from $8$ to $12$.  Thus, the complexity will be reduced since lexicographically, $(12,12) < (14,8)$.  Moreover, every subgraph of ${\rm int}\, A^\pm \cap G$ will be associated with a $\partial$-close pinching. Specifically, the reader can check that in ${\rm int}\, A^+ \cap \LL$, there will be three edges with endpoint pairs $(1,6)$, $(7,3)$, and $(4,5)$; and in ${\rm int}\, A^- \cap \LL$, there will be three edges with endpoint pairs $(3,7)$, $(5,4)$ and $(6,1)$---all outermost edges in their respective annuli.
\end{example}

\section{Possible isotopy classes of $\e(S^2)$ for $|\cC| = 5$.}
\label{section: 5 creases}

We expect such a classification to become much more difficult as the number of crease curves increases. As an illustration of the variety of behavior that might arise, we examine one of the configurations consisting of five curves. In this configuration, shown in Fig.~\ref{Fig: 5creaseconfig}, four of the curves bound disks in $S^2$, while the fifth curve has turning number $-3$.

\begin{figure}[ht]

\labellist
\small
\pinlabel $\gamma_1$ [c] at 60 110
\pinlabel $\gamma_2$ [c] at 60 20
\pinlabel $\gamma_3$ [c] at 180 110
\pinlabel $\gamma_4$ [c] at 180 20
\pinlabel $\gamma_5$ [c] at 240 80
\endlabellist   

\centering
\includegraphics[width=.3\linewidth]{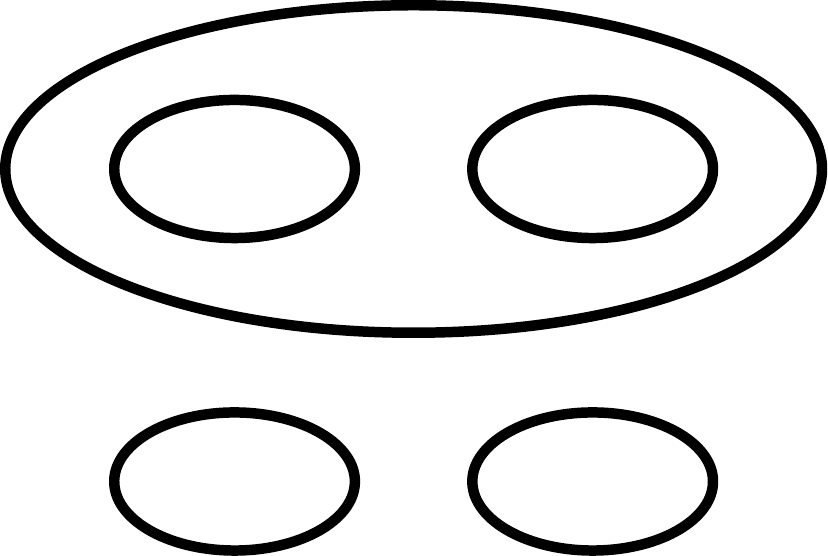}
\caption{The five-crease configuration under discussion.}
\label{Fig: 5creaseconfig}
\end{figure}
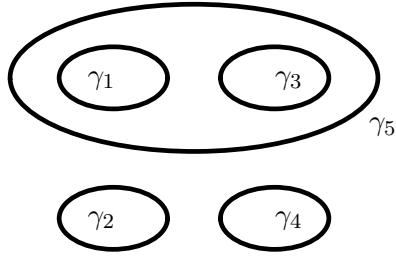

We begin by considering which crease sign assignments are realizable. Drawing a parallel to the three-curve case, we might hope each choice of sign assignment for the four disk-bounding curves determines that of the fifth, $-3$-turning number curve, and in turn a unique regular isotopy class of embedding.

To explore this, we introduce a method for constructing embeddings.

\begin{definition}
Let $\e,\e':S^2\to\bR^2\times\bR$ be normal embeddings with smooth crease sets $\cC$ and $\cC'$, and fix curves $\gamma\in \cC,\gamma'\in\cC$ such that at least one of $\gamma,\gamma'$ has positive crease sign. Let $\D,\D'\subset S^2$ be disks such that $\D\cap\cC\subset\gamma$ and $\D'\cap\cC'\subset\gamma'$. Then the {\it connect sum $\e\#_{\gamma\#\gamma'}\e'$ of $\e$ and $\e'$ along $\gamma$ and $\gamma'$} is the embedding formed by taking a smooth connect sum of $\e(S^2)$ and $\e'(S^2)$ along $\e(\D)$ and $\e'(\D')$.
\end{definition}

It is straightforward to see that, up to regular isotopy, the resulting embedding is well-defined and determined by the choice of curves $\gamma$ and $\gamma'$. The crease sign condition on $\gamma$ and $\gamma'$ is necessary for the result to be an embedding: if, say, $\gamma$ has negative folding, in order to create a smooth connect sum, the $\e'$ factor must lie inside the ball bound by $\e$. The crease signs of the resulting embedding depend on the crease signs of $\gamma$ and $\gamma$ as well, in that if, say, $\gamma$ has negative folding, $\gamma\#\gamma'$ will have negative sign, the other curves of $\cC$ will remain unchanged, but the signs of all other curves of $\cC'$ will flip.

Returning to the five-curve configuration, this connect sum operation yields examples for every choice of crease sign for the four disk-bounding curves, shown in part in Table~\ref{tab:5-curves}. The unlisted assignments may be obtained through a symmetry of one of the connect sum factors or the resulting embedding.

\begin{table}[ht]
\begin{minipage}{\textwidth}
\centering
\caption{Realizations of embeddings based on the crease signs of disk-bounding curves, enumerated in the first four columns. The first factor in each sum corresponds to $\gamma_1$ and $\gamma_2$; the second to $\gamma_3$ and $\gamma_4$. The last column indicates the crease sign of $\gamma_5$ in the given realization.}
\label{tab:5-curves}
\begin{tabular}{@{}cccccc@{}}
\toprule
\ $\gamma_1$ & $\gamma_2$ & $\gamma_3$ & $\gamma_4$ & Realization      & $\gamma_5${\ } \\ \midrule
\ $+$        & $+$        & $+$        & $+$        & $\e_{M}\#\e_{E}$ & $-${\ }       \\
\ $-$        & $-$        & $-$        & $-$        & $\e_{E}\#\e_{E}$ & $+${\ }       \\
\ $+$        & $+$        & $\pm$      & $\mp$      & $\e_{M}\#\e_{S}$ & $-${\ }       \\
\ $-$        & $-$        & $\pm$      & $\mp$      & $\e_{S}\#\e_{E}$ & $+${\ }       \\
\ $\pm$      & $\mp$      & $\pm$      & $\mp$      & $\e_{S}\#\e_{S}$ & $-${\ }       \\ \bottomrule
\end{tabular}
\end{minipage}
\end{table}

Note that this construction leaves a gap: we only give an example for one of the two possibilities for the crease sign of $\gamma$ in each case. However, $\gamma_5$'s sign is determined by those of $\gamma_1,\ldots,\gamma_4$ in rows 1-2 (since we must have at least one $+$ and one $-$ crease) and in row 4 (by the same proof as of Lemma~\ref{Lem: eliminating f_4}). We do not know whether examples with the opposite sign for $\gamma_5$ exist for rows 3 and 5.

Moreover, it seems plausible that non-regularly isotopic embeddings may have the same crease sign data, even accounting for further symmetry. Let $\bar{\e}_E$ be the mirror image of $\e_E$ through the $\bR^2$-plane. Then $\e_E\#\e_E$ and $\e_E\#\bar\e_E$ have the same crease sign data, yet do not seem to be regularly isotopic.


\section{Open questions}
\label{Section: problems}

Although there are large number of directions an investigation into the crease set could take---for example, surfaces with boundary in $\bR^2 \times \bR$---we supply only 
a few that immediately come to mind for closed surfaces.  The motivation of these problems is to advance the further development of the machinery of \S \ref{section: constructing embeddings} and \S\ref{section: classification set-up}.

\begin{problem}[Analyze the $S^2$ crease set with corners]
{\rm As previously mentioned, from \S\ref{section: constructing embeddings} onward our analysis deals with understanding the embeddings of $S^2$ into $\bR^2 \times \bR$ in the situation where the crease set is without corners.  An obvious direction of investigation is to carry out a similar analysis in the situation where corners occurred. There should be some interaction between this line of investigation and the one described in the previous sections.  In particular, referring back to {\mbox Fig.~\ref{Fig:Cheshire}} and the accompanying discussion,  one sees that a pair of ``canceling corners'' can be introduced by the creation of a dimple.  Fig.~\ref{Fig:cornercancelling} illustrates of another method of adding/canceling corners to the crease set.  The proposed problem is to not only give a constructive procedure for the geometric realization of a crease set with corners that satisfies a Gauss-Bonnet weighting criteria, but give a finite list of operations for moving between embeddings.}
\end{problem}

\begin{figure}[ht]
\centering
\includegraphics[width=.5\linewidth]{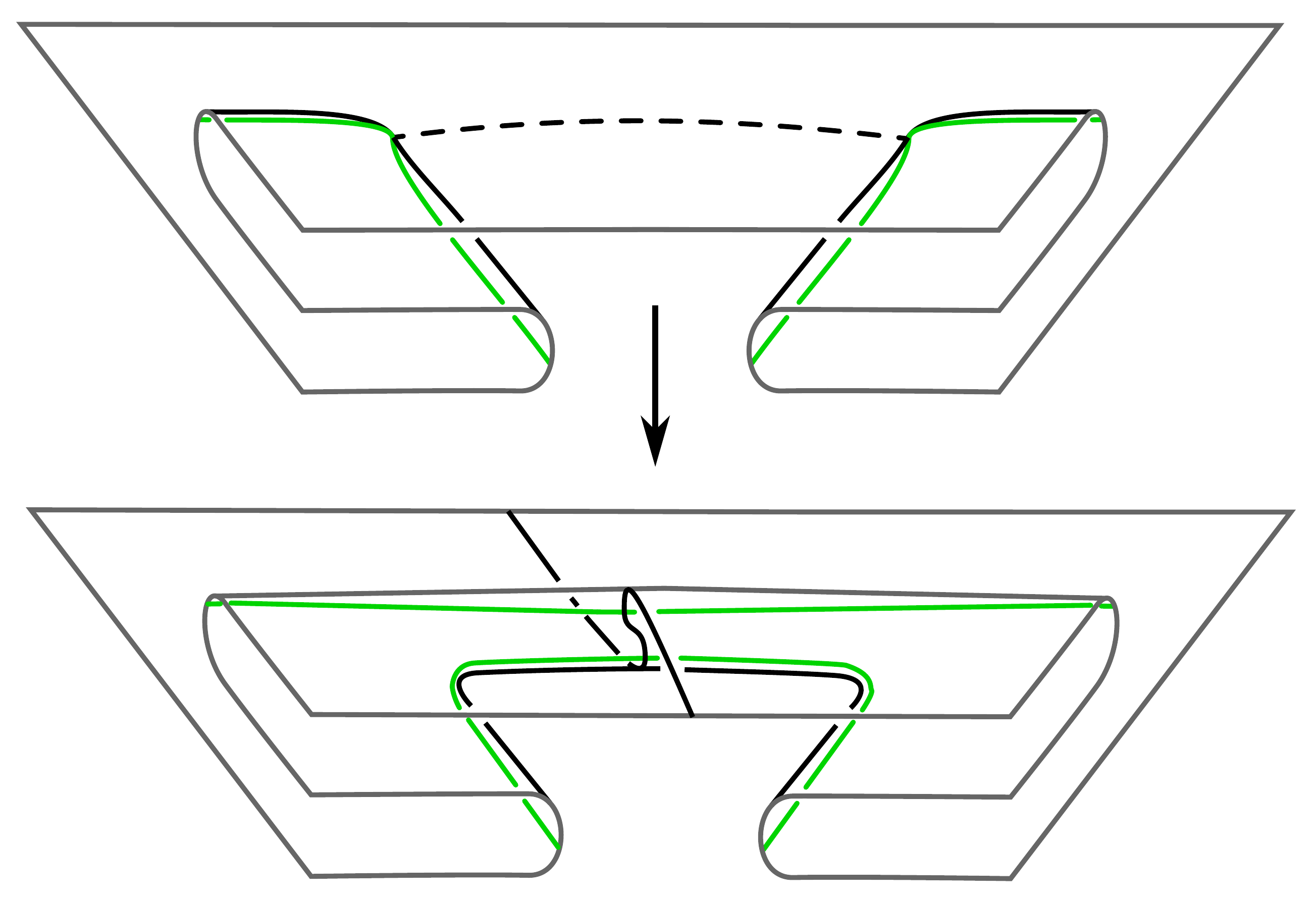}
\caption{The {\green green curves} represent $\cC$. The dashed arc between the corners in the upper diagram gives a strip to perform a corner-canceling isotopy. The arc in the lower diagram, at first flat, crosses the new crease set twice.}
\label{Fig:cornercancelling}
\end{figure}

\begin{problem}[Estimate the growth rate of embedding equivalence classes for $S^2$]
\label{problem: 2}
{\rm The construction procedure for the geometric realization of a crease configuration given in \S\ref{section: constructing embeddings} clearly involves making a lot of choices. In \S\ref{section: when |cC|=3}, we see there are exactly three non-isotopic embeddings with the same crease configuration of three curves; in \S\ref{section: 5 creases} we produce five with a common crease configuration. It is reasonable to suspect this is the case in general, and that a given crease configuration will have numerous geometric realizations.  The proposed problem is to calculate the grow rate of set of equivalency classes of embeddings for $S^2$ or any other closed orientable surface into $\bR^2 \times \bR$, as measured in comparison to the number of crease curves and corners. }
\end{problem}

\begin{problem}[Analyze the crease set for embeddings of $S^1 \times S^1\into\bR^2 \times \bR$]
{\rm We restrict our analysis here to embeddings of $S^2\into \bR^2\times\bR$. It is natural to consider which phenomena persist or collapse in positive genus. Moving just one step along this path, an investigation of the crease set for embeddings $\e : S^1 \times S^1 \into \bR^2 \times \bR$ might proceed as follows.  Given a configuration pair, $(S^1 \times S^1, \cC)$, whose closed components of $S^1 \times S^1 \setminus \cC$ satisfy Equ.~\ref{equation: GB for barK}, any component of $\cC$ is a s.c.c.\ of one of three types:  (1) homologically trivial in $S^1 \times S^1$; (2) a curve that is the boundary of an embedded disk contained in the solid torus that $\e(S^1 \times S^1)$ bounds; or, (3) a longitudinal curve on $\e(S^1 \times S^1)$.  If $\e(S^1 \times S^1)$ is unknotted then this last type of curve {\em may} also bound an embedded disk.  If type-(2) occurs then type-(3) cannot occur and vice-versa.  With this in mind, the proposed problem is to give a construction of the embedding for geometrically realizing the crease set $\cC \subset S^1 \times S^1$.}
\end{problem}

\begin{problem}[Analyze the crease set for embeddings of $S^2 \into \bR^2 \times \bR^2$]
{\rm It is very tempting to push an analysis of the crease set into dimension $4$.  Since $\bR^4$ has a natural product structure of $\bR^2 \times \bR^2$, one would in fact have two crease configurations to work with---one for each factor projection map.  The obvious first interesting examples one might consider are that of knotted $2$-spheres.}

\end{problem}




\end{document}